\documentclass[12pt, reqno]{amsart}

\usepackage[
backend=biber,
style=alphabetic,
doi=false,isbn=false,url=false]{biblatex}
\usepackage[english]{babel}
\usepackage{graphicx}
\usepackage{subcaption}
\usepackage{amssymb}
\usepackage{amsthm}
\usepackage{amsmath}
\usepackage{listings}
\usepackage{lineno}
\usepackage[margin=3cm]{geometry}
\usepackage[all,cmtip, color,matrix,arrow]{xy}
\usepackage{marvosym}
\usepackage{tipa}
\usepackage{comment}
\usepackage[shortlabels]{enumitem}
\usepackage{nicematrix}

\usepackage{amsmath}
\usepackage[utf8]{inputenc}
\usepackage{bbold}
\usepackage{pifont}
\usepackage{bm}
\usepackage{wrapfig}
\usepackage{bbold}
\usepackage{subcaption}
\usepackage{float}
\usepackage{mathtools}
\usepackage{aliascnt}
\newaliascnt{eqfloat}{equation}
\newfloat{eqfloat}{h}{eqflts}
\floatname{eqfloat}{Equation}
\usepackage{dirtytalk}
\usepackage{wasysym}
\usepackage{amsaddr}
\usepackage{tipa}
\usepackage{tikz-cd}
\usepackage{tikz-3dplot}

\definecolor{turkisD}{HTML}{579C87}
\definecolor{grunH}{HTML}{A1C181}
\definecolor{gelb}{HTML}{FCCA46}
\definecolor{orange}{HTML}{FE7F2D}
\definecolor{blauD}{HTML}{233D4D}
\usepackage{hyperref}
\hypersetup{
 colorlinks=true,
 linkcolor=blauD, citecolor=blauD,
}
\usepackage[capitalize]{cleveref}  
\crefname{thm}{Theorem}{Theorems}

\newcommand*{\ORGeqfloat}{}
\let\ORGeqfloat\eqfloat
\def\eqfloat{%
  \let\ORIGINALcaption\caption
  \def\caption{%
    \addtocounter{equation}{-1}%
    \ORIGINALcaption
  }%
  \ORGeqfloat
}

\newcommand{\Def}[1]{{\bf #1}}

\newtheorem{thm}{Theorem}[section]
\newcommand{\newsharedtheorem}[2]{%
  \newaliascnt{#1}{thm}\newtheorem{#1}[#1]{#2}\aliascntresetthe{#1}}
\newsharedtheorem{prop}{Proposition}
\newsharedtheorem{lm}{Lemma}
\newsharedtheorem{cor}{Corollary}
\newsharedtheorem{claim}{Claim}
\newsharedtheorem{conj}{Conjecture}
\newtheorem*{conj:kernel}{Conjecture \ref{conj:kernelspan}} 

\theoremstyle{definition}
\newsharedtheorem{defin}{Definition}
\newsharedtheorem{smpl}{Example}
\newsharedtheorem{obs}{Observation}
\newsharedtheorem{quest}{Question}
\newsharedtheorem{prob}{Problem}
\newsharedtheorem{rem}{Remark}
\newsharedtheorem{algo}{Algorithm}
\newsharedtheorem{constr}{Construction}

\crefname{thm}{Theorem}{Theorems}
\crefname{prop}{Proposition}{Propositions}
\crefname{lm}{Lemma}{Lemmas}
\crefname{cor}{Corollary}{Corollaries}
\crefname{claim}{Claim}{Claims}
\crefname{conj}{Conjecture}{Conjectures}
\crefname{defin}{Definition}{Definitions}
\crefname{smpl}{Example}{Examples}
\crefname{obs}{Observation}{Observations}
\crefname{quest}{Question}{Questions}
\crefname{prob}{Problem}{Problems}
\crefname{rem}{Remark}{Remarks}
\crefname{algo}{Algorithm}{Algorithms}
\crefname{constr}{Construction}{Constructions}
\DeclareMathOperator{\im}{im}

\DeclareMathOperator{\rk}{\mathrm{rk}}

\DeclareMathOperator{\spn}{\mathrm{span}}
\DeclareMathOperator{\cone}{\mathrm{cone}}

\DeclareMathOperator{\QSym}{\mathtt{QSym}}
\DeclareMathOperator{\WQSym}{\mathtt{WQSym}}
\DeclareMathOperator{\ncPsi}{\text{\Neptune}}
\DeclareMathOperator{\Mat}{\mathtt{Mat}} 
\newcommand{\unifMat}[2]{\mathtt{U}_{#1}^{#2}}
\DeclareMathOperator{\MatLinSpace}{\mathtt{Mat}}
\DeclareMathOperator{\PolSpace}{\mathtt{PS}}
\DeclareMathOperator{\cPsi}{\boldsymbol{\Psi}}

\DeclareMathOperator{\conv}{conv}
\newcommand\rspan{\mathrm{span}_\R}
\DeclareMathOperator{\cl}{cl}

\newcommand{\opi}{\vec{\boldsymbol{\pi}}}
\newcommand{\otau}{\vec{\boldsymbol{\tau}}}
\newcommand{\osigma}{\vec{\boldsymbol{\sigma}}}

\newcommand{\redcross}{\textcolor{red}{\ding{54}}}

\newcommand{\SM}{\mathrm{SM}} 
\newcommand{\leqG}{\preceq_\text{Gale}} 
\newcommand{\leqP}{\preceq_\text{path}} 
\newcommand{\bases}{\mathcal{B}} 
\newcommand{\neighborhood}{\mathit{N}}
\newcommand{\matroid}{\mathrm{M}} 
\newcommand{\natroid}{\mathrm{N}} 
\newcommand{\uniformatroid}[2]{\mathrm{U}_{#1}^{#2}}
\newcommand{\groundset}{U} 
\newcommand{\othergroundset}{V} 
\newcommand{\coloop}[1]{\widetilde{#1}} 
\newcommand{\nd}{{d}} 
\DeclareMathOperator{\rank}{rk}
\newcommand{\wqscmatrix}{\mathbb{M}} 
\newcommand{\wqscmatrixconj}{\mathbb{N}} 

\newcommand{\bw}{\mathbf{w}} 
\newcommand{\bx}{\mathbf{x}} 
\newcommand{\al}{\text{\textbf{\textscripta}}} 

\newcommand{\vv}{\mathsf{v}}

\newcommand{\vx}{\mathsf{x}}
\newcommand{\vy}{\mathsf{y}}

\newcommand{\ve}{\mathsf{e}}
\newcommand{\be}{\mathsf{e}}

\newcommand{\onebb}{\mathbb{1}} 

\newcommand{\Pol}{\mathsf{P}}
\newcommand{\setofPol}{\mathcal{P}} 
\newcommand{\subdiv}{\mathcal{S}} 
\newcommand{\Qol}{\mathsf{Q}}
\newcommand{\hypersimplex}[2]{\Delta_{#1}^{#2}}
\newcommand{\normalcone}{\mathsf{N}}
\newcommand{\face}{\mathsf{F}} 
\newcommand{\Nfan}{{\mathcal{N}}} 
\newcommand{\Hyp}{\mathsf{H}} 
\newcommand{\BraidA}{\mathcal{B}} 
\newcommand{\BraidC}{\mathsf{B}} 
\newcommand{\vertex}{\mathsf{v}}

\newcommand{\Kone}{\mathsf{K}}
\newcommand{\Done}{\mathsf{D}}
 \newcommand{\stanperm}{\Pi}

\newcommand{\Z}{\mathbb{Z}}

\newcommand{\R}{\mathbb{R}}

\newcommand{\bbone}{\mathbb{1}}

\newcommand{\Mco}{\text{\scorpio}}
\newcommand{\Mnco}{\text{\Scorpio}}

\newcommand{\uX}{\underline{X}} 
\newcommand{\uY}{\underline{Y}} 
\newcommand{\ur}{\underline{r}} 
\newcommand{\NEP}{\mathfrak{P}_\mathsf{NE}} 
\newcommand{\Path}{\mathfrak{p}} 
\newcommand{\Qath}{\mathfrak{q}} 

\newcommand{\Sym}{\mathfrak{S}} 

\newcommand{\indicatorgroup}{\mathbb{I}}

\newcommand{\Eulerian}{A} 

\newcommand{\linorder}{\ell}

\begin{document}

\title{Chromatic word-quasisymmetric functions of matroids}

\author{Raul Penaguiao\textsuperscript{1}}
\email{raul.penaguiao@proton.me}
\address{University of Basel}
\footnotetext[1]{Computational Physiology and Biostatistics Group, Department of Biomedical Engineering, University of Basel, Basel, Switzerland}
\author{Sophie Rehberg}
\email{rehberg.sophie@uqam.ca}
\address{Université du Québec à Montréal}
\date{September 23, 2026} 

\begin{abstract}
Billera, Jia, and Reiner (2009) introduced the quasisymmetric functions of matroids and showed that this defines a Hopf algebra homomorphism which is a valuative invariant, i.e., isomorphic matroids have the same quasisymmetric function and polytopal subdivisions of matroid base polytopes define relations among the corresponding quasisymmetric functions.
In this project we study an analogue in non-commuting variables, the word-quasisymmetric functions.
To every matroid $M$ we associate a word-quasisymmetric function $\psi(M)$ and call this the chromatic word-quasisymmetric functions of a matroid.

Matroids and word-quasisymmetric functions form Hopf algebras, and our map $\psi$ between them is a homomorphism.
We want to study the kernel, equivalently the image, of the map $\psi$ from matroids to word-quasisymmetric functions, that is, we would like to understand which matroids are indistinguishable by the chromatic word-quasisymmetric functions.
The map $\psi$ is not an invariant, but we can show that it is valuative.
Using Schubert matroids and nested matroids, special classes of matroids, we prove a lower bound of $2^d-d$ for
the rank of the map $\psi$ from matroids to the chromatic word-quasisymmetric functions in degree $d$
and conjecture the upper bound of $d!$ is tight.

\end{abstract}

\maketitle

\tableofcontents

\section*{Acknowledgments}

The authors would like to thank Federico Ardila, Chris Eur, Alex Fink, Mieke Fink, and Georg Loho for helpful conversations.
Part of this work was done during a week-long visit at the MPI Leipzig.
The first author is very appreciative of helpful comments from Hunter Spink.
The second author was partially funded by the Deutsche Forschungsgemeinschaft (DFG, German Research Foundation) under Germany's Excellence Strategy – The Berlin Mathematics
Research Center MATH+ (EXC-2046/1, project ID: 390685689).

\section*{Disclosure of AI assistance}

During the preparation of this manuscript, the authors used Claude Sonnet 5 to develop code for the web interface, get suggestions for code snippets used in computations and for editing parts of the exposition.

All mathematical arguments and results are original work by the authors.

\section{Introduction}

A \Def{matroid} $\matroid=(\{1,\dots,\nd\},\bases)$ is a collection $\bases$ of subsets $B\subseteq\{1,\dots,\nd\}$ called bases, fulfilling certain axioms (see \cref{ssec:prelimMatroids} for precise definition).
For a weight function $f\colon\{1,\dots,\nd\}\to\R$ we define the weight of a basis $B\in\bases$ by
\begin{equation*}
  f(B)=\sum_{i\in B}f(i)\,.
\end{equation*}
Given a matroid $\matroid$, we say that a weight function $f$ is \Def{$\matroid$-generic} if there exists a unique basis $B$ in $\matroid$ with maximal value $f(B)$.

Consider, for example, the matroid $\unifMat{1}{3}=(\{1,2,3\}, \{\{1\}, \{2\}, \{3\}\})$.
Then a function $f\colon\{1,2,3\}\to\R$ is $\unifMat{1}{3}$-generic if and only if it has a unique maximal value among $\{f(1), f(2), f(3)\}$.

Let $\bw = \bw_1, \bw_2, \ldots $ be an infinite collection of non-commuting variables.
For every $f: [\nd] \to \Z_{> 0} $, we define the monomial $\bw_{f} = \bw_{f(1)} \bw_{f(2)} \cdots \bw_{f(\nd)}$
and for a matroid $\matroid=([\nd],\bases)$ we define the \Def{chromatic word-quasisymmetric function of $\matroid$} by
\begin{equation*}
 \ncPsi(\matroid) \coloneqq \sum_{\substack{f\colon[\nd]\to \Z_{\geq1}\\ \text{ is $\matroid$-generic}}} \bw_f\,.
\end{equation*}
This defines a map $\ncPsi : \MatLinSpace \to \WQSym$ from the matroids $\MatLinSpace$ to the word-quasisymmetric functions $\WQSym$, i.e., formal power series in the non-commuting variables ${\bw = \bw_1, \bw_2, \ldots }$ with certain symmetry conditions (see \cref{ssec:chrmwqsym}).

 Continuing the example from above, the chromatic word-quasisymmetric function of the rank one uniform matroid on three elements is
\begin{equation*}
  \ncPsi(\unifMat{1}{3}) = \sum_{i\neq j, j\neq k, k\neq i}\bw_i\bw_j\bw_k + \sum_{i<j} \bw_i^2\bw_j + \sum_{i<j} \bw_i\bw_j\bw_i + \sum_{i<j} \bw_j\bw_i^2\,.
\end{equation*}
In \Cref{ex:wordquasisym_new} we will give a more detailed discussion.
We will also see that the chromatic word-quasisymmetric function of a matroid is very closely related to the normal fan of the matroid base polytope (\Cref{sec:preliminaries}).

After allowing the variables in the chromatic word-quasisymmetric function of a matroid to commute we recover the quasisymmetric functions of matroids introduced by Billera--Jia--Reiner \cite{billera2009quasisymmetric}.
While the quasisymmetric functions of matroids are an isomorphism invariant, this is not the case for the chromatic word-quasisymmetric functions of matroids.

In this paper we study the map $\ncPsi : \MatLinSpace \to \WQSym$.
The linear span of matroids, $\MatLinSpace$, as well as word-quasisymmetric functions $\WQSym$ carry the structure of graded Hopf algebras and we will see that the map $\ncPsi $ is a graded Hopf algebra homomorphism (see \cref{sec:hopfmonoid}).
The Hopf algebra structure is not essential to the main proofs but provides additional structural insight and motivation.
The algebra $\WQSym$ plays an important role among combinatorial Hopf algebras: it is the terminal object in the category of combinatorial Hopf algebras (see \cite[Theorem 11.23]{aguiar2010monoidal}, \cite[Theorem 49]{penaguiao2020kernel}), meaning that any such morphism from matroids factors through it, making $\ncPsi$ the universal chromatic map.

We want to describe the kernel of $\ncPsi$, and compute the dimension of $\im\ncPsi\subseteq\WQSym$, which directly informs how well this map can distinguish matroids.
For the quasisymmetric function of matroids, this question is completely solved by combining results in \cite{billera2009quasisymmetric} and \cite{derksen2010valuative}.
We describe this in detail in \Cref{sec:qsym}.
Even though the map $\ncPsi\colon\MatLinSpace\to\WQSym$ from matroids to chromatic word-quasisymmetric functions of matroids is not an isomorphism invariant, our research question is motivated by the so-called kernel problem for isomorphism invariants:
Isomorphism invariants can help to decide whether combinatorial objects are isomorphic, by giving a negative certificate, i.e.,  if for two objects the isomorphism invariant differs they are certainly not isomorphic.

This kind of question has been studied extensively:
In graphs, for instance, one can use the degree sequence \cite{martin2008distinguishing}; in matroids, the count of non-broken circuit (nbc) bases \cite{ardila2023lagrangian}; and in polytopes, the f-vector \cite{ziegler2004convex}.
Usually, invariants also agree on some non-isomorphic objects.
This leads to the kernel problem, that is, understanding for which (non-isomorphic) objects the invariant under consideration is constant.
One well-studied instance of this approach is the chromatic symmetric function for graphs introduced by Stanley \cite{stanley1995symmetric}. It is still an open conjecture and an active area of research whether the chromatic symmetric function distinguishes non-isomorphic trees.
Similar chromatic%
\footnote{\emph{Chromatic} in the sense that they encode data of maps assigning positive integers to the ground set, which  we interpret as assigning colors in $\Z_{>0}$.}
invariants that have been studied are, e.g.,
the quasisymmetric function for generalized permutahedra \cite{billera2009quasisymmetric},
as well as a quasisymmetric function for posets \cite{feray2020cyclic}.
In \cite{breuer_scheduling_2016} a geometric perspective on chromatic quasisymmetric functions for scheduling problems is developed.
This offers a unifying perspective on several of the aforementioned invariants.

Recent developments have shifted interest toward analogues in non-commuting variables %
of classical chromatic functions. In particular, Gebhard and Sagan introduced a non-commutative symmetric function for graphs \cite{gebhard2001chromatic}; Féray extended this to posets \cite{feray2020cyclic}; Breuer--Klivans study this for scheduling problems \cite{breuer_scheduling_2016}; and a similar construction was developed for hypergraphic polytopes in \cite{penaguiao2020kernel}.
These constructions suggest that richer algebraic structures might carry more distinguishing power than their commutative counterparts.
Chromatic functions in non-commuting variables tend to have interesting images.
This is also the case for chromatic word-quasisymmetric functions of matroids, as we will see in the following.

We first show that the map $\ncPsi\colon\MatLinSpace\to\WQSym $ is valuative (\cref{thm:wqsym_are_val} and \cref{cor:ncPsistronglyvaluative}).
This is an important property in order to compute relations that are contained in the kernel $\ker\ncPsi$ and, using results from Derksen--Fink \cite{derksen2010valuative}, allows us to find a generating set for the image $\im\ncPsi$.

Our main results are the following (see below for precise definitions):
\begin{thm}
\label{thm:kernelcontains}
The kernel of $\ncPsi$ contains the valuative relations (see Equation~\eqref{eq:valdef}) and the  loop-coloop relations (see \Cref{prop:loopcolooprelation} below).
In this way, the function $\ncPsi$ is determined by its values on loopless nested matroids.
\end{thm}
\Cref{thm:kernelcontains} implies that the dimension of the $\nd$th graded piece in the image $\im\ncPsi$ is at most $\nd!$.
We can further prove the following result.
\begin{thm}
\label{thm:lower_bound}
The image of the $\nd$-degree component $\ncPsi_\nd$ has dimension at least ${2^\nd - \nd}$.
Specifically, the chromatic word-quasisymmetric functions of a collection of $2^\nd - \nd$ Schubert matroids are linearly independent.
\end{thm}
The proof of this result is completely elementary, the main tools being the Gale order (see \cref{ssec:schubert} for a definition), the greedy algorithm and basic row operations on matrices.
We conjecture that the upper bound is tight:%
\footnote{During the final stages of preparing this manuscript, with the assistance of an AI agent, we obtained a proof sketch for this conjecture. We are currently in the process of carefully verifying, analyzing and refining the arguments for forthcoming work.}
\begin{conj}
\label{conj:kernelspan}%
The kernel of $\ncPsi$ is spanned by the valuative relations and the loop-coloop relations.
Equivalently, the values of $\ncPsi$ on loopless nested matroids are linearly independent, and the image of $\ncPsi$ is a Hopf algebra whose homogeneous pieces have dimension $\nd!$.
\end{conj}
In particular, we conjecture that the linear independence is witnessed by a special collection of $\nd!$ set compositions, which we call \Def{max-min set compositions}. See \cref{sec:computational} for details.
Asymptotically, the bound in \cref{thm:lower_bound} falls short of the conjectured dimension.
If the conjecture holds, we obtain a new Hopf algebra with a basis indexed by permutations, an active area of research (see, e.g., \cite{li_hopf_2026} and references therein).

In what follows, we survey the relevant background on matroids, their matroid base polytopes, valuative functions, and word-quasisymmetric functions (\cref{sec:preliminaries}), prove the valuative nature of the map $\ncPsi $ (\cref{sec:weaklyvaluative}), and analyze the rank of $\ncPsi$ (\cref{sec:rank}).
In \cref{sec:qsym} we survey and combine results from Billera--Jia--Reiner \cite{billera2009quasisymmetric} and Derksen--Fink \cite{derksen2010valuative} to give a full description of the image and kernel for the map $\cPsi\colon\MatLinSpace\to\QSym$ assigning to every matroid the chromatic quasisymmetric function mentioned above.
This story is parallel to the one we develop for the non-commutative case and hence serves as a motivation.
Further appendices include technical results on Hopf morphisms (\cref{sec:hopfmonoid}), the double-chain description of nested matroids (\cref{appendix:DCarenested}), and computational data (\cref{sec:data}).

We further provide a Python package \cite{penaguiao2026chromatic} version 1.0.2, that we used to gather computational evidence supporting our conjecture.
There is also an easy-to-use web interface \cite{penaguiao2026web}.

\section{Preliminaries\label{sec:preliminaries}}

We recall some notation.
If $\nd\geq 0$, we write $[\nd] = \{ 1, \ldots, \nd\}$.
In this way, $[0]$ is the empty set.
Most combinatorial objects that we study can be labeled by elements of an arbitrary finite set $\groundset$, but we will introduce many of them with labels in $[\nd]$ for ease of notation. Schubert matroids are an exception: they require a total order on the ground set, and we use $[\nd]$ with its natural order throughout for that purpose.

We will work always over the reals $\R$, however every result in this paper works for a field with characteristic zero.

\subsection{Basic combinatorial objects and notation}\label{ssec:combbasics}

If $\al = (\al_1, \ldots, \al_k)$ is a tuple of real numbers, we write $|\al| = \sum_i \al_i$.
Let $\nd$ be a non-negative integer.
A \Def{composition} $\alpha $ of $\nd$ is a tuple $(\alpha_1, \ldots, \alpha_k)$ of positive integers such that $|\alpha| = \nd$.
We write $\alpha \models \nd$ and we say that $\ell(\alpha ) = k$ is the \Def{length} of $\alpha$.
There is a unique composition of $0$, the empty tuple.

For a finite set $\groundset$, a \Def{set composition} $\opi$ of $\groundset$ is a tuple $(\pi_1, \ldots, \pi_k)$ of non-empty disjoint subsets of $\groundset$ such that $\bigcup_i \pi_i = \groundset$.
We write $\opi \models \groundset$ and say that ${\ell(\opi) = k}$ is the \Def{length} of $\opi$.
There is a unique set composition of $\emptyset $.
The \Def{type} of a set composition $\opi $ of a set  $\groundset$ is a composition $\alpha \models |\groundset|$.
Explicitly, if $\opi = (\pi_1, \cdots , \pi_k)$, then $\alpha(\opi) = (|\pi_1|, \ldots, |\pi_k|)$.
Given a function $f:\groundset \to \R$, its \Def{set composition type}  $\opi(f) = (\opi(f)_1, \ldots , \opi(f)_k)$ is defined as the unique set composition such that $f$ is constant on each $\opi(f)_i$, and if $i< j, \, a \in \opi(f)_i, \, b\in \opi(f)_j$ then $f(a) < f(b)$.
 Vice versa, for a set composition $\opi\models\groundset$  we define the map $f_{\opi} \colon\groundset\to[k]$ by $f_{\opi}(a)\coloneqq j$ for $a\in \pi_j$.

\subsection{Basic polyhedral geometry}
We assume some familiarity with convex polytopes, see, e.g.,  \cite{Gruenbaum2003,Ziegler1998} for introductions.
For a direction $\vy\in\big(\R^d\big)^*$ we define the \Def{$\vy$-maximal face} $\Pol^\vy$ of a polytope $\Pol$ by

\begin{equation*}
  \Pol^\vy\coloneqq\{\vx\in \Pol\ \colon\ \vy(\vx)\geq \vy(\vx')\quad \text{for all } \vx'\in \Pol\}.
 \end{equation*}
 For a face $\face$ of a polytope $\Pol$ define the \Def{open} and \Def{closed normal cone} 
 $\normalcone_\Pol^\circ(\face)$ and $\normalcone_\Pol(\face)$ to be the set of all directions that (strictly) maximize $\face$ in $\Pol$, that is,
 \begin{equation*}
\begin{split}
 \normalcone_\Pol^\circ(\face) &\coloneqq \big\{\vy\in\big(\R^d\big)^* \ \colon\ \Pol^\vy=\face\big\}\\
 \normalcone_\Pol(\face) &\coloneqq \big\{\vy\in\big(\R^d\big)^* \ \colon\ \Pol^\vy\supseteq \face\big\}.
\end{split}
\end{equation*}
Collecting the normal cones $\normalcone_\Pol(\face)$ of all faces $\face$ of a polytope $\Pol$ defines the \Def{normal fan}
\begin{equation*}
 \Nfan(\Pol)\coloneqq \big\{\normalcone_\Pol(\face)\ \colon\ \face \text{ a face of }\Pol\big\}\,.
\end{equation*}
It can be checked that this is indeed a polyhedral fan.
Polytopes with the same normal fan are called \Def{normally equivalent}.
In particular, translating and scaling polytopes preserves the normal fan.

The 
\Def{braid arrangement} $\BraidA_d$ is the hyperplane arrangement consisting of the finite set of hyperplanes $\Hyp_{ij}\coloneqq\{\vx\in\R^d\ \colon\ \vx_i=\vx_j\}$ for $i,j\in [d]$, $i\neq j$.
The connected components of $\R^d\setminus \bigcup \BraidA_d$ are the \Def{(open) regions} of the  arrangement. 
The \Def{closed regions} of the braid arrangement are the topological closures of the open regions.
They are polyhedral cones and their faces are the \Def{faces} of the braid arrangement, also called \Def{braid cones}.
The braid cones can be described uniquely by set compositions $ \opi=(\pi_1,\dots,\pi_k)\models[\nd]$ (\Cref{lem:braid}).
We therefore denote them by $\BraidC_{\pi_1,\dots,\pi_k}$.
For more details about concepts on hyperplane arrangements see, for example, \cite{stanley2007introduction}.
The faces of the braid arrangement $\BraidA_d$ form the \Def{braid fan}.
\begin{lm}\label{lem:braid}
 The faces of the braid arrangement $\BraidA_{\nd}$, also called \Def{braid cones}, can be described uniquely by set compositions ${[\nd]}=T_1\uplus\dots\uplus T_k$:
  \begin{equation*}
   \begin{split}
    \BraidC_{\pi_1,\dots,\pi_k}\coloneqq& \big\{\vx\in\R^{\nd}\, \colon\, \vx_i=\vx_j\text{ for all }i,j\in \pi_a,\\
	& \hspace{12ex} \vx_i\leq \vx_j\text{ for all }i\in \pi_a, j\in \pi_b \text{ with } a<b\big\}\\
    =&\cone\{\onebb_{\pi_1}, \onebb_{\pi_1\cup \pi_2},\dots,
			\onebb_{\pi_1\cup\dots\cup \pi_{k-1}} \} + \rspan\{\onebb_{{[\nd]}}\}
   \end{split}
\end{equation*}
with
\begin{equation*}
 \dim \BraidC_{\pi_1,\dots,\pi_k}=k\,,
\end{equation*}
where $\onebb_T$ for some subset $T\subseteq {[\nd]}$ is the $0/1$-vector with entries equal to one for indices in the subset $T$ and zero otherwise.
\end{lm}

\begin{rem}\label{rem:mapsaspoints}
 Given a function $f\colon[\nd]\to\R$ we can identify this function with the point $(f(1), \dots,f(d))\in \R^d$. Similarly, any point $\vx=(\vx_1,\dots,\vx_d)$ defines a function $x\colon[d]\to \R$ by $x(i)=\vx_i$.
Therefore we will identify points in $\R^d$ and functions  $f\colon[\nd]\to\R$ throughout this work.
Then a function $f\colon[\nd]\to\R$ is contained in the braid cone $\BraidC_{\opi(f)}$ given by its set composition type $\opi(f)$ as defined in \Cref{ssec:combbasics} and the open braid cone $\BraidC_{\opi}$ contains all the points with set composition type $\opi$.
\end{rem}

We define the \Def{standard permutahedron}  $\stanperm_d$
as the convex hull of the $d!$ permutations of the point $(1,2,\dots,d)$, that is, we define the standard permutahedron $\stanperm_d$ by
 \begin{equation*}
  \stanperm_d\coloneqq\conv\big\{(\vx_i)_{i\in [d]}\in \R^d \colon \{\vx_i\}_{i\in [d]}=[d]\big\}
		\subseteq \R^d.
 \end{equation*}
Note that the standard permutahedron is of dimension $d-1$ since all vertices are contained in a hyperplane with constant coordinate sum.
 In our definition, standard permutahedra are integer polytopes.
 The facet description of the standard permutahedron is given by
\begin{align}\label{eq:permHdesc}
    \sum_{i=1}^d \vx_i&=d+\left(d-1\right)+\dots+1=\frac{d(d+1)}{2}=\tbinom{d+1}{2}\\
    \sum_{i\in T}\vx_i&\leq d+\left(d-1\right)+\dots+\left(d-\lvert{T}\rvert+1\right)=\tbinom{d+1}{2}-\tbinom{\lvert{T}\rvert+1}{2}  \quad \text{for all }T\subseteq [d].
\end{align}
Every face of the standard permutahedron can be described combinatorially by set compositions and the normal fan of the standard permutahedron is the braid arrangement defined above.
See \cref{abb:permcon} for an example.
\begin{figure}
\centering
\begin{subfigure}{.47\textwidth}
  \includegraphics[trim=180 50 190 115, clip, width=\linewidth]{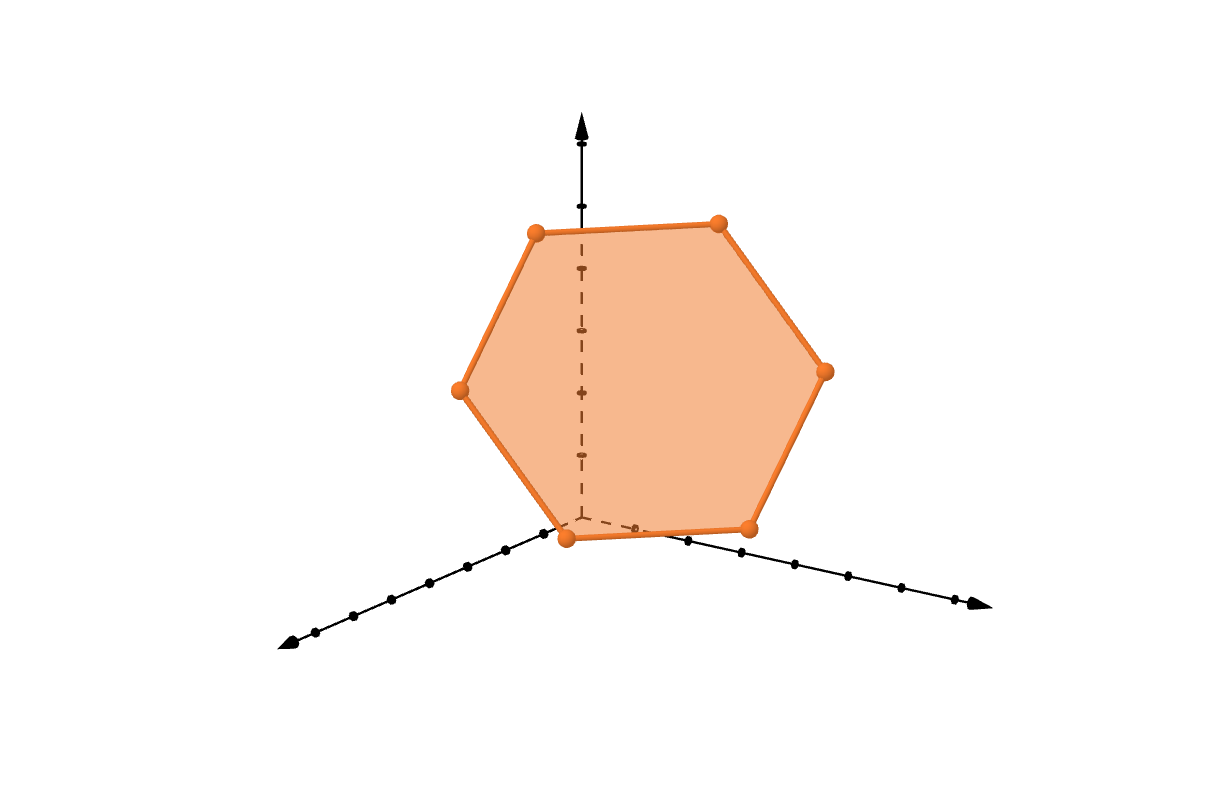}
  \caption{Standard permutahedron $\stanperm_{3}$}
\end{subfigure}\hfill
\begin{subfigure}{.47\textwidth}
  \includegraphics[trim=180 50 190 115, clip, width=\linewidth]{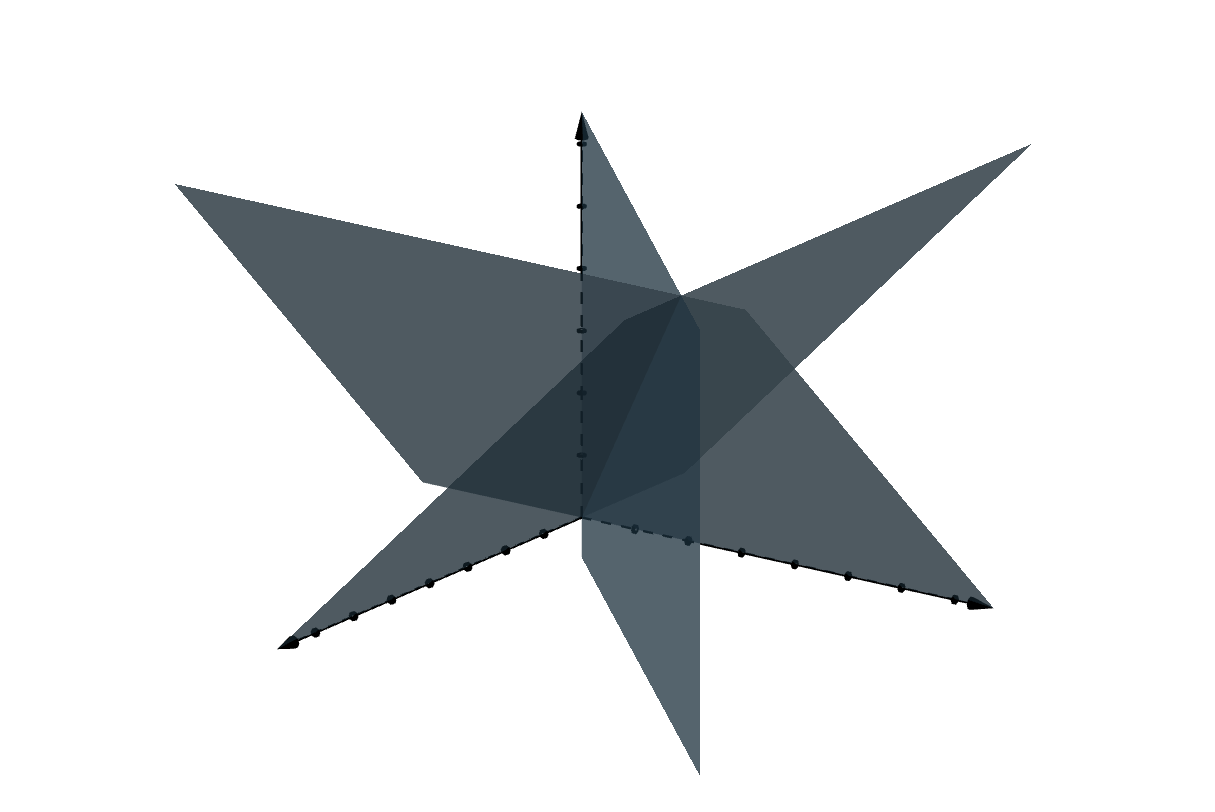}
  \caption{Braid arrangement $\BraidA_{3}$}\label{abb:braid}
\end{subfigure}
 \caption{The standard permutahedron $\stanperm_{3}\subseteq \R^3$ (left) and the normal fan in $(\R^3)^*$ (right), where the intersection line is the normal cone $\normalcone_{\stanperm_{3}}(\stanperm_{3})$, the half hyperplanes are the normal cones of the edges and the full-dimensional cones are the normal cones of the vertices of $\stanperm_{3}$.
}%
\label{abb:permcon}
\end{figure}

We say a  fan $\Nfan$ is a \Def{coarsening} of another  fan $\Nfan'$ if every cone in $\Nfan$ is the union of some cones in $\Nfan'$.
A polytope $\Pol\subseteq \R^d$ is a \Def{generalized permutahedron} or \Def{deformed permutahedra}
 if its normal fan $\Nfan(\Pol)$ is a coarsening of the normal fan $\Nfan(\stanperm_d)$ of the standard permutahedron $\stanperm_d$, that is, it is a coarsening of the fan induced by the braid arrangement $\BraidA_d$.
This is equivalent to saying a polytope $\Pol\subseteq \R^d$ is a generalized permutahedron if and only if every edge is parallel to $\be_i-\be_j$ for some $i\neq j\in[d]$.

\subsection{Matroids and matroid base polytopes}\label{ssec:prelimMatroids}
We introduce some matroid theory that can be found in \cite{Oxley2006}.
A \Def{matroid} is a pair $\matroid=(\groundset,\bases)$ where $\groundset$ is  a finite set and $\bases$ is a collection of subsets of $\groundset$, called \Def{bases}, such that
\begin{enumerate}
 \item[(B1)] $\bases\neq \emptyset$ 
 \item[(B2)] for distinct members $A\in\bases$ and $B\in\bases$ and every element $a\in A\setminus B$ there exists an element $b\in B\setminus A$ such that $A\setminus\{a\}\cup\{b\} \in\bases$ (basis exchange axiom).
\end{enumerate}
A subset $I\subseteq\groundset$ that is contained in some basis is called \Def{independent}, otherwise it is \Def{dependent}.
Given a matroid $\matroid = (\groundset, \bases)$, one can define the \Def{rank function} $\rk_\matroid: 2^\groundset \to \Z$ of the matroid as 
\begin{equation*}
\rk_\matroid(A) \coloneqq \max_{B \in \bases} \{ |A \cap B|\}
\end{equation*}
where we run over all bases $B$ of $\matroid$.
Note how we always have $\rk_\matroid(\emptyset) = 0$, we also have the submodularity relation
\begin{equation*}
\rk_\matroid(A) + \rk_\matroid(B) \geq \rk_\matroid(A\cap B) + \rk_\matroid(A \cup B)\, .
\end{equation*}
We call $\rk_\matroid(\groundset)$ the \Def{rank of the matroid} $\matroid$.
\begin{smpl}
Take the matroid $\matroid = ([4], \{ \{1, 2\}, \{1, 3\}, \{1, 4\}, \{2, 3\}, \{2, 4\}\})$.
Then we have, for instance, $\rk_\matroid(\{1\}) = 1$, $\rk_\matroid(\{1, 2, 3\}) = 2$, and $\rk_\matroid(\{3, 4\}) = 1$.
\end{smpl}
A subset  $F\subseteq \groundset $ is called a \Def{flat} if $\rk(F)<\rk(F\cup\{a\})$ for every element $a\in\groundset\setminus F$.
A subset $C\subseteq\groundset$ is called a \Def{circuit} if it is a minimal dependent set, i.e., it is dependent but removing any element gives an independent set.
A \Def{loop} (resp. \Def{coloop}) is an element $e\in\groundset$ that is contained in no (resp. every) basis.

For a map $f\colon\groundset\to\R$, we define the \Def{weight of a basis} $B\in\bases$ to be
\begin{equation*}
 f(B)\coloneqq\sum_{i\in B}f(i)\,.
\end{equation*}
If for a matroid $\matroid=(\groundset,\bases)$ there is a unique basis $B$ with maximum weight $f(B)$ then we say the map $f$ is an \Def{$\matroid$-generic map}.
Note that a map $f$ is $\matroid$-generic if and only if there is a basis $B$ such that for each basis $(B \setminus \{b\}) \cup \{c\}$ with $b \in B$ and $c \notin B$, we have $f(c)  <  f(b)$. Therefore $\matroid$-genericity for a map $f$ only depends on its set composition type  $\opi(f)$.
Recall that for a set composition $\opi = \pi_1 | \cdots | \pi_l \models U$ we define $f_{\opi} \colon U \to \R $ as
\begin{equation}
\label{eq:setcompositiongeneric}
    f_{\opi}(i) \coloneqq j \text{ whenever } i \in \pi_j\,.
\end{equation}
A set composition $\opi $ is \Def{$\matroid$-generic} if the function $f_{\opi}$ is $\matroid$-generic.
Note that by the greedy algorithm every set composition with only singleton blocks is generic for every matroid.
Similarly, the set composition with just one block is not $\matroid$-generic, except when the matroid has exactly one basis.

To every matroid $\matroid=(\groundset,\bases)$ there exists an associated \Def{matroid base polytope} $\Pol_\matroid\subset\R\groundset$
defined as the convex hull of the indicator vectors of the bases, i.e.,
\begin{equation*}
 \Pol_\matroid\coloneqq\conv\{\onebb_B \in\R\groundset\colon B\in \bases\}\,,\quad \text{ where } (\onebb_B)_i\coloneqq \begin{cases}
                                                                                    1, \text{ if }i\in B\\
                                                                                    0, \text{ else.}
                                                                                   \end{cases}
\end{equation*}
Note, since every basis $B\in\bases$ of a matroid $\matroid=(\groundset,\bases)$ has the same cardinality, 
the matroid base polytope $\Pol_\matroid$ is contained in a hyperplane defined by coordinate sum equals $\rank(\matroid)$ and has codimension at least $1$.
\begin{smpl}
    The matroid base polytope of the \Def{uniform matroid}
    \begin{align*}
     \uniformatroid{r}{\groundset}\coloneqq (\groundset,\{ I\subseteq \groundset\colon \lvert I\rvert = r\} )
    \end{align*}
    is the \Def{hypersimplex}
    \begin{align*}
        \hypersimplex{r}{\groundset}\coloneqq& \conv\{\onebb_I\in\R\groundset\colon I\subseteq\groundset, \lvert I\rvert = r\}\\
                =&[0,1]^\groundset\cap \left\{\vx\in\R\groundset\colon\sum_{i\in\groundset}\vx_i=r\right\}\,.
    \end{align*}

\end{smpl}

\begin{thm}[{\cite[Theorem 4.1]{gelfand_combinatorial_1987}}]
    A polytope $\Pol$ is a matroid base polytope (of a unique matroid) if and only if the following conditions hold:
    \begin{enumerate}
        \item the vertices of $\Pol$ are a subset of the vertices of a hypersimplex $\hypersimplex{r}{d}$,
        \item every edge direction is of the form $\ve_i-\ve_j$.
    \end{enumerate}
\end{thm}
It follows that matroid base polytopes are  generalized permutahedra, i.e., their normal fans coarsen the braid fan. See, e.g., \cite{postnikov2009permutohedra}.
Moreover, matroids and their matroid base polytopes are in bijection and we will identify them.
We will freely switch back and forth between matroids and matroid base polytopes, e.g., for a function defined on matroids we will implicitly assumed it is also defined for matroid base polytopes and vice versa.

\begin{obs}\label{obs:rkpol}
 Let $\Pol_\matroid$ be the matroid base polytope of the matroid $\matroid$. Then the rank function $\rk_\matroid$ can be computed as
\begin{equation*}
 \rk_\matroid(A) = \max_{\vx\in\Pol_\matroid} \sum_{i\in A}\vx_i\,.
\end{equation*}
\end{obs}

 \begin{rem}\label{rem:mapsaspoints2}
  As in \cref{rem:mapsaspoints}  we can identify functions  $f\colon \groundset\to\R$ with points $\vy_f$ in the dual vector space $\R^U=(\R  U)^*$.
Then the weight $f(B)$ of a basis is the value of the linear functional $\vy_f$ on the vertex $\onebb_B$ corresponding to the basis $B$.
The value $f(B)$ is a unique maximum among all bases in  matroid $\matroid$, i.e., $f$ is $\matroid$-generic, if and only if the linear functional $\vy_f$ is contained in the open normal cone $\normalcone_{\Pol_\matroid}(\onebb_B)$ of the matroid base polytope $\Pol_\matroid$ at the vertex $\onebb_B$ corresponding to the maximizing basis $B$.

Similarly, a set composition $\opi$ is $\matroid$-generic if and only if the corresponding open braid cone $\BraidC_{\opi}^\circ$ is contained in the interior of a full-dimensional cone in the normal fan of $\Pol_{\matroid}$, i.e., the normal cone of a vertex $\vertex$.

As observed above, every set compositions with only singleton blocks is generic for every matroid, since the normal fan of the matroid base polytope coarsens the braid fan.
Similarly, the set composition with just one block is not $\matroid$-generic, except when the matroid base polytope is a point.

\end{rem}

Given two matroids $\matroid, \natroid$ in respective disjoint ground sets $\groundset, \othergroundset$, we define the \Def{direct sum} of $\matroid, \natroid$, denoted $\matroid\oplus \natroid$, as the matroid in $\groundset\cup\othergroundset$ such that:
\begin{equation*}
\bases(\matroid\oplus \natroid) \coloneqq \{B_1 \cup B_2\ |\ B_1 \in \bases(\matroid), \ B_2 \in \bases(\natroid) \} \, .
\end{equation*}
A matroid is called \Def{connected}, if it cannot be written as the direct sum of two non-empty matroids.
For a subset $F\subseteq\groundset$ the \Def{restriction} $\matroid|_F$ is defined as
\begin{equation*}
 \matroid|_F \coloneqq (F, \{B\cap F \ |\ B\in\bases(\matroid)\}).
\end{equation*}
A subset $F\subseteq\groundset$ is a \Def{connected flat}, if it is a flat and $\matroid|_F$ is connected.

We can define a vector space $\Mat_d$ by taking formal linear combinations of matroids on the ground set $[d]$:
\begin{equation}\label{eq:defmatspace}
 \Mat_d\coloneqq \spn \{ \matroid\ | \text{ matroid with ground set $[d]$ }\} \, .
\end{equation}
With the direct sum of matroids as a product we can define the graded vector space
\begin{equation*}
\MatLinSpace \coloneqq \oplus_{d \geq 0 }\Mat_d\,.
\end{equation*}
This is a Hopf algebra when endowed with the product (via direct sum of matroids) and coproduct (via deletion and contraction).
The interested reader can see more on this in \cite[Chapter 13]{aguiar2023hopf} and \cite[Section 12]{schmitt1994incidence}.

In this paper, to simplify notation we mostly work with the ground set $\groundset=[\nd]$, except in the Hopf monoid context in \cref{sec:hopfmonoid}.

\subsection{Special matroids: Schubert matroids and nested matroids}

We are interested in a special subset of matroids, that has many nice combinatorial descriptions --- and many names used throughout the literature. We will use the term \emph{nested matroid} for any matroid in this subset and the term \emph{Schubert matroid} for a specific representative in every isomorphism class, i.e., no two Schubert matroids are isomorphic.

The name \emph{Schubert matroid} was coined in \cite{sohoni1999rapid}, and these matroids were originally studied in \cite{oxley_matroids_1982} under the name of \emph{transversal matroids}.
Schubert matroids arise in many different, seemingly unrelated ways.
For instance, these can be defined via lattice path matroids with top path $N^rE^{d-r}$.

In many places the term Schubert matroids is used for all the matroids isomorphic to these, as in \cite{eur2023stellahedral,ferroni_valuative_2024,ferroni_polytope_2025}.
As mentioned above, we will not follow this convention.

The class of Schubert matroids \say{up to isomorphism} has also gathered a collection of nomenclature.
For example in \cite{mayhew2006equitable} these were called \emph{nested matroids} (the name we will be using throughout this paper) owing to the fact that the underlying structure of cyclic flats forms a nested collections of sets, whereas in \cite{klivans2003combinatorial} these were called \emph{shifted matroids}, owing to their relation with shifted simplicial complexes.
In \cite{berget2009tableaux}, these matroids were called \emph{freedom matroids} and in \cite{bonin2006lattice} they were called \emph{generalized Catalan matroids}.
In \cite{billera2009quasisymmetric} these were called \emph{PI-matroids}.
Nested matroids were introduced in \cite{crapo1965single} to show that there are at least $2^\nd$ isomorphism classes of matroids of size $\nd$.
Later, these were studied by \cite{bonin2006lattice} and \cite{hampe_intersection_2017}.

See \Cref{ssec:nestedmatroids} for more details.

\subsubsection{Schubert matroids}\label{ssec:schubert}

In order to define a Schubert matroid we need a total order on the ground set $\groundset$.
As mentioned above, we will use $\groundset=[\nd]$, which comes with the order of the natural numbers.

We start by defining the \Def{Gale order} $\leqG$ on subsets of the same size.
For subsets $A=\{a_1<\dots<a_r\}\subseteq [\nd]$ and $B=\{b_1<\dots<b_r\}\subseteq [\nd]$ define
\begin{equation}\label{eq:defGaleorder}
 B\leqG A \quad \mathrel{\vcentcolon\Leftrightarrow}  \quad b_i\leq a_i, \text{ for all } i=1,\dots,r
\end{equation}
For a subset $A\subseteq [\nd]$ we define the \Def{Schubert matroid} $\SM(A)$
by
 \begin{align}\label{eq:defSchubertmatroid}
  \SM(A)&\coloneqq ([\nd],\{B\subseteq [\nd] \colon \lvert B\rvert=r,\, B\leqG A\} )
 \end{align}
Note that the rank of the Schubert matroid $\SM(A)$ is $r=\lvert A\rvert$.

We present an alternative way of constructing Schubert matroids using lattice paths.
This will be useful for demonstrations later, as well as introducing a good visualization for the matroids.
On a $(\nd-r)\times r$ grid, a \Def{north-east path} $\Path$ (\Def{NE-path} for short) is a sequence of steps from $(0,0)$ to $(\nd-r,r)$ where each step increases exactly one of the two coordinates by one.
In this way, we will have $r$ steps to the north (increasing the second coordinate) and $\nd-r$ steps to the east (increasing the first coordinate).
Equivalently, it is a connected collection of edges with no bends of the form $\llcorner $ or $\urcorner$.
Let $\NEP(\nd, r) $ be the set of NE-paths in the $ (\nd-r)\times r$ grid starting on the bottom left and ending at the top right corner.
We can also write a path $\Path\in\NEP(\nd,r)$ as a sequence $(c_i)_{i=1}^d$ of characters in $\{N, E\}$ according to which direction (north $N$ or east $E$) the path takes from the bottom left to the top right.
Then a path $\Path\in\NEP(\nd,r)$ contains precisely  $r$ north steps $N$.
There is a canonical bijection between path in $\NEP(\nd,r)$ and $r$-element subsets of $[\nd]$ by defining
\begin{align*}
 B_\Path&\coloneqq \{i\in [\nd] | c_i = N\} \subseteq [\nd]\,,\\
 \Path_B&\coloneqq(c_i)_{i=1}^\nd\in\NEP(\nd,r)\,, \text{ where } c_i=\begin{cases}
                                           N\,,\text{ if }i\in B\,,\\
                                           E\,,\text{ if }i\notin B\,.
                                          \end{cases}
\end{align*}
Therefore the set of NE-path $\Path\in\NEP(\nd,r)$ on the  $ (\nd-r)\times r$ grid has cardinality $\binom{\nd}{r}$.

In $\NEP(\nd, r)$ we can define the order $\leqP$, and say that $\Path \leqP\Qath $ if every segment in $\Qath$ is in $\Path$ or above some segment in $\Path$.
Note that $B\leqG A$ if and only if $\Path_A \leqP \Path_B$, i.e., the corresponding posets are anti-isomorphic.
We can now equivalently define a Schubert matroid $\SM(\Path)$ for every $\Path\in\NEP(\nd,r)$ by
\begin{equation*} \SM(\Path) \coloneqq ([\nd]\,,\{ B_\Qath\ \colon \ \Path \leqP \Qath\} )\, . \end{equation*}
Of course, $\SM(A)=\SM(\Path_A)$.

\begin{smpl}\label{smpl:sm}
     The two paths given in \cref{fig:samplepathmatroidschubert4} define matroids on $[4]$ of rank $2$:
    \begin{align}
    \bases(\SM(\{1,3\})) &=\{\{1,2\},\{1,3\}\}\label{eq:SM13}\\
     \bases(\SM(\{2,4\})) &=\{\{1,2\},\{1,3\},\{1,4\},\{2,3\},\{2,4\}\} = \binom{[4]}{2}\setminus\{3,4\}\label{eq:SM24}
    \end{align}

\begin{figure}
    \centering

\begin{tikzpicture}[scale=1]
 \draw[step=1cm, thick, gray] (0,0) grid (2,2);
 \draw[orange, ultra thick, opacity=0.8] (0,0) -- (1,0) -- (1,1) -- (2,1) -- (2,2);
 \draw[blauD, ultra thick, opacity=0.8] (0,0) -- (0,1) -- (1,1) -- (1,2) -- (2,2);
\end{tikzpicture}%
\tdplotsetmaincoords{70}{120}
\begin{tikzpicture}[tdplot_main_coords,line cap=butt,line join=round,c/.style={circle,fill,inner sep=1pt},
        declare function={a=3.6;h=2.6;hm=-2.6;}, baseline ={(0,-2)}]
        \path
        (0,0,0) coordinate (A)
        (a,0,0) coordinate (B)
        (a,a,0) coordinate (C)
        (0,a,0) coordinate (D)
    (a/2,a/2,h)  coordinate (S)
    (a/2,a/2,hm)  coordinate (T);
\fill[orange, opacity=0.15] (A) -- (B) -- (C) --(D);
\fill[orange, opacity=0.2] (A) -- (B) -- (T);
\fill[orange, opacity=0.2] (B) -- (C) -- (T);
\fill[orange, opacity=0.2] (C) -- (D) -- (T);
\fill[orange, opacity=0.2] (D) -- (A) -- (T);
\draw[blauD, opacity=0.7,line width=3pt] (B) -- (T);
\draw[gray, thick] (S) -- (D) -- (C) -- (B) -- cycle (S) -- (C) (T) -- (C) (T) -- (D) (T) -- (B);
\draw[dashed, gray, thick] (S) -- (A) --(D) (A) -- (B) (T) --(A);

\node[anchor=north] at (A) {$(1,0,0,1)$};
\node[anchor= east] at (B) {$(1,0,1,0)$};
\node[anchor=south east] at (C) {$(0,1,1,0)$};
\node[anchor=west] at (D) {$(0,1,0,1)$};
\node[anchor=south ] at (S) {$(0,0,1,1)$};
\node[anchor=north] at (T) {$(1,1,0,0)$};
    \end{tikzpicture}

    \caption{Two paths (orange and blue) defining the two matroids in \cref{smpl:sm} and the corresponding matroid base polytopes as polytopes contained in the $\hypersimplex{2}{4}$. The one in blue corresponds to a matroid with two bases as given in Equation~\eqref{eq:SM13} as there are only two paths between the blue path and the top path. The path in orange corresponds to the matroid from Equation~\eqref{eq:SM24}.}
    \label{fig:samplepathmatroidschubert4}
\end{figure}
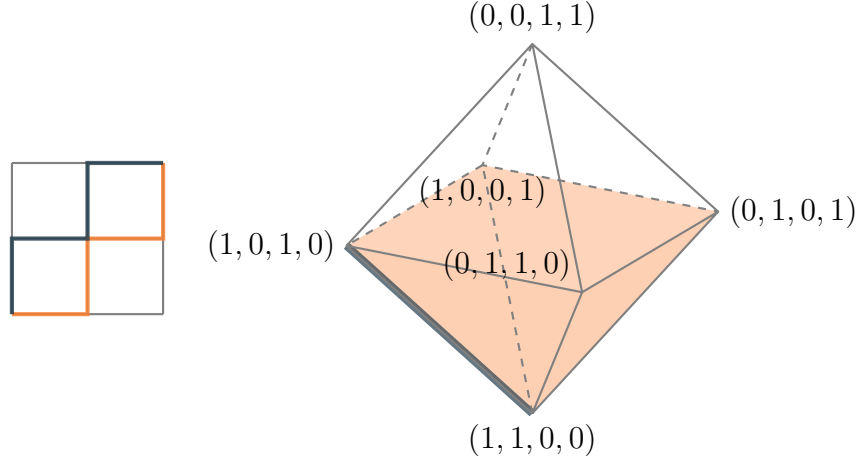
\end{smpl}

Hence, Schubert matroids are examples of \Def{lattice path matroids}, matroids whose bases are defined by NE-paths in between two fixed paths.
For Schubert matroids the upper path is $N^rE^{d-r}$.
For  more information about path matroids see, e.g., \cite{bonin2006lattice}.

\begin{smpl}
If we have a set $A = \{ 1, \ldots , r \}$, that is a set of consecutive integers that contains $1$, then there is only one basis in $\SM(A)$, corresponding to the path matroid with respect to the lower path $N^rE^{d-r}$.
We like to call them \Def{point Schubert matroids},
as their matroid base polytope is a point.
They will play a particular role in describing the kernel of our chromatic invariants, see \cref{obs:points} below.
\end{smpl}

\begin{smpl}
On the other end of the spectrum, if $A = \{ \nd - r + 1, \ldots , \nd \}$, corresponding to the path matroid with lower path  $E^{d-r}N^r$, that is a set of consecutive integers that contains $\nd$, we get the \Def{uniform matroid} of rank $r$ on ground set $[\nd]$, that is, every subset of size $r$ is a basis in this matroid.
\end{smpl}

\subsubsection{Nested matroids}\label{ssec:nestedmatroids}

There are a couple of equivalent ways of defining nested matroids.
Specifically, we can define a nested matroid by analyzing its lattice of flats, as done in \cite{oxley_matroids_1982}; by construction from double chains, the definition used in \cite{derksen2010valuative}; by inductive so-called PI-extensions, as defined in \cite{bonin2003lattice}; and by matroids isomorphic to a Schubert matroid.
In \cite{mieke2024}, other equivalent descriptions of nested matroids are documented.

We are interested in relating the concept of double chains, introduced by \cite{derksen2010valuative}, to some of the other  definitions of nested matroids introduced in the literature.
Below we present these definitions and summarize the results from the literature that establish that these are indeed equivalent.
We leave for \cref{appendix:DCarenested} the details of how a matroidal double chain relates to the other definitions of a nested matroid as well as that the connection between nested matroids and double chains is indeed one-to-one, as assumed in \cite{derksen2010valuative}.

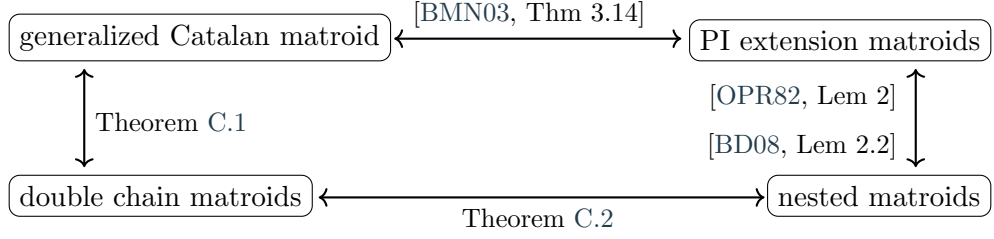
\begin{figure}
    \centering
    \begin{tikzpicture}
        \draw (-2,0) node[draw, rounded corners, anchor=north west] {\small{double chain matroids}};
        \draw (-2,1.5) node[draw, rounded corners, anchor=south west] {\small{generalized Catalan matroid}};
        \draw (11,0) node[draw, rounded corners, anchor=north east] {\small{nested matroids}};
        \draw (11,1.5) node[draw, rounded corners, anchor=south east] {\small{PI extension matroids}};
        \draw[thick, <->] (-1,0.1) -- (-1,1.4);
        \draw[thick, <->] (10,0.1) -- (10,1.4);
        \draw[thick, <->] (3.1,1.8) -- (6.9,1.8);
        \draw[thick, <->] (2.1,-0.3) -- (8,-0.3);
        \node at (3.2,1.8) [anchor=south west] {\footnotesize{\cite[Thm 3.14]{bonin2003lattice}}};
        \node at (9.9,0.7) [anchor=north east] {\footnotesize{\cite[Lem 2.2]{bonin2008lattice}}};
        \node at (9.9,0.7) [anchor= south east] {\footnotesize{\cite[Lem 2]{oxley_matroids_1982}}};
        \node at (-1,0.7) [anchor=west] {\footnotesize{\cref{thm:doublechains_are_matroids}}};
        \node at (5,-0.3) [anchor=north] {\footnotesize{\cref{thm:doublechains_are_nestedmatroids}}};%
    \end{tikzpicture}
   \caption{Diagram displaying where to find proofs relating the different notions of nested matroids.}
    \label{fig:nestedmatroidsdiagram}
\end{figure}

A flat $F\subseteq \groundset $ of a matroid $\matroid $ is called a \Def{cyclic flat} if it is the union of all circuits that it contains, that is
\begin{equation*}
F = \bigcup_{\substack{ C \text{ circuit } \\ C \subseteq F }} C \, .
\end{equation*}
A matroid $\matroid$ with ground set $\groundset$ is called a \Def{nested matroid} if the poset of cyclic flats (ordered by inclusion) is a chain.

For a matroid $\matroid $ with rank $r$, ground set $\groundset$, bases $\mathcal B$ and $e\not\in \groundset$, we define the \Def{free extension} of $\matroid $ by $e$ to be the matroid $M+e$ of rank $r$ that has bases
\begin{equation*}
\mathcal B \cup \left\{ B \cup \{e\} \setminus \{b\} | B \in \mathcal B, b \in B \right\} \, .
\end{equation*}
We also define the \Def{isthmus addition} of $\matroid $ by $e$ to be the matroid $M\oplus e$ of rank $r+1$ that has bases
\begin{equation*}
\{ B \cup \{ e \} | B \in \mathcal B \} \, .
\end{equation*}
A matroid $\matroid$ is said to be a \Def{PI-extension matroid} if it arises from the empty matroid by consecutively adding isthmuses and taking free extensions.
Both works \cite[Lemma 2]{oxley_matroids_1982} and \cite[Lemma 2.2]{bonin2008lattice} established that the definitions of PI-extension matroids and nested matroids are equivalent.

A matroid $\matroid$ with ground set $\groundset$ is called a \Def{generalized Catalan matroid} if it is isomorphic to a Schubert matroid.
In \cite[Theorem 3.14]{bonin2003lattice}, it was shown that a matroid is a PI-extension matroid if and only if it is a generalized Catalan matroid.
For more on lattice path matroids in general, see also \cite{bonin2006lattice}.

A pair of chains $(\uX, \ur)$, where $\uX~=~(X_1,~\ldots,~X_k~=~\groundset~)$ is a series of sets and $\ur = (r_1, \ldots, r_k)$ a series of integers, is called a \Def{matroidal double chains} if it satisfies the following properties:
\begin{enumerate}[(a)]
\item\label{it:DCsetinclusion} $\emptyset \subsetneq X_1 \subsetneq \cdots \subsetneq X_k = \groundset $

\item\label{it:rankinequalities} $0 \leq r_1 < \cdots < r_k = r$.

\item\label{it:DcRanktogether} $0 < |X_1| - r_1 < \cdots <|X_{k-1}| - r_{k-1} \leq |X_k| - r_k = |\groundset | - r$.
\end{enumerate}
We note that the first inequality in \ref{it:rankinequalities} and the last inequality in \ref{it:DcRanktogether} are not strict.
For a matroidal double chain $(\uX, \ur)$, we define the polytope $\Pol(\uX, \ur)$ as
\begin{equation}\label{eq:defpolofdoublechain}
    \Pol(\uX, \ur)\coloneqq \{\vx\in[0,1]^d\colon \sum_{i=1}^d \vx_i =r\,,  \sum_{i\in X_j}\vx_i \leq r_j \forall j\}\,.
\end{equation}
In \cref{thm:doublechains_are_matroids} we show that for every matroidal double chain $(\uX,\ur)$ the polytope $\Pol(\uX,\ur)$ is a matroid base polytope $\Pol_\matroid$.
Hence a double chain $(\uX,\ur)$ defines the unique matroid $\matroid(\uX, \ur)$.
A matroid $\matroid$ with ground set $\groundset $ is called a \Def{double-chain matroid} if there is a matroidal double chain $(\uX, \ur)$ such $\matroid(\uX, \ur) = \matroid$.

For convenience we will show in \cref{appendix:DCarenested} that matroidal double chains actually define matroid base polytopes and hence matroids.
Furthermore, we show that double-chain matroids are isomorphic to Schubert matroids and two distinct matroidal double chains give rise to different double chain matroids.
Vice versa, every generalized Catalan matroid has a representation as a matroidal double chain.
This establishes that matroidal double chains are a faithful representation of nested matroids.

\begin{smpl}\label{smpl:DC}
    \begin{figure}
    \centering
    \tdplotsetmaincoords{70}{120}
    \begin{tikzpicture}[tdplot_main_coords,line cap=butt,line join=round,c/.style={circle,fill,inner sep=1pt},
        declare function={a=3.6;h=2.6;hm=-2.6;}, baseline ={(0,-2)}]
        \path
        (0,0,0) coordinate (A)
        (a,0,0) coordinate (B)
        (a,a,0) coordinate (C)
        (0,a,0) coordinate (D)
    (a/2,a/2,h)  coordinate (S)
    (a/2,a/2,hm)  coordinate (T);
\fill[orange, opacity=0.15] (A) -- (B) -- (C) --(D);
\fill[orange, opacity=0.2] (A) -- (B) -- (T);
\fill[orange, opacity=0.2] (B) -- (C) -- (T);
\fill[orange, opacity=0.2] (C) -- (D) -- (T);
\fill[orange, opacity=0.2] (D) -- (A) -- (T);
\fill[grunH, opacity=0.5] (S)--(D)--(C);
\draw[blauD, opacity=0.7,line width=3pt] (B) -- (T);
\draw[gray, thick] (S) -- (D) -- (C) -- (B) -- cycle (S) -- (C) (T) -- (C) (T) -- (D) (T) -- (B);
\draw[dashed, gray, thick] (S) -- (A) --(D) (A) -- (B) (T) --(A);

\node[anchor=north] at (A) {$(1,0,0,1)$};
\node[anchor= east] at (B) {$(1,0,1,0)$};
\node[anchor=south east] at (C) {$(0,1,1,0)$};
\node[anchor=west] at (D) {$(0,1,0,1)$};
\node[anchor=south ] at (S) {$(0,0,1,1)$};
\node[anchor=north] at (T) {$(1,1,0,0)$};
    \end{tikzpicture}
    \caption{Three matroid base polytopes contained in $\hypersimplex{2}{4}$, corresponding to three double chains given in  \cref{smpl:DC}. }
    \label{fig:samplepathmatroidnested4}
\end{figure}

    Consider $\nd = 4 $ and $r = 2$. 
	The corresponding polytopes $\Pol(\uX, \ur)$ are given in \cref{fig:samplepathmatroidnested4}, all contained in the hypersimplex.

	The matroid base polytope depicted in blue is given by
	\begin{align*}
	\Pol((\{4\}, \{2, 3, 4\}, [4]), (0, 1, 2))
	&=	\left\{\vx\in[0,1]^4\colon \vx_4=0\,, \vx_2+\vx_3=1\,,  \sum_{i=1}^4 \vx_i =2\right\}\,.
	\end{align*}

    The matroid base polytope in orange corresponds to a matroid with double chain $((\{3, 4\}, [4]), (1, 2))$.
    Finally, the double chain $((\{1\}, [4]), (0, 2))$ defines the matroid base polytope shown in green.

\end{smpl}

\begin{obs}
Whenever $k = 1$, the corresponding matroid is a uniform matroid.
Whenever $k = 2, r_1 = 0, |X_1| = d-r$, the corresponding polytope is a point.
\end{obs}

\begin{smpl}
Consider the ground set $\groundset = \{0, 1, 2, 3, 4, 5, 6, 7, 8, 9, a\}$ and, on it we define the matroidal double chain given by 
\begin{equation*}\uX = (\{2, 3, a\}, \{0, 1, 2, 3, 5, a\}, \{0, 1, 2, 3, 5, 7, 9, a\}, \groundset)\end{equation*} and $\ur = (1, 3, 4, 5)$.
According to \cref{thm:doublechains_are_nestedmatroids}, to this matroidal double chain it corresponds a generalized Catalan matroid.
Indeed, this generalized Catalan matroid is described in \cref{fig:nestedmatroidpath}.
See \Cref{constr:path} for a general construction.
\end{smpl}

\begin{figure}
    \centering
    \begin{tikzpicture}[scale=0.8]
 \draw[step=1cm, thick, gray] (0,0) grid (6,5);
 \draw[orange, ultra thick, , opacity=0.8] (0,0) -- (2,0) -- (2,1) -- (3,1) -- (3,2) -- (4,2) -- (4,4) -- (6,4) -- (6,5);
 \node[anchor= south ] at (0.5,0){$8$};
 \node[anchor= south ] at (1.5,0){$4$};
 \node[anchor=  west] at (2,0.5){$6$};
 \node[anchor=  south] at (2.5,1){$7$};
 \node[anchor=  west] at (3,1.5){$9$};
 \node[anchor=  south] at (3.5,2){$1$};
 \node[anchor=  west] at (4,2.5){$5$};
 \node[anchor=  west] at (4,3.5){$0$};
 \node[anchor=  south] at (4.5,4){$a$};
 \node[anchor=  south] at (5.5,4){$2$};
 \node[anchor=  west] at (6,4.5){$3$};
\end{tikzpicture}

    \caption{A generalized Catalan matroid, corresponding to a path in the $6\times 5$ grid with its labels. }
    \label{fig:nestedmatroidpath}
\end{figure}

We note what it means for a  nested matroids to be loopless in the language of the aforementioned different definitions.

\begin{obs}\label{prop:looplessnestedmatroid}
A nested matroid is called \Def{loopless} (or \Def{loopfree}) if one of the following equivalent conditions hold:
\begin{enumerate}[(i)]
 \item The matroid does not contain loops, \emph{i.e.} elements that are not contained in any basis.
 \item The only cyclic flat of $\matroid$ with rank $0$ is the empty set.
 \item The matroid is a PI-extension from $U_1^1$, the uniform matroid on one element with rank one.
 \item The matroid is a generalized Catalan matroid that corresponds to a path $\Path$ in the grid whose last step is a  north step.
 \item The matroid $\matroid = \matroid(\uX, \ur)$ is a double-chain matroid that has $r_1>0$.
\end{enumerate}
\end{obs}

We enumerate all loopless nested matroids for $d = 2$ and $d = 3$ in \cref{tab:loopless_SM_d2d3}.

\subsection{Valuative functions}
In this section we will be dealing with valuative functions on polytopes.
Over time and throughout the literature, different notions of \emph{weakly valuative function} arose (see \cite{groemer1978extension,derksen2010valuative,
billera2009quasisymmetric,ardila2023valuations,
eur2023stellahedral}).
On the other hand, the definition of \emph{strongly valuative} is unambiguous throughout.
In \cite[Appendix]{eur2023stellahedral} many notions of weakly valuative are surveyed.

For a polyhedron $\Pol\subseteq\R^d$, its \Def{indicator function} is a function $\onebb[\Pol] : \R^\nd \to \{0, 1\}$ given by
\begin{equation*}
 \bbone[\Pol](\vx)\coloneqq\begin{cases}
                    1&\text{ if } \vx\in\Pol\\
                    0&\text{ if } \vx\notin\Pol
                   \end{cases}\,.
\end{equation*}
Given a set of polyhedra $\setofPol$ in $\R^\nd$, the \Def{indicator $\Z$-module of polyhedra}  $\indicatorgroup(\setofPol)$ is the $\Z$-module of functions $\R^\nd \to \Z$ additively generated by all indicator functions $\onebb[\Pol]$ of polytopes $\Pol\in\setofPol$.

Note that in the indicator $\Z$-module of polyhedra of $\setofPol$, additive equivalents of inclusion exclusion relations hold.
For instance, if $\Pol, \Qol$ are two polyhedra in $\setofPol$ such that $\Pol\cap \Qol $ and $\Pol\cup \Qol$ are both polyhedra in $\setofPol$, then
\begin{equation}\label{eq:simple_valuative}
\onebb[\Pol] + \onebb[\Qol] = \onebb[{\Pol\cap \Qol}] + \onebb[{\Pol \cup \Qol}]\, .
\end{equation}
This should not be confused with the McMullen algebra of polytopes, introduced by {McMullen} in \cite{mcmullen1989polytope}, where this additive relation also holds, but a translation invariance also holds.

Let $\setofPol$ be a collection of polyhedra.
We say that $\subdiv=\{\Pol_1, \ldots , \Pol_m\}\subseteq \setofPol$ is a
\Def{polyhedral subdivision} of a polyhedron $\Pol\in \setofPol$ into polyhedra in $\setofPol$ if $\cup_{i=1}^m \Pol_i = \Pol$, and for every $1 \leq i \leq j \leq m$ the intersection $\Pol_i\cap \Pol_j\in\setofPol$ is  a face of both $\Pol_i, \Pol_j$.
Some definitions require polyhedral subdivisions to be closed under taking faces, i.e., if $\Pol_i\in\subdiv$ and $\face$ is a face of $\Pol_i$ then $\face\in\subdiv$.
Both subdivisions depicted in \cref{fig:subdivision} are not closed under taking faces.
\begin{rem}
The definition presented in \cite[Definition 3.1]{derksen2010valuative} for a polyhedral subdivision is not strong enough to make the definition of weakly valuative functions coherent.
Specifically, from the definition therein the intersection of two polytopes $\Pol_i\cap \Pol_j$ is only required to be \emph{contained} in a proper face of $\Pol_i, \Pol_j$.
Adjusting the definition to the one above does not bring material changes to the proof thereafter, according to \cite{fink_personalcomunication}.
\end{rem}
Let $\setofPol$ be a collection of polytopes, and $A$ an additive group.
A map $\varphi:\setofPol \to A$ is \Def{weakly valuative} if
for every  polyhedral subdivision $\{\Pol_1, \ldots , \Pol_m\}$ of a polytope $\Pol\in\setofPol$ into polytopes in $\setofPol$, the so called \Def{valuative relation} holds, that is we have
\begin{equation}
\label{eq:valdef}
\sum_{I\subseteq [m] } (-1)^{|I|} \varphi(\Pol_I) = 0 \, ,
\end{equation}
where $\Pol_I \coloneqq \cap_{i \in I} \Pol_i$, and we use the convention that $\Pol_{\emptyset } = \Pol$.

Notice that, a priori, the relations required in Equation~\eqref{eq:valdef} for weakly valuative functions depend on the definition of polyhedral subdivision, since different polyhedral subdivisions impose different relations as we see in  \Cref{smpl:subdiv}.
However, \Cref{rem:subdivweakly} explains how this issues is solved.

\begin{smpl}\label{smpl:subdiv}
In \cref{fig:subdivision} we present two subdivisions of the triangle $T$.
A valuative function $\varphi$ would have to satisfy \eqref{eq:valdef} for each polyhedral subdivisions.
For the subdivision in \cref{fig:subdivision1} we get:
\begin{equation*}
0 = \varphi(\Pol) - \varphi(\mathsf{T}_1) - \varphi(\mathsf{T}_2) - \varphi(\mathsf{T}_3) + \varphi(\underbrace{\mathsf{T}_1\cap \mathsf{T}_2}_{=\mathsf{S}_3}) + \varphi(\underbrace{\mathsf{T}_1\cap \mathsf{T}_3}_{=\mathsf{S}_2}) + \varphi(\underbrace{\mathsf{T}_2\cap \mathsf{T}_3}_{=\mathsf{S}_1}) - \varphi(\underbrace{\mathsf{T}_1\cap \mathsf{T}_2\cap \mathsf{T}_3}_{=\mathsf{P}_0})
\end{equation*}
For the subdivision in \cref{fig:subdivision2} we get:
\begin{equation*}
\begin{split}
0 &= \varphi(\Pol) - \varphi(\mathsf{T}_1) - \varphi(\mathsf{T}_2) - \varphi(\mathsf{T}_3) - \varphi(\mathsf{S}_2)\\
& \qquad + \varphi(\underbrace{\mathsf{T}_1\cap \mathsf{T}_2}_{=\mathsf{S}_3}) + \varphi(\underbrace{\mathsf{T}_1\cap \mathsf{T}_3}_{=\mathsf{S}_2}) + \varphi(\underbrace{\mathsf{T}_2\cap \mathsf{T}_3}_{=\mathsf{S}_1})
+ \varphi(\underbrace{\mathsf{S}_2\cap \mathsf{T}_1}_{=\mathsf{S}_2})+ \varphi(\underbrace{\mathsf{S}_2\cap \mathsf{T}_2}_{=\mathsf{P}_0}) + \varphi(\underbrace{\mathsf{S}_2\cap \mathsf{T}_3}_{=\mathsf{S}_2}) \\
&\qquad\qquad - \varphi(\underbrace{\mathsf{T}_1\cap \mathsf{T}_2\cap \mathsf{T}_3}_{=\mathsf{P}_0}) - \varphi(\underbrace{\mathsf{S}_2\cap \mathsf{T}_2\cap \mathsf{T}_3}_{=\mathsf{P}_0}) - \varphi(\underbrace{\mathsf{T}_1\cap \mathsf{S}_2\cap \mathsf{T}_3}_{=\mathsf{S}_2}) - \varphi(\underbrace{\mathsf{T}_1\cap \mathsf{T}_2\cap \mathsf{S}_2}_{=\mathsf{P}_0})\\
& \qquad\qquad\qquad + \varphi(\underbrace{\mathsf{T}_1\cap \mathsf{T}_2\cap \mathsf{T}_3\cap \mathsf{S}_2}_{=\mathsf{P}_0})\\
&= \varphi(\Pol) - \varphi(\mathsf{T}_1) - \varphi(\mathsf{T}_2) - \varphi(\mathsf{T}_3) - \varphi(\mathsf{S}_2)\\
&\qquad+ \varphi(\mathsf{S}_3) + \varphi(\mathsf{S}_2) + \varphi(\mathsf{S}_1) + \varphi(\mathsf{S}_2) + \varphi(\mathsf{P}_0) + \varphi(\mathsf{S}_2)\\
& \qquad\qquad- \varphi(\mathsf{P}_0) - \varphi(\mathsf{P}_0) - \varphi(\mathsf{S}_2) - \varphi(\mathsf{P}_0) + \varphi(\mathsf{P}_0)\\
&= \varphi(\Pol) - \varphi(\mathsf{T}_1) - \varphi(\mathsf{T}_2) - \varphi(\mathsf{T}_3) + \varphi(\mathsf{S}_1) + \varphi(\mathsf{S}_2) + \varphi(\mathsf{S}_3) - \varphi(\mathsf{P}_0)\, .
\end{split}
\end{equation*}
We notice that we end up with the same relation after canceling terms with opposite signs in either case.

\begin{figure}
\centering
\begin{subfigure}{.49\textwidth}
\centering
 \begin{tikzpicture}
 \draw[very thick] (0,0) -- (6,0) -- (3,5) -- (0,0);
 \fill[orange!20] (0,0) -- (6,0) -- (3,5) -- (0,0);
 \draw[very thick, orange] (0,0) -- (3,2);
 \draw[very thick] (3,2) -- (6,0);
 \draw[very thick] (3,5) -- (3,2);
 \node[anchor= south] at (3,0.3) {$\textcolor{orange}{\mathsf{T}_1}$};
 \node at (4,2.3) {$\textcolor{orange}{\mathsf{T}_2}$};
 \node at (2.1,2.3) {$\textcolor{orange}{\mathsf{T}_3}$};
 \node at (4.5,1.4) {$\mathsf{S}_1$};
 \node at (1.5,1.4) {$\textcolor{orange}{\mathsf{S}_2}$};
 \node[anchor = west] at (3,3.5) {$\mathsf{S}_3$};
 \node[anchor= north] at (3,1.95) {$\mathsf{P}_0$};
 \node[anchor=south] at (3,5) {$\mathsf{P}_1$};
 \node[anchor= north east] at (0,0) {$\mathsf{P}_2$};
 \node[anchor= north west] at (6,0) {$\mathsf{P}_3$};
\end{tikzpicture}
\caption{$\{\mathsf{T}_1,\mathsf{T}_2,\mathsf{T}_3, \mathsf{S}_2\}$}\label{fig:subdivision2}
\end{subfigure}%
\begin{subfigure}{.49\textwidth}
\centering
\begin{tikzpicture}
 \draw[very thick] (0,0) -- (6,0) -- (3,5) -- (0,0);
  \fill[orange!20] (0,0) -- (6,0) -- (3,5) -- (0,0);
 \draw[very thick] (0,0) -- (3,2) -- (6,0);
 \draw[very thick] (3,5) -- (3,2);
 \node[anchor= south] at (3,0.3) {$\textcolor{orange}{\mathsf{T}_1}$};
 \node at (4,2.3) {$\textcolor{orange}{\mathsf{T}_2}$};
 \node at (2.1,2.3) {$\textcolor{orange}{\mathsf{T}_3}$};
 \node at (4.5,1.4) {$\mathsf{S}_1$};
 \node at (1.5,1.4) {$\mathsf{S}_2$};
 \node[anchor = west] at (3,3.5) {$\mathsf{S}_3$};
 \node[anchor= north] at (3,1.95) {$\mathsf{P}_0$};
 \node[anchor=south] at (3,5) {$\mathsf{P}_1$};
 \node[anchor= north east] at (0,0) {$\mathsf{P}_2$};
 \node[anchor= north west] at (6,0) {$\mathsf{P}_3$};
\end{tikzpicture}
\caption{ $\{\mathsf{T}_1,\mathsf{T}_2,\mathsf{T}_3\}$}\label{fig:subdivision1}
\end{subfigure}
\caption{Two distinct subdivisions of the  triangle $\conv(\mathsf{P}_1,\mathsf{P}_2,\mathsf{P}_3)$.
\label{fig:subdivision}}
\end{figure}
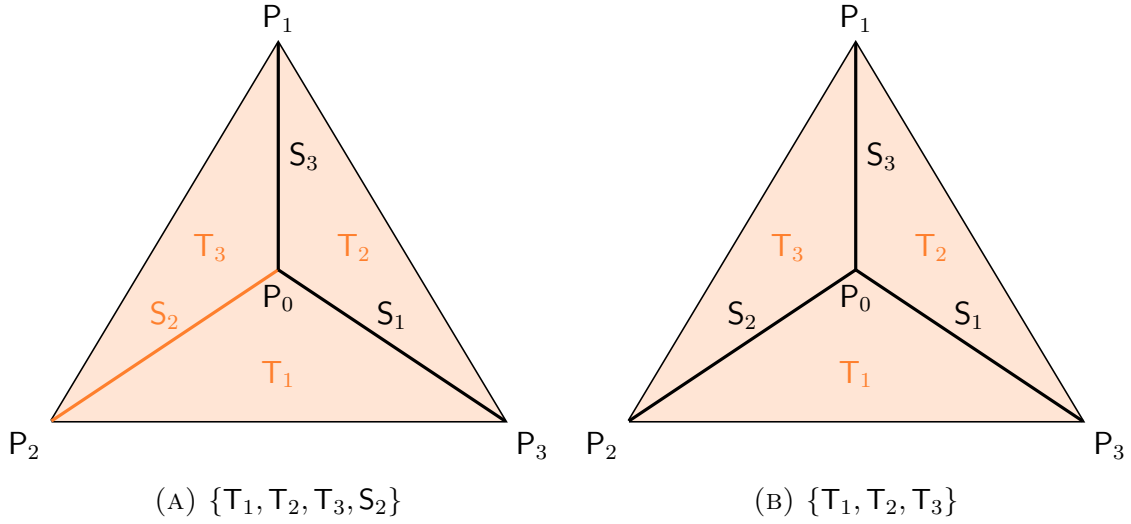

\end{smpl}

\begin{rem}\label{rem:subdivweakly}
As the previous example made clear, the definition of weakly valuative function entails a linear relation that relies on the definition of polyhedral subdivisions.
Even though these subdivisions look very different, these lead to equivalent linear relations, after a few cancellations.
To make things simple, a cancellation-free formula is presented in \cite[Theorem 3.5.]{ardila2010valuations} whenever $\setofPol$ is the set of all matroid base polytopes:
Consider a polyhedral subdivision $\subdiv = \{\Pol_1, \ldots , \Pol_m\}$ of a polytope $\Pol$.
We say that a polytope $\Qol$ is an \Def{internal face} of $\subdiv$ if it is a face of some $\Pol_i$ and it contains an interior point of $\Pol$.
The valuative relation in \eqref{eq:valdef} is equivalent to:
\begin{equation*}
\varphi(\Pol) = \sum_{\Qol} \varphi(\Qol) (-1)^{\dim \Qol - \dim \Pol} \, ,
\end{equation*}
where the sum runs over all internal faces $\Qol$ of $\subdiv$.
This works for every intersection-closed collection of polytopes, according to \cite{ardila_personalcomunication}.
\end{rem}

A function $\varphi:\setofPol\to A$ is called a \Def{strongly valuative function} if there exists an additive map $\widehat{\varphi}: \indicatorgroup(\setofPol) \to A$ such that $\varphi(\Pol) = \widehat{\varphi}(\onebb[\Pol])$ for every polytope $\Pol \in \setofPol$.
\begin{equation*}
  \begin{tikzcd}
    \setofPol \arrow[rr, "\varphi"] \arrow[dr, "\onebb"']&&A\\
    &\indicatorgroup(\setofPol)\arrow[ur, "\exists \widehat{\varphi}"']&
\end{tikzcd}
\end{equation*}

Strongly valuative maps are always weakly valuative, the converse is not true in general.
However, in the setting of this paper those two notions are equivalent.

\begin{thm}[{\cite[Corollary 3.9.]{derksen2010valuative}}]
A weakly valuative function on the set $\mathcal P$ of matroid base polytopes is a strongly valuative function.
\end{thm}
From now on, we will call weakly or strongly valuative functions on matroids simply \Def{valuative}.
We will also say the map is a \Def{valuation}.

\subsection{Chromatic word-quasisymmetric functions of matroids}\label{ssec:chrmwqsym}

In this section we work over an infinite collection of non-commuting variables $\bw = \bw_1, \bw_2, \ldots $.
For every $f: [\nd] \to \Z_{> 0} $, we define the monomial $\bw_{f} = \bw_{f(1)} \bw_{f(2)} \cdots \bw_{f(\nd)}$.
A \Def{formal power series} of degree $d$ on the non-commuting variables $\bw $ is a formal sum $F(\bw) = \sum_{f:[\nd] \to \Z_{\geq1}} c_{f} \bw_{f}$.

A \Def{word-quasisymmetric function} of degree $\nd$ is a formal sum $\sum c_{f} \bw_{f}$ such that for two functions $f, g\colon [\nd]\to \R$ with the same set composition type, that is $\opi(f) = \opi(g)$ are the same set composition of $[\nd]$, the coefficients $c_f, c_g$ of $\bw_{f}, \bw_{g} $ are the same.
We write $\WQSym_\nd $ for the vector space of word-quasisymmetric functions of degree $\nd$, and we let $\WQSym = \oplus_{\nd\geq 0} \WQSym_\nd$ be the space of finite sums of homogeneous word-quasisymmetric functions, which we simply call \Def{word-quasisymmetric functions}.
This structure was defined in \cite{bergeron2009hopf}.

The algebra of word-quasisymmetric functions should not be confused with the algebra of \emph{non-commutative symmetric functions} (see for instance \cite{hicks2024combinatorial}).
The latter one arises as the non-commutative version of symmetric functions, arising as the dual space of the quasisymmetric functions.

The \Def{monomial basis} $\{ \Mnco_{\otau} \}_{\substack{\otau \models [\nd ]\\ \nd \geq 0}}$ is a basis of the algebra of word-quasisymmetric functions, is indexed by set compositions $\otau$ and given by
\begin{equation*}
\Mnco_{\otau} \coloneqq \sum_{\substack{f \colon [d]\to \Z_{\geq1}\\ \opi(f)=\otau}} \bw_f \, ,
\end{equation*}
where the sum runs over all functions $f\colon[d]\to \Z_{\geq1}$ such that their set composition type  $\opi(f)$ equals $\otau$.

For a matroid $\matroid=([\nd],\bases)$ we define the \Def{chromatic word-quasisymmetric function of $\matroid$} by
\begin{equation}\label{eq:matroidbaseswordchromaticfunction}
 \ncPsi(\matroid) \coloneqq \sum_{\substack{f\colon[\nd]\to \Z_{\geq1}\\ \text{ is $\matroid$-generic}}} \bw_f\,.
\end{equation}
this defines a map $\ncPsi : \MatLinSpace \to \WQSym$, where $\MatLinSpace$ denotes the (graded) vector space of matroids (see Equation~\eqref{eq:defmatspace}).
On the other hand, a function $f\colon[\nd]\to \Z_{\geq1}$ is $\matroid$-generic if and only if the type $\opi(f)$ of $f$ is $\matroid$-generic.
Therefore
\begin{equation*}
\ncPsi(\matroid) = \sum_{\substack{\substack{\opi\models[d] \\ \text{is $\matroid$-generic}}}} \Mnco_{\opi}\,\end{equation*}
and this is indeed a word-quasisymmetric function.
Recall that every set composition with only singleton blocks is generic for every matroid.
Hence, those monomials will be part of every chromatic word-quasisymmetric function of a matroid.
Note that a function $f\colon[\nd]\to \Z$ is $\matroid$-generic precisely when the integer point $(f(1),\dots,f(d))\in\Z^d$ is contained in the open normal cone of some vertex of the matroid base polytope.

\begin{smpl}\label{ex:wordquasisym_new}
Consider the matroid on the ground set $[3]$  of rank one with three bases, that is $\unifMat{1}{3}=([3], \{\{1\},\{2\},\{3\}\})$
Let us compute the chromatic word-quasisymmetric function of this matroid in terms of the monomial basis, that is, using \eqref{eq:matroidbaseswordchromaticfunction}, by going through each of the thirteen set compositions of $[3]$.

 As observed earlier,  every set composition with three parts is $\matroid$-generic and the set composition with just one block is never $\matroid$-generic.
It is left to check the remaining six set compositions with two parts:
 $1|23$, $2|13$, and $3|12$ are not $\unifMat{1}{3}$-generic; the other three are $\unifMat{1}{3}$-generic.
Then,
 \begin{equation}\label{eq:matroidbaseswordchromaticfunctionUniform}
\begin{split}
 \ncPsi(\matroid)= &\ \Mnco_{12|3}+\Mnco_{13|2}+\Mnco_{23|1}\\
 &\quad+\Mnco_{1|2|3}+\Mnco_{1|3|2}+\Mnco_{3|1|2}+\Mnco_{3|2|1}+\Mnco_{2|3|1}+\Mnco_{2|1|3}\,.
\end{split}
\end{equation}
 \begin{figure}
  \begin{subfigure}{.49\linewidth}
    {\includegraphics[trim=230 80 400 40, clip, width=\linewidth]{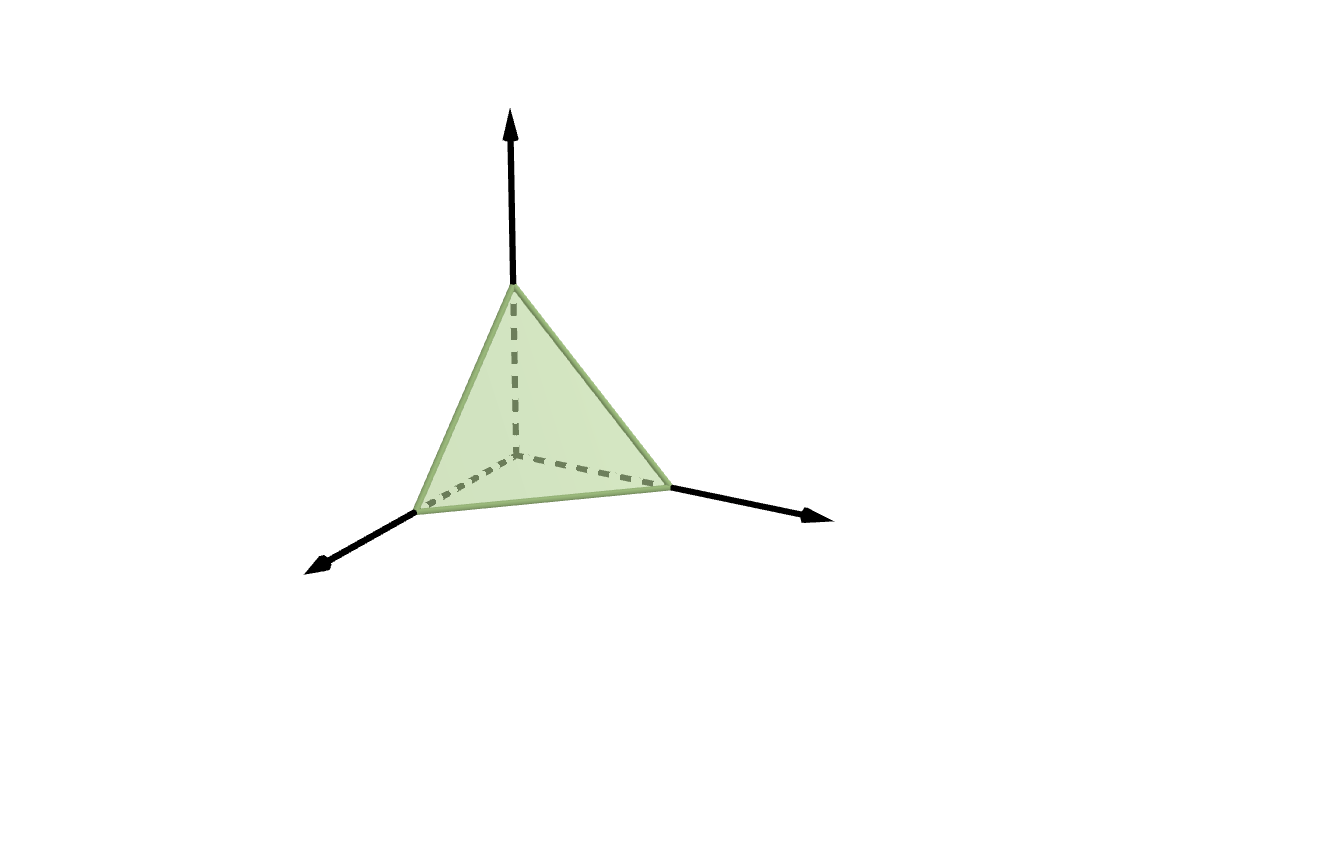}}
    \caption{Matroid base polytope in \cref{ex:wordquasisym_new} is a triangle in $\R^3$.}
  \end{subfigure}
  \begin{subfigure}{.49\linewidth}
  {\includegraphics[trim=230 80 400 40, clip, width=\linewidth]{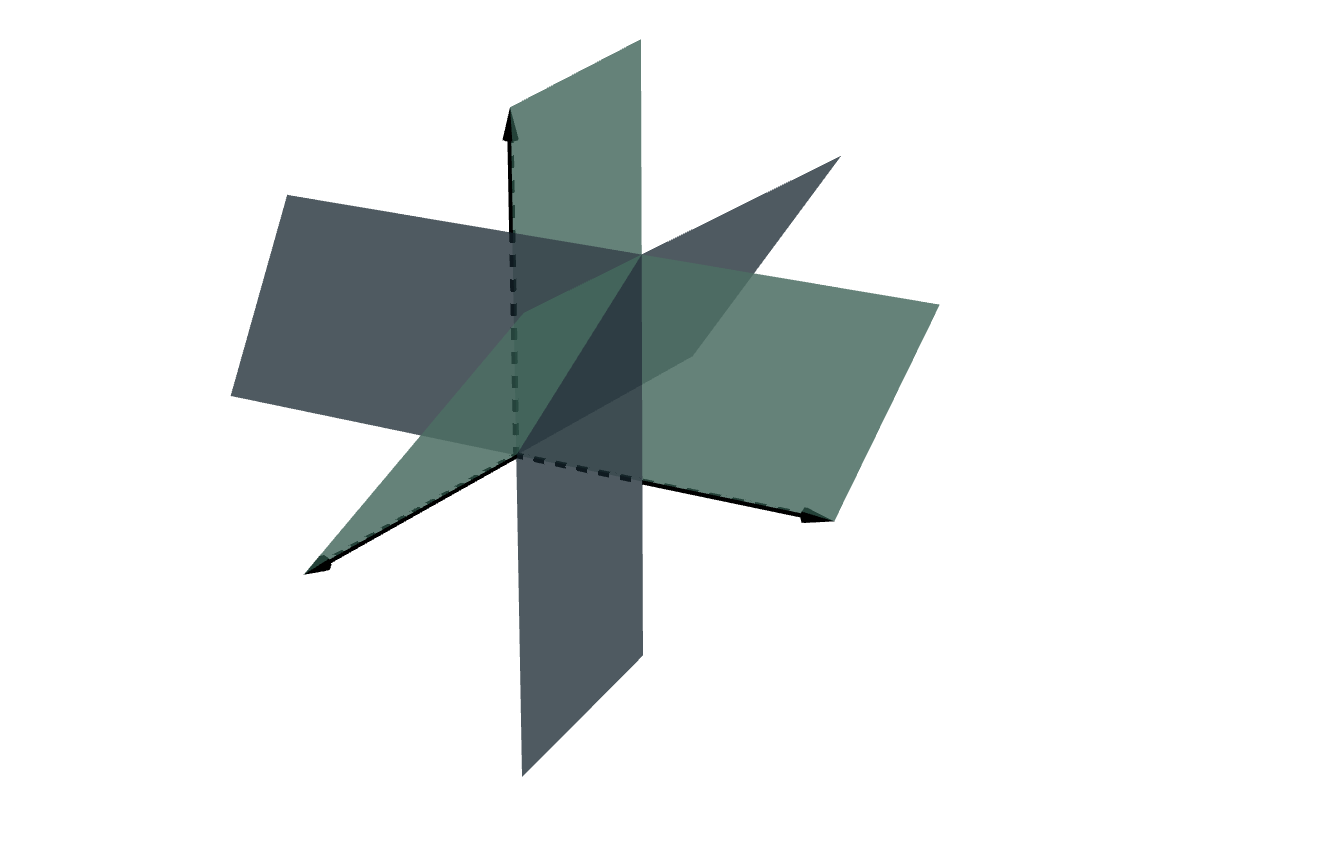}}
  \caption{Normal fan in dark blue, coarsens the braid fan (missing parts in petrol).}
  \label{fig:matroidpolytopeandbraidarrangement_new_fan}
  \end{subfigure}
    \caption{ The matroid base polytope $\Pol_{\unifMat{1}{3}}$, its normal fan, and the braid arrangement. Note that every cone of the braid arrangement that intersects the interior of a full dimensional cone of the normal fan of $\Pol_{\matroid}$ contributes to the sum in \eqref{eq:matroidbaseswordchromaticfunctionUniform}. }\label{fig:matroidpolytopeandbraidarrangement_new}
 \end{figure}

This can also be seen geometrically.
Recall that a monomial basis element $\Mnco_{\opi}$ corresponds to a cone $\BraidC_{\opi}$ of the braid arrangement (\Cref{lem:braid}).
Consider the matroid base polytope of $\unifMat{1}{3}$, as well as its normal fan $\mathcal N$, pictured in \cref{fig:matroidpolytopeandbraidarrangement_new}, juxtaposed with the normal fan (dark blue), coarsening the braid fan (dark blue and petrol).
Recall from \cref{rem:mapsaspoints2} that a set composition $\opi$ contributes to the sum $\ncPsi(\unifMat{1}{3})$ if and only if the corresponding open braid cone $\BraidC_{\opi}^\circ$ is contained in the interior of  the normal cone $\normalcone_{\Pol_{\unifMat{1}{3}}}^\circ(\vertex)$ of a vertex $\vertex=\onebb_B$.
The petrol cones in \cref{fig:matroidpolytopeandbraidarrangement_new_fan} correspond to the generic set compositions with two blocks.
\end{smpl}

\begin{prop}
For every matroid $\matroid$ on a set of size $\nd$, the power series $\ncPsi(\matroid)$ is a word-quasisymmetric function of degree $\nd$.
\end{prop}
\begin{proof}
This is equivalent to the claim that the matroid base polytopes are generalized permutahedra,
so their normal fans coarsen the braid fan. 
The cones in the braid fans are defined by set compositions. 
This implies $\matroid$-genericity of a function $f$ only depends on its set composition type.
\end{proof}

By letting the variables $\bw_1, \bw_2, \ldots $ commute on a word-quasisymmetric function, we get a quasisymmetric function (see \Cref{sec:qsym}).
Indeed, this is a linear map and the image of a monomial basis element $\Mnco_{\otau }$ is $\Mco_{\alpha(\otau)}$, which is a quasi-symmetric function, so by linearity we can define the map $\mathrm{comu}: \WQSym \to \QSym$.
It turns out that the chromatic word-quasisymmetric functions defined above are compatible with the (commutative) quasisymmetric functions of matroids.

\begin{prop}
\begin{equation*} \mathrm{comu} \circ \ncPsi = \cPsi \, .\end{equation*}
\end{prop}

\begin{proof}
    For every matroid $\matroid$ we have
    \begin{equation*}
    \mathrm{comu} \circ \ncPsi(\matroid) =
    \sum_{\substack{\substack{\opi\models[d] \\ \text{is $\matroid$-generic}}}} \mathrm{comu}( \Mnco_{\opi}) = 
    \sum_{\substack{\substack{\opi\models[d] \\ \text{is $\matroid$-generic}}}} \Mco_{\alpha(\opi)} =  \cPsi(\matroid)
    \end{equation*}
\end{proof}

\section{Chromatic word-quasisymmetric functions of matroids are weakly valuative}\label{sec:weaklyvaluative}

In this section we will prove the following result:

\begin{thm}\label{thm:wqsym_are_val}
The map $\ncPsi\colon\MatLinSpace\to\WQSym$ assigning every matroid its chromatic word-quasisymmetric function is weakly valuative.
\end{thm}

For quasisymmetric functions of matroids this was already proven in \cite[Theorem 7.4]{billera2009quasisymmetric}.
For sake of completeness we include the proof for its non-commutative counterpart $\ncPsi$, which follows the same steps as the commutative case.

We recall a variation on a classical result in convex geometry from \cite{lawrence1988valuations}.
First, we define for a closed convex set $A \subseteq \R^d$, its open polar cone
\begin{equation*}
 \Done^\circ(A) \coloneqq \{ y \in \R^d | \forall_{x \in A} \langle x, y\rangle < 0  \}\, .
\end{equation*}

\begin{prop}[{\cite[Proposition 7.2]{billera2009quasisymmetric}}, {\cite[Corollary IV.1.6]{barvinok_course_2012}} and references therein]\label{prop:magic_dual}
Let $A_1, \ldots , A_k$ be closed convex sets such that 
\begin{equation*}
 \sum_i \alpha_i\onebb[A_i] = 0
\end{equation*}
for some coefficients $\alpha_i\in\R$.
Then we have
\begin{equation*}
 \sum_i \alpha_i\onebb[\Done^\circ(A_i)] = 0\, .
\end{equation*}
\end{prop}

For a set $B\subseteq [d]$ let $\onebb_B \in \R^d$ be the \Def{indicator vector} of $B$.
If $B$ is a base of the matroid $\matroid$, let $\Kone_B(\matroid)$ be the \Def{vertex cone} of $\Pol_\matroid$ corresponding to the vertex $\onebb_B$.
That is,
\begin{equation*}
 \Kone_B(\matroid)=\cone \{\vx - \onebb_B | \vx\in \Pol_\matroid \}\,.
\end{equation*}
It follows from $\Pol_\matroid = \cup_i \Pol_{\matroid_i}$ that 
\begin{equation*}
 \Kone_B(\matroid) =  \bigcup_{i : B \in {\matroid_i}} \Kone_B(\matroid_i)
\end{equation*}
for matroid base polytope subdivisions (see, e.g., \cite[Lemma 5.2.2]{Beck2018}).

In the following, for a collection of matroids $\{ \matroid_i \}_i $, we write $ \bigcap_i \matroid_{i}$ for the collection of bases $\{B | B \in \matroid_i \text{ for all $i$ }\} $.
Note that if this collection of matroids is part of a matroid base polytope subdivision, $ \bigcap_i \matroid_{i}$ is either empty or defines a matroid, according to the definition of polytopal subdivision (since faces of matroid base polytopes are matroid base polytopes again).

\begin{lm}
 For a matroid base polytope subdivision $\Pol_\matroid = \bigcup_i \Pol_{\matroid_i}$ and a fixed basis $B\in \matroid$ we have
\begin{equation*}
 \bigcap_{i : B \in {\matroid_i}} \Kone_B(\matroid_{i})  = \Kone_B\left(\bigcap_{i : B \in {\matroid_i}} \matroid_{i} \right)\,.
\end{equation*}
\end{lm}
\begin{proof}
Let 
\begin{equation*}
    \neighborhood_\matroid(B)\coloneqq\{C \colon C\text{ a bases of }\matroid\text{ such that } \lvert B\setminus C\rvert=1\}
\end{equation*}
denote the \Def{neighborhood} of basis $B$ in matroid $\matroid$.
Equivalently, bases $C\in \neighborhood_\matroid(B)$ correspond to all the vertices $\onebb_C$ in $\Pol_\matroid$ that share an edge with $\onebb_B$.

For a matroid base polytope $\Pol_\matroid$ and a basis $B$ we have
\begin{equation*}
 \Kone_B(M) = \cone\{ \onebb_C - \onebb_B | \text{ $C\in \neighborhood_\matroid(B)$}\} \, .
\end{equation*}
It is then a simple observation that 
\begin{equation*}
\Kone_B(M_i) = \cone\left\{ \onebb_C - \onebb_B \Big| \substack{C \in \neighborhood_\matroid(B) \\ \text{ is basis of $M_i$ }}\right\} \, ,
\end{equation*}
if $B$ is a basis in $\matroid_i$ and
\begin{equation*}
\Kone_B\left(\bigcap_{i : B \in {\matroid_i}} \matroid_{i}\right) = \cone\left\{ \onebb_C - \onebb_B \Big|\substack{ C \in N_\matroid(B) \\ \text{$C$ is basis of every $M_i$ that has $B$ as basis }}\right\} \, .
\end{equation*}
From these two equations the lemma follows.
\end{proof}

Note that $\Done^\circ(\Kone_B(\matroid))=(\normalcone_{\Pol_\matroid}^\circ(\onebb_B))$.
Indeed:
\begin{align*}
    \Done^\circ(\Kone_B(\matroid))&=\{y\in\R^d\colon \langle \onebb_C - \onebb_B, y\rangle<0 \text{ for all }C\in \neighborhood_\matroid(B)\}\\
    &=\{y\in\R^d\colon \langle \onebb_C, y\rangle <\langle \onebb_B, y\rangle \text{ for all }C\in \neighborhood_\matroid(B)\}\\
    &=\{y\in\R^d\colon \langle \onebb_C, y\rangle <\langle \onebb_B, y\rangle \text{ for all bases }C\in \mathcal B (\matroid )\setminus \{ B \}\}\\
    &=\normalcone_{\Pol_\matroid}^\circ(\onebb_B)
\end{align*}

With this, we are ready to prove the main theorem of this section.

\begin{proof}[Proof of \cref{thm:wqsym_are_val}]
Let $\matroid$ be a matroid.
Recall that for every function $f:[d] \to \Z_{\geq 1}$ such that $\opi(f) $ is $\matroid$-generic, there is a unique basis $B$ that maximizes the weight of $f$.
This is equivalent to $\vy_f \in \normalcone_{\Pol_\matroid}^\circ(\onebb_B)=\Done^\circ(\Kone_B(\matroid))$.
Therefore, we can write
\begin{equation}
\label{eq:split_ncpsi}
\begin{split}
\ncPsi(\matroid ) = \sum_{\substack{f\colon[\nd]\to \Z_{\geq1}\\ \text{ is $\matroid$-generic}}} \bw_f 
 &= \sum_{B \in \mathcal B (\matroid) }  \sum_{\substack{f\colon[\nd]\to \Z_{\geq1}\\ \text{ has unique maximum} \\ \text{in $\matroid$ at $B$}}} \bw_f \\
&= \sum_{B \in \mathcal B (\matroid) }  \sum_{\substack{f\colon[\nd]\to \Z_{\geq1}\\ \vy_f\in  \normalcone_{\Pol_\matroid}^\circ(\onebb_B)}} \bw_f = \sum_{B \in \mathcal B (\matroid) } \sum_{\vy_f \in \Done^\circ(\Kone_B(\matroid))} \bw_f \, .
\end{split}
\end{equation}

Now suppose that we have a polytope subdivision $\Pol_\matroid = \bigcup_i \Pol(\matroid_i)$.
This means that an arbitrary intersection of these matroid base polytopes is either empty or a matroid base polytope.
In this context, we can abuse notation and write $\matroid \cap \natroid$ for the matroid arising from the bases that are both in $\bases (\matroid)$ and $\bases (\natroid)$.

This corresponds to a polyhedral subdivision of the vertex cones as follows:
\begin{equation*}
\Kone_B(\matroid) =  \bigcup_{i : B \in {\matroid_i}} \Kone_B(\matroid_i)\, .
\end{equation*}

An inclusion-exclusion relation applied to this equation gives us:
\begin{equation*}\onebb[\Kone_B(\matroid)] =
\sum_j (-1)^{j-1} \sum_{\substack{i_1 < \cdots < i_j \\ B \in {\matroid_i}}} \onebb[\Kone_B(\matroid_{i_1}) \cap \cdots \cap \Kone_B(\matroid_{i_j})]\, .\end{equation*}

From \cref{prop:magic_dual}, we get that 
\begin{equation*}
\begin{split}
\onebb[\Done^\circ(\Kone_B(\matroid))] &=  \sum_{j\geq 1} (-1)^{j-1} \sum_{\substack{i_1 < \cdots < i_j \\ B \in {\matroid_i}}} \onebb\left[\Done^\circ\left(\Kone_B(\matroid_{i_1}) \cap \cdots \cap \Kone_B(\matroid_{i_j})\right)\right]\\
		    &=  \sum_{j\geq 1} (-1)^{j-1} \sum_{\substack{i_1 < \cdots < i_j \\ B \in {\matroid_i}}} \onebb\left[\Done^\circ\left(\Kone_B(\matroid_{i_1} \cap \cdots \cap \matroid_{i_j})\right)\right]
\end{split}
\end{equation*}

Summing up this equation for each base $B$, and using \cref{eq:split_ncpsi}, we get
\begin{equation*}
\ncPsi(\matroid) = \sum_{j\geq 1 } (-1)^{j-1} \sum_{\substack{i_1 < \cdots < i_j \\ B \in \Pol_{\matroid_i}}} \ncPsi(\matroid_{i_1} \cap \cdots \cap \matroid_{i_j})\, ,
\end{equation*}
as desired.
\end{proof}

\section{Rank computation}\label{sec:rank}

Recall our main conjecture:

\begin{conj:kernel}
The kernel of the map $\ncPsi$  from matroids to chromatic word-quasi\-sym\-metric functions is spanned by the valuative relations (see Equation~\eqref{eq:valdef}) and the loop-coloop relations (see \Cref{prop:loopcolooprelation} below).
In particular, as established in \Cref{thm:kernelcontains}, the function $\ncPsi$ is determined by its values on loopless nested matroids.
The image of $\ncPsi$ is a Hopf algebra whose homogeneous pieces have dimension $\nd!$.
\end{conj:kernel}

In this section we display progress towards this conjecture and split it into three parts: first we show that this collection of relations - valuative relations and loop-coloop relation - are indeed in the kernel of the $\ncPsi $ map, then we show that the map $\ncPsi_{\nd}$ has rank at least $2^{\nd } - \nd$, finally we present computational evidence and give refined conjectures.

\subsection{The upper bound}
In this section we establish the upper bound in the \cref{conj:kernelspan} by collecting results from the literature.
Specifically, we show that the image of the map $\ncPsi$ from matroids to word-quasisymmetric functions is generated by its values on loopless nested matroids, and that there are $\nd!$ many such matroids.

We showed in \cref{sec:weaklyvaluative} that the map $\ncPsi$ from matroids to word-quasisymmetric functions is weakly valuative.
Hence, the kernel of $\ncPsi$ contains the valuative relations (see Equation~\eqref{eq:valdef}).
We first show, that the kernel contains further relations of a different type.

\begin{prop}\label{prop:loopcolooprelation}
    Let $\matroid = ([d],\bases)$ be a matroid and let $b$ be a loop of $\matroid$.
    Define $\coloop{\matroid}$ to be the matroid obtained from $\matroid$ by replacing $b$ with a coloop, that is, $\coloop{\matroid}\coloneqq([d],\coloop{\bases})$ where $\coloop{\bases}=\left\{B\cup\{b\} \colon B\in\bases\right\}$.
    Then $\ncPsi(\matroid)=\ncPsi(\coloop{\matroid})$.
    In particular, applying this repeatedly, every matroid $\matroid$ with loops $b_1,\dots,b_j$ satisfies $\ncPsi(\matroid)=\ncPsi(\natroid)$, where $\natroid$ is the loopless matroid obtained by replacing all loops with coloops.
\end{prop}
\begin{proof}
 Since $b$ is a loop in $\matroid$, it is not contained in any basis.
 Therefore the matroid base polytope of $\coloop{\matroid}$ is a translation of that of $\matroid$.
 Recall that the normal fan of a polytope is invariant under translations and that the chromatic word-quasisymmetric function only depends on the normal fan of the matroid base polytope.
\end{proof}
For a nested matroid $\matroid $, we call relations of the form
\begin{equation}\label{eq:defloopcoloop}
 \ncPsi(\matroid)-\ncPsi(\coloop{\matroid})=0
\end{equation}
\Def{loop-coloop relations}.
Note that if $\matroid$ is loopless, the corresponding loop-coloop relation is trivial.

\begin{obs}\label{obs:points}
Schubert matroids $\matroid$ corresponding to sets of the form $\{1,2,\dots,k\}$ have a unique basis,
therefore their corresponding polytopes are just points and thus all $\ncPsi(\matroid ) $ are the same.

This suggests that the relations ``$\SM(\{1,2,\dots,\nd\})-\SM(\{1,2,\dots,k\})$'' are contained in the kernel $\ker(\ncPsi)$, and indeed are particular cases of the loop-coloop relations. 
We call such relations \Def{point relations}.
\end{obs}

\Cref{prop:loopcolooprelation} and \Cref{thm:wqsym_are_val} already imply the following Corollary.
\begin{cor}\label{thm:wqsymkern}
Since $\ncPsi$ is graded, we write $\ncPsi_d$ for its degree-$d$ component.
 The kernel $\ker(\ncPsi_d)$ contains the following relations:
 \begin{enumerate}
  \item \label{thm:val} valuative relations, i.e, relations of the form presented in Equation~\eqref{eq:valdef}.
  \item \label{thm:coloop} loop-coloop relations, i.e, relations of the form presented in Equation~\eqref{eq:defloopcoloop}.
 \end{enumerate}
\end{cor}
We remark that the main conjecture of this paper, \cref{conj:kernelspan}, states that the kernel $\ker(\ncPsi_d)$ is precisely the span of these relations.

We now pivot to finding explicit generators of the image of $\ncPsi_d$.
In the following we use statements from \cite{derksen2010valuative}.
On this reference, the authors use chains of subsets and rank vectors satisfying certain conditions to define megamatroids, polymatroids, and subclasses of matroids via their polyhedra.
In \cref{appendix:DCarenested} we explain the connections to nested matroids and reformulate statements by Derksen and Fink in that language.

Recall, for a polytope $\Pol\subseteq\R^d$, write $\bbone[\Pol]$ for its indicator function $\onebb[\Pol] : \R^\nd \to \{0, 1\}$ which is given by $\onebb[\Pol](\vx) = 1 $ if and only if $\vx\in \Pol$.

Let us denote the group generated by indicator functions on matroid base polytopes with rank $r$ and ground set~$[\nd]$
\begin{equation}\label{eq:defmatpolspace}
\begin{split}
 \PolSpace_{\nd, r}\coloneqq&\indicatorgroup(\{\Pol_\matroid \colon \matroid \text{ on } [\nd],\, \rk(\matroid)=r\}) \\
 =&\spn(\{\bbone[\Pol_\matroid ]\colon \matroid \text{ on } [\nd],\, \rk(\matroid)=r\})
 \end{split}
\end{equation}
using the notation of \cref{sec:preliminaries}.
We also write $\PolSpace_{\nd} = \bigoplus_r \PolSpace_{\nd, r}$ and $\PolSpace = \bigoplus_\nd \PolSpace_{\nd}$.

\begin{thm}[{\cite[Theorem 5.4]{derksen2010valuative}}]\label{thm:DF10_original}
The $\Z$-module $\PolSpace_\nd $ generated by indicator functions on matroid base polytopes with ground set $[\nd]$ is freely generated by
\begin{equation*}
\{ \bbone[\Pol_\natroid] | \text{$\natroid$ is a nested matroid on $[\nd]$} \}\,.
\end{equation*}
\end{thm}

We will use the following weaker version that deals with the objects as vector spaces instead of $\Z$-modules.
The $\Z$-module statement implies the $\R$-vector space statement below by tensoring with $\R$.
For convenience we use the same notation for vector spaces and the underlying additive group.

\begin{thm}[reformulation of {\cite[Theorem 5.4]{derksen2010valuative}}]\label{thm:DF5.4_refomulated}
The real vector space $\PolSpace_\nd $ generated by indicator functions of matroid base polytopes is spanned by
\begin{equation*}
\{\bbone[\Pol_\natroid]\ |\ \text{$\natroid$ is a nested matroid on $[\nd]$} \}\,.
\end{equation*}
\end{thm}

We would like to apply \Cref{thm:DF5.4_refomulated} to our (weakly valuative) map $\ncPsi\colon \Mat \to \WQSym$.
We achieve this by the notion of strong valuativity, using the following general result on valuative properties by Derksen and Fink.

\begin{thm}[{\cite[Corollary~$3.9$]{derksen2010valuative}}]\label{thm:DFweakvaluativeisstrong}
Let $A$ be an abelian group and let $ \Mat (\nd,r)$ denote the free abelian group generated by the set of matroids on the ground set $[d]$ with rank $r$.
Then, a homomorphism $\varphi\colon \Mat (\nd, r) \to A$ is weakly valuative if and only if it is strongly valuative.
\end{thm}

It is worth noting that $\WQSym$ can be seen as an Abelian group by forgetting their Hopf algebra structure.
Recall that we denote $\Mat$ for the Hopf algebra of matroids.

\begin{cor}\label{cor:ncPsistronglyvaluative}
The map $\ncPsi\colon \Mat\to\WQSym$ from matroids to word-quasisymmetric functions is strongly valuative.
\end{cor}

\begin{proof}
 Recall that we showed in \Cref{thm:wqsym_are_val} that the map $\ncPsi\colon\MatLinSpace\to\WQSym$ is weakly valuative.
    Furthermore, recall that the Hopf algebra $\Mat$ as an Abelian group is the direct sum of the graded pieces $\Mat (d,r)$ as defined in \Cref{thm:DFweakvaluativeisstrong}.

    Therefore, a group homomorphism $\ncPsi: \Mat \to \WQSym$ can be decomposed into its pieces $\ncPsi_{\nd, r}: \Mat(\nd, r) \to \WQSym$.
    This is in such a way that $\ncPsi_{\nd, r}(\Pol_\matroid) = \ncPsi(\Pol_\matroid)$ for every matroid $\matroid $ with rank $r$ and ground set $[\nd]$.

    Since $\ncPsi$ is weakly valuative, each piece $\ncPsi_{\nd, r}$ is weakly valuative.
     \Cref{thm:DFweakvaluativeisstrong} implies that each $\ncPsi_{\nd,r}$ is strongly valuative, that is, there exists $\widehat{\ncPsi}_{\nd, r}: \PolSpace_{\nd, r} \to \WQSym $ such that $\ncPsi_{\nd,r}(\matroid) = \widehat{\ncPsi}_{\nd, r} \circ \onebb[\Pol_\matroid]$.

    Now we observe that $\PolSpace$ is indeed the direct sum $\bigoplus_{\nd, r} \PolSpace_{\nd, r}$ (see Equation~\eqref{eq:defmatpolspace}).
    Recall that matroid base polytopes of matroids on different ground sets live in different vector spaces and matroid base polytopes are of codimension one, specifically they are contained in the hyperplane defined by coordinate sum equal to the rank of the matroid.
    Hence, matroid base polytopes of matroids on the same ground set but with different ranks are contained in different hyperplanes.
    It follows that the functions generated by a graded piece $\PolSpace_{d,r}$ only ever take non-zero values for points in the hyperplane with coordinate sum $r$ contained in $\R^d$.
    Therefore, elements contained in different graded pieces will always be independent.

    With this we construct a map $\widehat{\ncPsi}: \PolSpace\to \WQSym $, as the  unique direct sum
    \begin{equation*}
     \widehat{\ncPsi}\coloneqq\oplus_{d,r}\widehat{\ncPsi}_{d,r}\,,
    \end{equation*}
    since $\PolSpace$ is a direct sum $\bigoplus_{\nd, r} \PolSpace_{\nd, r}$.
    This function satisfies
    \begin{equation}\label{eq:strongvaluationncPsi}
     \ncPsi(\matroid) = \widehat{\ncPsi} \circ \onebb[\Pol_\matroid]
    \end{equation}
    for every matroid, as desired.
\end{proof}

The factorization of the map $\widehat{\ncPsi}: \PolSpace \to \WQSym$ in the proof of \cref{cor:ncPsistronglyvaluative}  satisfying $ \ncPsi(\matroid ) = \widehat{\ncPsi}(\onebb{[{\Pol_\matroid}]})$ for every matroid $\matroid $
can be  summarized in the following commutative diagram:
\begin{equation*}
  \begin{tikzcd}
   \Mat \arrow[rr, "\ncPsi"] \arrow[dr, "\onebb{[\Pol_{(\cdot)}]}"'] & & \WQSym \\
    & \PolSpace \arrow[ur, "\widehat{\ncPsi}"'] &
\end{tikzcd}\,.
\end{equation*}

\begin{cor}\label{cor:nestedmatroidsgeneratetheimage}
The $d^{\text{th}}$ graded piece of the image of the map $ \ncPsi\colon \MatLinSpace\to\WQSym$ is generated by
\begin{equation*}
(\im \ncPsi)_d = \{\ncPsi(\natroid) \ |\ \text{$\natroid$ is a loopless nested matroid on }[d]\} \subseteq\WQSym_d\, .
\end{equation*}
Moreover, we have that $\dim((\im\ncPsi)_d)\leq d!$.
\end{cor}
\begin{proof}
Indeed, from \Cref{thm:DF5.4_refomulated}, for every matroid $\matroid $ of rank $r$ on ground set $[\nd]$ there is a decomposition in $\PolSpace_{\nd,r}$ into nested matroids, where $\alpha_j $ are coefficients in $\R$:
\begin{equation*}\onebb[\Pol_\matroid] = \sum_j \alpha_j \onebb[\Pol_{\natroid_j}] \, ,\end{equation*}
for nested matroids $\natroid_j$.
Therefore, applying $\widehat{\ncPsi}$ on both sides and using Equation~\eqref{eq:strongvaluationncPsi} we have
\begin{equation*}
\ncPsi(\matroid ) = \sum_j \alpha_j \ncPsi (\natroid_j) \ \in\ \spn \{\ncPsi(\natroid) | \text{$\natroid$ is a nested matroid}\}\, .
\end{equation*}
By \Cref{prop:loopcolooprelation}, nested matroids that only differ by transforming loops into coloops have the same chromatic word-quasisymmetric function.

We use \cite[Theorem 4.5.]{hampe_intersection_2017}, where it was shown that the number of loopless nested matroids on the set $[\nd ]$ of rank $r$ is the Eulerian number $ \Eulerian_{r-1,\nd}$.
Recall that there are no loopless matroids of rank $0$.
 From $\sum_{i=0}^{\nd-1} \Eulerian_{i,\nd} =\nd!$ it follows that there are  $\nd!$ loopless nested matroids on the ground set $[\nd]$ for rank $r\in\{1,\dots,\nd\}$.
We also reprove this last (coarser) result in the Appendix in \Cref{prop:hampe_dfact}.
Therefore the $d^\text{th}$ graded piece of the image of the map $\ncPsi\colon \Mat\to\WQSym$ is generated by a set of size $\nd!$, hence its dimension is $\leq\nd!$, establishing the upper bound.
\end{proof}

We note that in \cite[Theorem 1.5]{derksen2010valuative}, other enumerative results find the exponential power series of the number of nested matroids with a fixed rank and dimension.

\subsection{The lower bound}

We now prove the lower bound given in \cref{thm:lower_bound}.
Recall from \cref{obs:points} that the point Schubert matroids, i.e., those with only one basis, all have the same chromatic word-quasisymmetric function. The following Theorem states that the chromatic word-quasisymmetric functions for Schubert matroids except the aforementioned are linearly independent

\begin{thm}\label{thm:lower_bound_specific}
Consider the Schubert matroids generated by $A\subseteq[\nd]$ such that $A\neq\{1,2,\dots,k\}$ for any $k=0,1,\dots,\nd-1$.
Note that this implies $A\neq\emptyset$ but $A=[d]$ is possible.
The corresponding chromatic word-quasisymmetric functions $\ncPsi(\SM(A)) $ are linearly independent.
Therefore, the rank of the $\nd$ component of $\ncPsi$ is at least $2^{\nd}-\nd$.
\end{thm}

Note, the Schubert matroids under consideration above may contain loops.
However, by converting all loops into coloops we get a loopfree nested matroid with the same chromatic word-quasisymmetric function.

In order to prove \cref{thm:lower_bound_specific} we need two Lemmas.

\begin{lm}\label{lm:piTindependent}
 Let $\SM(A)$ be the Schubert matroid corresponding to
 \begin{equation*}
  A=\{a_1<a_2<\dots<a_r\}\subset[\nd]\,.
 \end{equation*}
 Let $t_1<t_2<\dots<t_l$ be an increasing sequence of elements in $[\nd]$.

 If $\{t_1<t_2<\dots<t_{i}\}\subset[\nd]$ is an independent set in $\SM(A)$ for some $i=0,1,\dots,l-1$ and $\{t_1<t_2<\dots<t_{i}\}\cup\{t_{i+1}\}$ is dependent in $\SM(A)$,
 then $\{t_1<t_2<\dots<t_{i}\}\cup\{t_j\}$ is also dependent in $\SM(A)$  for all $j=i+1,\dots,l$.
\end{lm}
\begin{proof}
 The set $\{t_1<t_2<\dots<t_{i}\}$ is independent  in $\SM(A)$ if and only if 
 \begin{align*}
  t_1\leq a_{r-i+1}\ \text{ and, }\
  \dots, \text{ and, }\
  t_{i-1}\leq a_{r-1},\ \text{ and, }\
  t_{i}\leq a_r \, .
 \end{align*}
If $\{t_1<t_2<\dots<t_{i}\}\cup\{t_{i+1}\}$ is not independent, then we have
 \begin{align*}
  t_1> a_{r-i}\text{ or, } \
  \dots, \text{ or, }\
  t_{i-1}> a_{r-2}, \text{ or, } \
  t_{i}> a_{r-1}, \ \text{ or, }\
  t_{i+1}> a_r \, .
\end{align*}
 This implies for $j=i+1,\dots,l$ that $t_j\geq t_{i+1}$, so
  \begin{align*}
  t_1> a_{r-i}\text{ or, } \
  \dots, \text{ or, }\
  t_{i-1}> a_{r-2}, \text{ or, } \
  t_{i}> a_{r-1}, \ \text{ or, }\
  t_j> a_r\, .
 \end{align*}
 Hence $\{t_1<t_2<\dots<t_{i}\}\cup\{t_j\}$ is also not independent for all $j=i+1,\dots,l$.
\end{proof}

\begin{lm}\label{lm:lower_bound_upper_triangular}
For $S = \{s_1 < \dots < s_l\}$ a subset of $[d]$, let $\{c_1 < \dots < c_{\nd-l}\}$ be its complementary set, and define the set composition of $[d]$:

\begin{equation}\label{eq:opi_T}
 \opi_S \coloneqq c_{\nd-l} | c_{\nd-l-1} | \dots | c_2 | c_1 s_l | s_{l-1}| \dots |s_2 | s_1 \, .
\end{equation}

Let $A = \{a_1 < \dots < a_r\}$ and $T = \{t_1 < \dots < t_m\}$ be subsets of $[\nd]$ that are not of the form $\{1, \dots, k\}$ for any $k = 0, \dots, \nd-1$ (in particular, $A\neq\emptyset$, $T\neq\emptyset$).
Then:
\begin{enumerate}[(i)]

\item\label{item:notgeneric} The set composition $\opi_A$ is not $\SM(A)$-generic.

\item\label{item:smallAgeneric} If $|A| < |T|$ then $\opi_T $ is $\SM (A)$-generic.

\item\label{item:equalAgeneric} If $|A| = |T|$ and $T \not\leqG A $ then $\opi_T$ is $\SM (A) $-generic.

\end{enumerate}
\end{lm}

Note that $A=[\nd]$,$T=[\nd]$ is possible.

\begin{proof}
 \ref{item:notgeneric}
By definition of Schubert matroid we have that $A$ is a basis in $\SM(A)$. Since $A\neq\{1,\dots,l\}$ we know that $c_1<a_r$ and hence $A'\coloneqq A\setminus\{a_r\}\cup\{c_1\}$ is also a basis.
By definition (see Equation~\eqref{eq:setcompositiongeneric}), we have that $f_{\opi_A}$ has the same value for the two bases $A$ and $A'$.
Hence, $\opi_A$ is not $\SM(A)$-generic.
 
 \ref{item:smallAgeneric} 
 Let $A=\{a_1<\dots<a_r\}\subseteq [\nd]$ with $r<m=\lvert T\rvert$.
 We want to show that there is a unique $\opi_T$-maximal basis $B\in \SM(A)$.
 We construct a basis of $\SM(A)$ greedily and show that there is a unique way of doing so.
 We construct the basis iteratively, i.e.,
 we try to add elements into independent sets starting with the smallest element in $T$ ($T=\{t_1<t_2<\dots<t_{m-1}<t_m\}$).
 Since $\lvert A\rvert =r<m=\lvert T\rvert$, either $\{t_1<\dots<t_r\}$ is a basis in $\SM(A)$ and hence the unique $\opi_T$ maximal basis,
 or we can apply \Cref{lm:piTindependent} and the unique $\opi_T$-maximal basis is $\{t_1<\dots<t_i\}\cup\{c_1<\dots<c_{r-i}\}$ for some $i\in\{1,\dots,m-1\}$.
 
\ref{item:equalAgeneric} 
By assumption ($T \not\leqG A $), the set $T$ is not a basis in $\SM(A)$.
Therefore, there is an element $i\in\{0,1,\dots,m\}$ such that $\{t_1<t_2<\dots<t_{i-1}\}$ is independent,
but $\{t_1<t_2<\dots<t_{i-1}\}\cup\{t_{i+1}\}$ is not independent in $\SM(A)$.
Again, by \Cref{lm:piTindependent}, $\{t_1<t_2<\dots<t_{i-1}\}\cup\{t_j\} $ is dependent for all $j=i+1,\dots,m$ and the Greedy algorithm constructs the unique $\opi_T$-maximal basis $\{t_1<t_2<\dots<t_{i-1}\}\cup\{c_1<\dots<c_{r-i}\}$ in $\SM(A)$.
\end{proof}

\begin{proof}[Proof of \cref{thm:lower_bound_specific}]
Recall that word-quasisymmetric functions have a basis indexed by set compositions.
In order to prove the linear independence, we will construct a matrix $\wqscmatrix$ with Schubert matroids indexing rows, set compositions indexing columns and entries recording whether the set composition is generic or not:
\begin{equation}\label{eq:defmatrix}
 \wqscmatrix_{\SM(A),\opi}\coloneqq\begin{cases}
                            1, & \text{if $\opi$ is $\SM(A)$-generic }\\
                            0, & \text{if $\opi$ is not $\SM(A)$-generic }
                           \end{cases}
\end{equation}

We can show that with the right choices of set compositions and ordering of rows and columns, after one row operation, we get a lower triangular matrix.

We define a total order on the collection of subsets $A\subset[d]$ that are not of the form $\{1,2,\dots,k\}$ for $k=0,1,\dots,d-1$ as follows:
\begin{itemize}
 \item we first order subsets increasingly by their size, i.e., starting with one-element subsets;
 \item for subsets with the same size we choose a linear extension of the Gale order $\leqG$, i.e., we successively choose minimal elements in the Gale order.
\end{itemize}
Rows in the matrix $\wqscmatrix$ are now indexed by Schubert matroids $\SM(A)$ following the total order on subsets defined above. 
Columns are indexed by set compositions $\opi_T$ as defined in \eqref{eq:opi_T} in \cref{lm:lower_bound_upper_triangular}, again respecting the total order on subsets defined above.

Note that the last row of $\wqscmatrix$ is indexed by the Schubert matroid $\SM([d])$ and the last column is indexed by the set composition $ d| d-1 | \dots | 1$.
Since $\SM([d])$ has only one basis, every set composition is generic and the last row in $\wqscmatrix$ consists of $1$s.
Similarly, the set composition $ d| d-1 | \dots | 1$ is generic for every matroid and the last column in $\wqscmatrix$ consists of $1$s.

It now follows from \cref{lm:lower_bound_upper_triangular} part \ref{item:notgeneric} that, except for the last row, the diagonal entries are $0$.
Moreover, in each row indexed by $\SM(A)$, entries above the diagonal entry correspond to set compositions $\opi_T$ where either $\lvert A\rvert < \lvert T\rvert$ or $ \lvert A\rvert = \lvert T\rvert$ and $T \not\leqG A $, so by \cref{lm:lower_bound_upper_triangular} parts \ref{item:smallAgeneric} and \ref{item:equalAgeneric} those entries are equal to $1$.

We now subtract the last row (containing all $1$s) from every other row. Then diagonal entries (except the last one) are $-1$ and the matrix is lower triangular.
\end{proof}

\begin{smpl}
 For $d=4$ the matrix $\wqscmatrix$ as defined in Equation~\eqref{eq:defmatrix} in the proof above is the following $12\times 12$ matrix:

\begin{equation*}
\begin{bmatrix}
0 & 1 & 1 & 1 & 1 & 1 & 1 & 1 & 1 & 1 & 1 & 1 \\
0 & 0 & 1 & 1 & 1 & 1 & 1 & 1 & 1 & 1 & 1 & 1 \\
0 & 0 & 0 & 1 & 1 & 1 & 1 & 1 & 1 & 1 & 1 & 1 \\
1 & 1 & 1 & 0 & 1 & 1 & 1 & 1 & 1 & 1 & 1 & 1 \\
1 & 1 & 1 & 0 & 0 & 1 & 1 & 1 & 1 & 1 & 1 & 1 \\
1 & 1 & 1 & 0 & 1 & 0 & 1 & 1 & 1 & 1 & 1 & 1 \\
1 & 1 & 1 & 0 & 0 & 0 & 0 & 1 & 1 & 1 & 1 & 1 \\
1 & 1 & 1 & 0 & 0 & 0 & 0 & 0 & 1 & 1 & 1 & 1 \\
1 & 1 & 1 & 1 & 1 & 1 & 1 & 1 & 0 & 1 & 1 & 1 \\
1 & 1 & 1 & 1 & 1 & 1 & 1 & 1 & 0 & 0 & 1 & 1 \\
1 & 1 & 1 & 1 & 1 & 1 & 1 & 1 & 0 & 0 & 0 & 1 \\
1 & 1 & 1 & 1 & 1 & 1 & 1 & 1 & 1 & 1 & 1 & 1 \\
\end{bmatrix}
\end{equation*}

We note that the order that we chose on sets is the following: 
\begin{equation*}
\{2\}, \{3\}, \{4\}, \{1, 3\}, \{1, 4\}, \{2, 3\}, \{2, 4\}, \{3, 4\}, \{1, 2, 4\}, \{1, 3, 4\}, \{2, 3, 4\}, \{1, 2, 3, 4\}\end{equation*}

Thus, for instance, the second row and first column shows that the Schubert matroid $\SM ( \{  3\} ) $, which has bases $\{ \{1\}, \{2\}, \{3\}\}$ is not $ (4 | 3 | 12 )$-generic.
\end{smpl}

\subsection{Conjectures based on computations}\label{sec:computational}
\hspace{1ex}

With the package developed in \cite{penaguiao2026chromatic} we are able to present some conjectures aided with data from computations.
The main conjecture supported by computational evidence is \cref{conj:kernelspan}, as discussed in the following, but we are also able to present some geometrically flavored conjectures based on these data.

 A set composition $(\pi_1, \ldots, \pi_k)=\opi \models [d]$ is a \Def{max-min set composition} if
 \begin{equation*}
  \max \pi_i  > \min \pi_{i+1}  \text{ for }i=1,\dots,k-1\,.
 \end{equation*}
Note that the set compositions we used to prove the lower bound (defined in Equation~\eqref{eq:opi_T}) are a special kind of max-min set compositions.
\begin{prop} The collection of
 max-min set compositions of $[\nd]$ is in bijection with permutations $\Sym_d$.
\end{prop}

\begin{proof}
We show this by constructing a bijection from max-min set compositions to permutations: order every part $\pi_i$ increasingly, concatenate parts to get inline notation of the permutation. 
This is invertible: indeed, for a permutation $\sigma$, identify descents and use those to define the parts, this describes uniquely the inverse image.
\end{proof}

\begin{conj}\label{conj:kernelspanspecific}
For fixed $d$ define a matrix with rows indexed by loopless nested matroids on $[d]$ and columns indexed by max-min set compositions of $[d]$ by
\begin{equation*}
 \wqscmatrixconj_{\natroid,\opi}\coloneqq\begin{cases}
                            1, & \text{if $\opi$ is $\natroid$-generic }\\
                            0, & \text{if $\opi$ is not $\natroid$-generic }
                           \end{cases}\,.
\end{equation*}
We conjecture this matrix to have full rank $d!$.
\end{conj}

The choice of max-min set compositions is arbitrary, as it is very well possible that some other choice of columns allows for a non-singular matrix.
It is, however, the one choice that has stood the test of computational analysis performed by the authors.

\begin{smpl}\label{smpl:d4nested}
Comparing this matrix with the one obtained in \cref{thm:lower_bound_specific}, we observe that the resulting matrix $\wqscmatrixconj_{\natroid,\opi}$ cannot be easily rearranged so that all zeroes are above or on the diagonal.
This fundamental step in \cref{thm:lower_bound_specific} cannot be made in full generality here, which jeopardizes our current proof strategy.

Specifically, for $\nd = 3$, and using the codebase in \cite{penaguiao2026chromatic}, specifically the file \textit{03\_min\_max\_conjecture\_testing.ipynb}:

\begin{verbatim}
python code/package/matrix_computation_small.py
\end{verbatim}
we obtain the following matrix
\begin{equation*}
\begin{bmatrix}
0 & 1 & 0 & 1 & 0 & 1 \\
0 & 0 & 1 & 0 & 1 & 1 \\
1 & 1 & 1 & 1 & 1 & 1 \\
0 & 1 & 1 & 0 & 1 & 1 \\
0 & 0 & 0 & 1 & 1 & 1 \\
0 & 1 & 1 & 1 & 0 & 1 \\
\end{bmatrix}
\end{equation*}

For $\nd = 4$ we can generate the $24\times 24 $ matrix
{\small
\begin{equation*}
 \begin{bmatrix}
0 & 1 & 0 & 1 & 0 & 1 & 0 & 1 & 0 & 1 & 0 & 1 & 0 & 1 & 0 & 1 & 0 & 1 & 0 & 1 & 0 & 1 & 0 & 1 \\
0 & 0 & 1 & 0 & 1 & 1 & 0 & 0 & 1 & 0 & 1 & 1 & 0 & 0 & 1 & 0 & 1 & 1 & 0 & 0 & 1 & 0 & 1 & 1 \\
0 & 0 & 0 & 0 & 0 & 0 & 1 & 1 & 0 & 0 & 0 & 0 & 1 & 1 & 1 & 1 & 0 & 0 & 1 & 1 & 1 & 1 & 1 & 1 \\
1 & 1 & 1 & 1 & 1 & 1 & 1 & 1 & 1 & 1 & 1 & 1 & 1 & 1 & 1 & 1 & 1 & 1 & 1 & 1 & 1 & 1 & 1 & 1 \\
0 & 1 & 0 & 1 & 0 & 1 & 0 & 1 & 1 & 0 & 1 & 1 & 0 & 1 & 1 & 0 & 1 & 1 & 0 & 1 & 1 & 0 & 1 & 1 \\
0 & 0 & 1 & 0 & 1 & 1 & 1 & 1 & 0 & 0 & 0 & 0 & 1 & 1 & 1 & 1 & 0 & 0 & 1 & 1 & 1 & 1 & 1 & 1 \\
0 & 1 & 1 & 0 & 1 & 1 & 0 & 1 & 0 & 1 & 0 & 1 & 0 & 0 & 0 & 1 & 1 & 1 & 0 & 0 & 0 & 1 & 1 & 1 \\
0 & 0 & 0 & 0 & 0 & 0 & 0 & 0 & 1 & 0 & 1 & 1 & 1 & 1 & 1 & 1 & 0 & 0 & 1 & 1 & 1 & 1 & 1 & 1 \\
0 & 0 & 0 & 1 & 1 & 1 & 0 & 0 & 0 & 1 & 1 & 1 & 0 & 1 & 0 & 1 & 0 & 1 & 0 & 1 & 1 & 1 & 0 & 1 \\
0 & 0 & 1 & 0 & 0 & 0 & 1 & 1 & 1 & 0 & 0 & 0 & 0 & 0 & 1 & 0 & 1 & 1 & 1 & 1 & 1 & 1 & 1 & 1 \\
0 & 1 & 1 & 1 & 0 & 1 & 0 & 1 & 1 & 1 & 0 & 1 & 0 & 1 & 1 & 1 & 0 & 1 & 0 & 1 & 0 & 1 & 0 & 1 \\
0 & 0 & 0 & 0 & 1 & 1 & 1 & 1 & 0 & 0 & 1 & 1 & 1 & 1 & 1 & 1 & 1 & 1 & 0 & 0 & 1 & 0 & 1 & 1 \\
0 & 0 & 1 & 0 & 1 & 1 & 0 & 1 & 1 & 0 & 1 & 1 & 0 & 0 & 1 & 0 & 1 & 1 & 0 & 0 & 1 & 0 & 1 & 1 \\
0 & 1 & 1 & 0 & 1 & 1 & 0 & 1 & 1 & 0 & 1 & 1 & 1 & 1 & 1 & 1 & 0 & 0 & 1 & 1 & 1 & 1 & 1 & 1 \\
0 & 0 & 0 & 0 & 1 & 1 & 0 & 0 & 1 & 0 & 1 & 1 & 0 & 1 & 1 & 0 & 1 & 1 & 0 & 0 & 1 & 0 & 1 & 1 \\
0 & 0 & 0 & 1 & 1 & 1 & 1 & 1 & 1 & 0 & 0 & 0 & 0 & 1 & 1 & 0 & 1 & 1 & 1 & 1 & 1 & 1 & 1 & 1 \\
0 & 0 & 1 & 0 & 0 & 0 & 0 & 0 & 1 & 0 & 1 & 1 & 0 & 0 & 1 & 0 & 1 & 1 & 0 & 1 & 1 & 0 & 1 & 1 \\
0 & 1 & 1 & 1 & 0 & 1 & 1 & 1 & 0 & 0 & 1 & 1 & 1 & 1 & 1 & 1 & 1 & 1 & 0 & 1 & 1 & 0 & 1 & 1 \\
0 & 0 & 1 & 0 & 1 & 1 & 0 & 0 & 0 & 0 & 1 & 1 & 0 & 0 & 0 & 1 & 1 & 1 & 0 & 0 & 1 & 0 & 1 & 1 \\
0 & 0 & 1 & 0 & 0 & 0 & 0 & 0 & 0 & 1 & 1 & 1 & 0 & 0 & 0 & 1 & 1 & 1 & 1 & 1 & 1 & 1 & 1 & 1 \\
0 & 0 & 1 & 0 & 1 & 1 & 0 & 1 & 1 & 0 & 0 & 0 & 0 & 0 & 1 & 0 & 1 & 1 & 0 & 0 & 0 & 1 & 1 & 1 \\
0 & 1 & 0 & 0 & 1 & 1 & 0 & 1 & 1 & 1 & 0 & 1 & 1 & 1 & 1 & 1 & 1 & 1 & 0 & 0 & 0 & 1 & 1 & 1 \\
0 & 0 & 1 & 0 & 1 & 1 & 0 & 0 & 1 & 0 & 1 & 1 & 0 & 1 & 1 & 1 & 0 & 0 & 0 & 1 & 1 & 1 & 0 & 1 \\
0 & 0 & 1 & 1 & 1 & 1 & 1 & 1 & 1 & 1 & 1 & 1 & 0 & 1 & 1 & 1 & 0 & 1 & 0 & 1 & 1 & 1 & 0 & 1
\end{bmatrix}
\end{equation*}
}

For instance, the row corresponding to $\SM(\{2,4\})$ is detailed in \cref{tab:setcomp_SM24}.
\end{smpl}

It might be useful to analyse the chromatic word-quasisymmetric functions in different bases.
In the following conjecture we introduce and use a so called \emph{alternating sum bases} of the $\WQSym$.
Specifically, arises as an alternating sum of the monomial basis:
\begin{equation*}
\mathbb{P}_{\opi} = \sum_{\otau \preceq \opi} (-1)^{\ell(\opi) - \ell(\otau)} \mathbb{M}_{\otau}\, ,\end{equation*}
where $\otau\preceq\opi$ means that $\otau$ results from $\opi$ by a coarsening of set compositions.
We conjecture the following about this matrix.

\begin{conj}[Alternating sum matrix]
Half of the rows are zero in this matrix for every dimension.
Furthermore, the remaining entries are $1$, $0$ or $-1$.
\end{conj}

This conjecture was observed for $d\leq 6 $ in the code
\begin{verbatim}
python code/package/dimension_computation_csf.py
\end{verbatim}

Computations for $\nd\leq 5$ suggest that the chromatic word-quasisymmetric functions of loopless nested matroids might even form a lattice basis.

\section{Further work}

The result on the number of loopless nested matroids by Hampe \cite{hampe_intersection_2017} is finer. Hampe counts nested matroids (loopless) by rank: those are the Eulerian numbers.

\begin{quest}
What is the rank of each matrix-component separated out by matroid-rank?
\end{quest}
This is a simpler question that might help to gain a better understanding towards the main conjecture of this paper.

The results from \cite{derksen2010valuative} that we used here apply to polymatroids and megamatroids, therefore a lower bound on the corresponding rank follows immediately.
\begin{quest}
What is a extension of the chromatic word-quasisymmetric functions presented here to polymatroids and mega matroids?
\end{quest}

In this paper, we appealed to \cite{derksen2010valuative} in order to show that the chromatic quasisymmetric and word-quasisymmetric function of matroids are strongly valuative functions.
\begin{quest}
What is the actual witness function to the strong valuativity of the chromatic quasisymmetric and word-quasisymmetric functions?
\end{quest}

In \cite{ferroni_polytope_2025}, a new polytope of all matroids is created, where the coefficients of the indicator function of a matroid polytope, as written as a linear combination of indicator functions of nested matroid base polytopes, form the vertices of the so called \emph{polytope of all matroids} $\Omega_{r,{\nd}}$.
Further, in \cite[section 7.2.]{ferroni_polytope_2025}, the \emph{polytope of $\varphi$} is introduced for a valuative function $\varphi$.
This polytope is always a projection of the polytope of all matroids $\Omega_{r,{\nd}}$.

\begin{quest}
What is the corresponding polytope for the valuative function $\ncPsi$ introduced here, as well as $\cPsi$.
\end{quest}

For example, when $d = 5$ and rank $3$ there are six different loopless Schubert matroids, so $\Omega_{r,{\nd}}$ is a polytope in $\R^6$.

For graphs and Minkowski sums of simplices such a basis set, the kernel of the corresponding function was described in \cite{penaguiao2020kernel}, so an extension of the question above is also valid in this context.

For the final question, we recall that lattice-path matroids are defined in \cite{Oxley2006}.
This is the equivalent of the Stanley tree conjecture from \cite{stanley1995symmetric} for matroids.

\begin{quest}
Are loopless lattice-path matoids distinguished by their chromatic quasisymmetric function?
\end{quest}

\appendix
\crefalias{section}{appendix}

\section{Quasisymmetric functions of matroids}\label{sec:qsym}

We review the work by Billera-Jia-Reiner \cite{billera2009quasisymmetric} and combine it with results from Derksen-Fink \cite{derksen2010valuative} to tell the parallel story of quasisymmetric functions of matroids.
We add a description of the relations contained in the kernel of the map $\cPsi\colon\MatLinSpace\to\QSym$.

\subsection{Preliminaries}

Let $\nd$ be a non-negative integer.
A \Def{composition} $\alpha $ of $\nd$ is a tuple $(\alpha_1, \ldots, \alpha_k)$ of positive integers such that $|\alpha| = \nd$, where $|\alpha| \coloneqq \sum_i \alpha_i$.
We write $\alpha \models \nd$ and we say that $\ell(\alpha ) = k$ is the \Def{length} of $\alpha$.
There is a unique composition of $0$, the empty tuple.

In this section we work over a countable infinite collection of commuting variables $\bx = \bx_1, \bx_2, \ldots $.
For a map $f\colon[\nd]\to \Z_{\geq1}$,
we define the degree $\nd$ monomial
\begin{equation*}
 \bx_f\coloneqq \Pi_{i\in[\nd]} \bx_{f(i)}\,.
\end{equation*}
A formal power series of degree $\nd$ on $\bx $ is a formal sum
\begin{equation*}
 F(\bx) = \sum_{ f\colon[\nd]\to \Z_{\geq1}} c_{f} \bx_{f}\,.
\end{equation*}

A \Def{quasisymmetric function} of degree $\nd$ is a formal series $\sum c_{f} \bx_{f}$ such that for every composition $\alpha\models \nd$, the coefficient of $\bx_{i_1}^{\alpha_1} \cdots \bx_{i_k}^{\alpha_k}$ is independent of the choice of indices $i_1 < \cdots < i_k$.
We write $\QSym_\nd $ for the vector space of quasisymmetric functions of degree $\nd$, and we let $\QSym = \oplus_{\nd\geq 0} \QSym_\nd$ be the space of finite sums of homogeneous quasisymmetric functions, which we simply call \Def{quasisymmetric functions}.

The space of quasisymmetric functions is a ring, as observed in \cite{gessel1984multipartite}, under the usual polynomial product extended to formal sums.
The \Def{monomial basis} indexed by compositions $\alpha$ is given by
\begin{equation*}
\Mco_{\alpha} = \sum_{i_1 < \cdots < i_k}\bx_{i_1}^{\alpha_1}\cdots\bx_{i_k}^{\alpha_k} \, ,
\end{equation*}
where $k=l(\alpha)$ is the length of the composition $\alpha$ and the sum runs over all ordered k-tuples of indices.
In this way we have, for instance, $\Mco_{(1, 1)} = \sum_{i<j} \bx_i \bx_j$.
One can see that $\Mco_{(1)}^2 = 2 \Mco_{(1, 1)} + \Mco_{(2)}$.
Recall that, for $\nd > 0$, the number of compositions of $\nd$ is $2^{\nd-1}$ (stars and bars) and hence the dimension of the $\nd$-th graded piece $\QSym_d$ is also $2^{\nd-1}$.

Recall that we say a function $f\colon \groundset \to\R$ is an $\matroid$-generic map for a matroid $\matroid=(\groundset,\bases)$ if there is a unique basis $B\in \bases$ with maximum weight $f(B)$.
Similarly, we call a set composition $\opi$ an \Def{$\matroid$-generic set composition} if the map $f_{\opi} \colon\groundset\to[\nd]$ is $\matroid$-generic.

Let $\matroid$ be a matroid with ground set $\groundset$.
The (chromatic) \Def{quasisymmetric function of a matroid} $\cPsi(\matroid)$ is defined as
\begin{align}\label{eq:defqsymmatroid}
 \cPsi(\matroid) \coloneqq \sum_{\substack{f:[\nd]\to\Z_{\geq1}\\ \text{ is $\matroid$-generic}}}  \bx_f
  = \sum_{\substack{\opi \text{ is $\matroid$-generic}\\\text{set composition of }[\nd]}} \Mco_{\alpha(\opi)}\, .
\end{align}
The second equality follows from the fact, that whether a function $f$ is $\matroid$-generic or not, only depends on the set composition type $\opi(f)$ (see also \cite[Proposition 2.1]{billera2009quasisymmetric}).

\subsection{Rank computations}

Billera-Jia-Reiner showed that the quasisymmetric function of matroids  defines a Hopf algebra morphism from the matroid Hopf algebra to the algebra of quasisymmetric functions,
is an isomorphism invariant, and it is (weakly) valuative \cite{billera2009quasisymmetric}.
Combining the latter with the result by Derksen-Fink \cite[Theorem 3.5]{derksen2010valuative}, that strongly valuative is equivalent to weakly valuative for matroids implies that the kernel of the map $\cPsi\colon\MatLinSpace\to\QSym$ contains all the valuative relations, that is,
\begin{equation}\label{eq:appvaluative}
 0=\sum \alpha_i \onebb[\Pol_{\matroid_i}] \quad \Rightarrow\quad
 0=\sum \alpha_i \cPsi({\matroid_i})\,.
\end{equation}
Hence, the kernel of the map $\cPsi\colon\MatLinSpace\to\QSym$ contains the \Def{valuative relations} of the form in Equation~\eqref{eq:appvaluative}.
They also show that PI-matroids on $\nd$ elements such that the first operation performed in the construction is adding an isthmus generate the $\nd$th graded piece $\QSym_d$\footnote{The result is in fact finer, see \cite[Theorem A.2]{billera2009quasisymmetric}} \cite[Appendix]{billera2009quasisymmetric}.
We will only use, that the map $\cPsi\colon\MatLinSpace\to\QSym$ is surjective to derive a complete description of the kernel $\ker(\cPsi)$.
Except for more explicit description of the relations in the kernel, this does not add new results, but we think that it is instructive to see similarities and differences in comparison with the word-quasisymmetric functions.

\begin{cor}[{\cite[Theorem A.2]{billera2009quasisymmetric}}, rephrased]
The kernel of $\cPsi$ is spanned by the valuative relations (Equation~\eqref{eq:appvaluative}) and the loop-coloop relations (\Cref{prop:sameqsywithup}).

In this way, the function $\cPsi$ is determined by its value on Schubert matroids $\SM(A)$ for $A\subseteq[\nd]$ with $\nd\in A$.
\end{cor}

For a subset $A\subseteq[\nd]$ with $\nd\notin A$ we define the \Def{up operation} by
\begin{equation*}
 A^\uparrow\coloneqq\{1\}\cup \{a+1\colon a\in A\} \,.
\end{equation*}
This defines a bijective map between $r$-subsets of $[\nd]$ not containing $\nd$ and $(r+1)$-subsets of $[\nd]$ that do contain $1$.

\begin{prop}\label{prop:sameqsywithup}
If $d\not \in A$ then  $\cPsi(\SM(A))=\cPsi(\SM(A^{\uparrow}))$.
\end{prop}
We call the relations of the form $\cPsi(\SM(A))-\cPsi(\SM(A^{\uparrow}))\in\ker(\cPsi)$ \Def{loop-coloop relations}.
We will give two proofs for \cref{prop:sameqsywithup}, one elementary and combinatorial and one using the product of the Hopf algebra structures and the fact that quasisymmetric functions are a Hopf algebra homomorphism.

\begin{proof}[First proof of \Cref{prop:sameqsywithup}]
Let us denote by $\SM_{\groundset}(A)$ a Schubert matroid $\matroid$ with ground set $\groundset$ generated from $A$.
Recall the definition of direct product of matroids  in \cref{ssec:prelimMatroids}.
Denote by $\unifMat{r}{[\nd]}$ the uniform matroid of rank $r$ in the ground set $[\nd]$.
Observe that for $A\subseteq[d-1]$
\begin{align*}
 \SM_{[d]}(A^{\uparrow}) &\cong \SM_{[d-1]}(A) \oplus \unifMat{1}{[1]} \, ,\\
 \SM_{[d]}(A) &\cong \unifMat{0}{[1]} \oplus \SM_{[d-1]}(A)  \, .
\end{align*}
So it follows from the multiplicative properties of $\cPsi$, as well as $\cPsi(\unifMat{1}{[1]}) = \cPsi(\unifMat{0}{[1]})=\Mco_{(1)}$, that
\begin{equation*}
\begin{split}
&\cPsi(\SM_{[\nd]}(A^{\uparrow})) = \cPsi(\SM_{[\nd-1]}(A) \oplus \unifMat{1}{[1]}) =  \cPsi(\SM_{[\nd-1]}(A) ) \cdot  \cPsi(\unifMat{1}{[1]})  =  \\
&\cPsi(\SM_{[\nd]}(A) ) \cdot  \cPsi(\unifMat{0}{[1]})   = \cPsi( \unifMat{0}{[1]} \oplus \SM_{[\nd]}(A) ) = \cPsi(\SM_{[\nd]}(A))
\end{split}
\end{equation*}
\end{proof}

We define the up operation for a set composition $\opi=(\pi_1,\dots,\pi_k)\models [\nd]$ by
\begin{equation*}
 \opi^\uparrow\coloneqq (\pi_1\boxplus1,\dots,\pi_k\boxplus1)\,,
\end{equation*}
where for $\pi\subset[\nd]$ we define
\begin{equation*}
 \pi\boxplus1\coloneqq\begin{cases}
             \{p+1 \colon p\in \pi\}\,,&\quad\text{if }\nd\notin\pi\,,\\
             \{1\}\cup\{p+1 \colon p\in \pi\setminus\{\nd\}\}\,,&\quad\text{if }\nd\in\pi\,.
            \end{cases}
\end{equation*}
Recall that $f_{\opi}(b)=j$  for $b\in\pi_j$ and for a subsets $B\subseteq [d]$ we have $f_{\opi}(B)=\sum_{b\in B} f_{\opi}(b)$. We collect some properties of the up operations and then give a second proof for \Cref{prop:sameqsywithup}.
\begin{lm}\label{lm:up}
 Properties of $\uparrow$-operation:
 \begin{enumerate}[a)]
  \item \label{it:finiteup} Let $A\subset[d]$ be a set with $d\notin A$. After a finite number of $\uparrow$-operations the resulting set $A^{{\uparrow}{...}{\uparrow}}$ will contain $d$.
  \item \label{it:SMAup} Let $A\subset[d]$ be a set with $d\notin A$ then   $\SM(A^{\uparrow})=\{B^{\uparrow} \colon B\in\SM(A) \}$.
  \item \label{it:fbup} For a set composition $\opi\models [\nd]$  we have $f_{\opi^{\uparrow}} ( b+1) = f_{\opi}(b) $ for $b\neq d$.
  \item \label{it:fBup} Let $B\subset[d]$ be a set with $d\notin B$, then $f_{\opi^{\uparrow}} ( B^{\uparrow}) = f_{\opi} ( B) + j$ where $d\in\pi_j$.
  \item \label{lm:up_arrow_preserves_properness} Let $A\subset[\nd]$ be such that $\nd\notin A$.
 Then, a set composition $\opi$ is $\SM(A)$-generic if and only if the set composition $\opi^{\uparrow}$ is $\SM(A^{\uparrow})$-generic.
 \end{enumerate}
\end{lm}
\begin{proof}
\begin{description}[style=multiline,leftmargin=2em]
 \item[\ref{it:finiteup}]  If the largest element in $A$ is $k$ then $d-k$ up operations are needed.

 \item[\ref{it:SMAup}]
  Recall that for $A=\{a_1<\dots<a_r\}\subseteq [d]$ the set of bases of the Schubert matroid $\SM(A)$ is
  \begin{align*}
   \SM(A)=\{ \{b_1<\dots<b_r\} \colon b_i\leq a_i \text{ for }i=1,\dots,r \}\,.
  \end{align*}
  Hence if $d\notin A$ then $d\notin B$ for every basis $B\in\SM(A)$ and from
  \begin{align*}
   A^\uparrow=\{1\}\cup  \{a_i+1\colon i=1,\dots,r\} \\
   B^\uparrow=\{1\}\cup \{b_i+1\colon i=1,\dots,r\} \\
  \end{align*}
  the claim follows.

  \item[\ref{it:fbup}] Let  $b\neq d$ then $b+1\in\opi^{\uparrow}_k=\pi_k\boxplus1$ if and only if $b\in\opi_k$.

  \item[\ref{it:fBup}] Note that $1\in\pi_j^\uparrow$ if and only if $d\in\pi_j$. Then we have
  \begin{align*}
   f_{\opi^{\uparrow}}( B^{\uparrow})&= f_{\opi^{\uparrow}}(1)+\sum_{b\in B}f_{\opi^{\uparrow}}(b+1)\\
    &= j+\sum_{b\in B} f_{\opi}(b) = j+ f_{\opi} ( B)\,.
  \end{align*}

  \item[\ref{lm:up_arrow_preserves_properness}]
  Let $\opi$ be a set composition with $d\in\pi_j$ and $A\subseteq [d-1]$.
  Then $\opi$ is $\SM(A)$-generic if and only if $f_{\opi}(B)$ has a unique maximum among the bases $B\in\SM(A)$.
  Because of \ref{it:fBup} this is the case if and only if $f_{\opi^\uparrow}(B^\uparrow)$ has a unique maximum among the bases $B^\uparrow\in\SM(A^\uparrow)$, that is, if and only if  $\opi^\uparrow$ is $\SM(A^\uparrow)$-generic.

\end{description}
\end{proof}

\begin{proof}[Second proof of \Cref{prop:sameqsywithup}]

Recall that we write $\alpha(\opi) $ for the composition type of the set composition $\opi$.
We denote the coefficient of a quasisymmetric function corresponding to a monomial basis element $\Mco_{\alpha}$ as $[\Mco_{\alpha}] $.
It follows from \eqref{eq:defqsymmatroid} that
\begin{align*}
 [\Mco_{\alpha}] \cPsi(\SM(A^{\uparrow})) &= \lvert \{ \opi \, \colon \,  \SM(A^{\uparrow})\text{-generic set composition such that } \alpha(\opi ) = \alpha \} \rvert \,, \\
 [\Mco_{\alpha}] \cPsi(\SM(A)) &= \lvert \{ \opi \, \colon \,  \SM(A)\text{-generic set composition such that } \alpha(\opi ) = \alpha \} \rvert \, .
\end{align*}

Note that \cref{lm:up}\ref{lm:up_arrow_preserves_properness} shows that $\opi$ is $\SM(A)$-generic if and only if $\opi^{\uparrow}$ is $\SM(A^{\uparrow})$-generic.
Furthermore, $\alpha(\opi^{\uparrow}) = \alpha(\opi)$.
This constructs a bijection between the two sets above.
Since the composition $\alpha$ is generic, we conclude that $\cPsi(\SM(A)) = \cPsi(\SM(A^{\uparrow}))$.
\end{proof}

\begin{cor}
 $\QSym_d$ is generated by quasisymmetric functions that come from Schubert matroids defined by sets $A\subset[d]$ that contain $d$.
\end{cor}
\begin{proof}
 From \cite[Theorem 6.3]{derksen2010valuative} we know that any indicator function $\onebb[\Pol_\matroid]$ of a matroid $\matroid$ can be uniquely written in terms of indicator functions of Schubert matroids, i.e.,
 \begin{equation*}
 \onebb[\Pol_\matroid] = \sum_{A\subseteq [d]} c_A \onebb[\Pol_\SM(A) ]
 \end{equation*}
 for some coefficients $c_A\in\Z$.
 Since the map $\cPsi\colon\MatLinSpace\to\QSym$ is strongly valuative, we get
  \begin{align*}
  \cPsi(\matroid) &= \sum_{A\subseteq [d]} c_A \cPsi(\SM(A) )
  \end{align*}
  We can now apply \cref{prop:sameqsywithup}:
 \begingroup
\allowdisplaybreaks
 \begin{align*}
  \cPsi(\matroid) &= \sum_{A\subseteq [d]} c_A \cPsi(\SM(A) )\\
      &= \sum_{\substack{A\subseteq [d-1]}} c_A \cPsi(\SM(A) ) + \sum_{\substack{A\subseteq [d]\\ d\in A }} c_A \cPsi(\SM(A) ) \\
      &=\sum_{\substack{A\subseteq [d-1]}} c_A \cPsi(\SM(A^{\uparrow}) ) + \sum_{\substack{A\subseteq [d]\\ d\in A }} c_A \cPsi(\SM(A) ) \\
      &=\sum_{\substack{A\subseteq [d-2]}} c_A \cPsi(\SM(A^{\uparrow}) ) +\sum_{\substack{A\subseteq [d-1]\\ d-1\in A}} c_A \cPsi(\SM(A^{\uparrow}) ) + \sum_{\substack{A\subseteq [d]\\ d\in A }} c_A \cPsi(\SM(A) ) \\
      &=\sum_{\substack{A\subseteq [d-2]}} c_A \cPsi(\SM(A^{\uparrow\uparrow}) ) +\sum_{\substack{A\subseteq [d-1]\\ d-1\in A}} c_A \cPsi(\SM(A^{\uparrow}) ) + \sum_{\substack{A\subseteq [d]\\ d\in A }} c_A \cPsi(\SM(A) ) \\
      &=...
 \end{align*}
 \endgroup
 Note, if $A\subset[d-1]$ with $d-1\in A$ then $d\in A^{\uparrow}$.
 After finitely many iterations, we will obtain an expression of $\cPsi(\matroid)$ in terms of $ \cPsi(\SM(A))$ such that $d\in A$.
\end{proof}

Note that $\QSym_d$ has a basis indexed by compositions of $[d]$, there are $2^{d-1}$ compositions of $[d]$ and hence $\QSym_d$ has dimension $2^{d-1}$.
\begin{cor}\label{claim}
 The set  $\{\cPsi(\SM(A))\colon A\subset[\nd], d\in A\}$ is linearly independent.
\end{cor}
\begin{proof}
This follows from the surjectivity of the map $\cPsi\colon\MatLinSpace\to\QSym$ \cite[Theorem  A.1]{billera2009quasisymmetric} and
the fact that there are $2^{d-1}$ compositions of $d$, hence the $\QSym_d$ has dimension $2^{d-1}$
as well as the fact the number of subsets $A\subseteq [d]$ that contain $d$ is also $2^{d-1}$.

\end{proof}

\section{Hopf monoid morphism}\label{sec:hopfmonoid}

This paper deals with the chromatic word-quasisymmetric function of matroids $\ncPsi$, which turns out to be a Hopf monoid morphism.
This map is realisable from a construction in \cite[Section 5.4]{penaguiao2020algebraic}.

In this section we present a self contained proof that the map $\ncPsi $ is a Hopf monoid morphism.

\begin{thm}\label{thm:ncpsi_is_hopfmorphism}
The map $\ncPsi: \Mat \rightarrow \WQSym$ is a Hopf monoid morphism.
Further, the corresponding map between Hopf algebras is a Hopf algebra morphism.
\end{thm}

Let us turn to generalised permutahedra to establish this result.
We consider a generalised permutahedron $\Pol \in \R A$.
For a direction $\vy \in \R^A$, let $\Pol^{\vy}$ denote the maximum face of $\Pol$ in direction $\vy$.
For a set composition $\opi \models A$, let $\Pol^{\opi} = \Pol^{\vy_{\opi}}$, where $\vy_{\opi} $ is a direction in the open braid cone defined by the set composition $\opi$ in the braid fan.

For a generalized permutahedron $\Pol$ and a set composition $\opi$, we say that a set composition $\opi$ is $\Pol$-generic when the face $\Pol^{\opi}$ is a vertex.

Define, for a generalised permutahedron $\Pol$, the chromatic word-quasisymmetric function
\begin{equation*}\ncPsi_{\mathtt{GPer}}(M) \coloneqq \sum_{\substack{\opi \models I \\ \opi \text{ is $\Pol$-generic} }} \Mnco_{\opi} \, . \end{equation*}

First we note that $\ncPsi = \ncPsi_{\mathtt{GPer}} \circ \mathrm{rank}$, where $\mathrm{rank}: \mathtt{Mat} \rightarrow \mathtt{GPer}$ is the usual function mapping a matroid to its matroid base polytope.
This was shown to be a Hopf monoid morphism in \cite[Proposition 14.3]{aguiar2023hopf}.
To show \cref{thm:ncpsi_is_hopfmorphism}, it suffices therefore to show that $\ncPsi_{\mathtt{GPer}}$ is a Hopf monoid morphism.

\begin{lm}\label{lm:product-hopfmonoid}
Let  $A$, $B$ be two  disjoint sets and let $\Pol\subseteq\R A$, $\Qol\subseteq\R B$ be generalised permutahedron.
A set composition $\opi\models A\sqcup B$ is $\Pol\times \Qol$-generic if and only if 
$\opi|_A$ is $\Pol$-generic and $\opi|_B$ is $\Qol$-generic. 
\end{lm}
\begin{proof}
Recall that
\begin{equation*}
 (\Pol\times\Qol)^\vy = \Pol^{\vy|_A}\times\Qol^{\vy|_B}\,.
\end{equation*}
Now let $\vy\in\R^{A\sqcup B}$ be some direction in the open cone defined by the set composition $\opi$ in the braid fan.
Then $\vy|_A \in\R^A$ (resp.$\vy|_B \in\R^B$)  lies in the open braid cone associated to the set composition $\opi|_A$ (resp. $\opi|_B$).
Since a Cartesian product of polytopes is a point if and only if the two factors are points, the claim follows.
\end{proof}

Let us now introduce some auxiliary notation.
For two disjoint sets $I$ and $J$ and set compositions $\otau = (\tau_1, \ldots, \tau_l) \models I $ and $ \osigma = (\sigma_1, \ldots, \sigma_m) \models J$, we define $\otau\oplus \osigma \models I \sqcup J$ as the set composition $(\tau_1, \ldots, \tau_l, \sigma_1, \ldots, \sigma_m)$.

If $a, b \in I$, we say that $  a <_{\otau } b$ if $a\in \tau_i$, $b\in \tau_j$ with $i < j$.
For two sets $A, B\subseteq I $ we say that $  A <_{\otau } B$ if $ a <_{\otau } b$ for every $a \in A, b\in B$.

For $A, B$ disjoint sets and $\Pol$ a generalised permutahedron in $\R {(A \sqcup B)}$, recall that $\Pol|_A \times \Pol/_A\subseteq \R A \times \R B$ is the face of the generalized permutahedron $\Pol$ that maximizes $\sum_{i\in A} \vx_i $.
In that way, $\Pol|_A $ is a generalized permutahedron in $\R A $ and $\Pol/_A$ is a generalized permutahedron in $\R B$.

\begin{lm}\label{lm:coproductp1-hopfmonoid}
Let $A, B$ be disjoint sets and let $\Pol$ be a generalised permutahedron in $\R {(A \sqcup B)}$.
For a  set composition $\opi \models A\sqcup B$ such that $A <_{\opi } B $ and define $\otau  = \opi|_A$ and $\osigma  = \opi|_B$.
Then $\opi $ is $\Pol$-generic if and only if $\otau$ is $\Pol|_A$-generic and $\osigma$ is $\Pol/_A$-generic.

Similarly, for every set compositions $\otau \models A, \osigma \models B$ such that $\otau$ is $\Pol|_A$-generic and $\osigma$ is $\Pol/_A$-generic, the set composition $\opi = \otau \oplus \osigma$ satisfies $A <_{\opi } B $, $\otau  = \opi|_A$, $\osigma  = \opi|_B$, and $\opi$ is $\Pol$-generic.
\end{lm}

\begin{proof}
Write $\opi = (\pi_1 ,  \ldots,  \pi_k)$.
Since  $A <_{\opi } B $, there is some $j \in \{0, \ldots , k\}$ such that $ A = \bigcup_{i=1}^j \pi_i$.
Then $\otau = \opi|_A = (\pi_1 ,  \ldots,  \pi_j)$, $ \osigma =\opi|_B = (\pi_{j+1} ,  \ldots,  \pi_k)$ and $\opi = \otau \oplus \osigma$.

In \cite[Proposition 5.4.]{aguiar2023hopf}, it is shown that
\begin{equation*}
P^{\opi} = \mu_{\pi_1, \ldots, \pi_k} \circ\Delta_{\pi_1, \ldots , \pi_k} (P)\, ,
\end{equation*}
where $\Delta_{\pi_1, \ldots , \pi_k}$ denotes the higher coproduct map (resp. $ \mu_{\pi_1, \ldots, \pi_k}$ the higher product map), i.e., the maps obtained by iterating the coproduct map $\Delta$ (resp. product map $\mu$).
See, e.g., \cite[Section 2.6]{aguiar2023hopf} for details.

It follows that 
\begin{equation*}
\begin{split}
P^{\opi} &= \mu_{\pi_1, \ldots, \pi_k} \circ \Delta_{\pi_1, \ldots , \pi_k} (P)  \\
 &= \mu_{A, B} \circ (\mu_{\pi_1, \ldots, \pi_j} \otimes \mu_{\pi_{j+1}, \ldots, \pi_k}) \circ (\Delta_{\pi_1, \ldots , \pi_j} \otimes \Delta_{\pi_{j+1}, \ldots, \pi_k}) \circ  \Delta_{A, B} (P)\\
 &= \mu_{A, B} \circ (\mu_{\pi_1, \ldots, \pi_j} \otimes \mu_{\pi_{j+1}, \ldots, \pi_k}) \circ (\Delta_{\pi_1, \ldots , \pi_j} \otimes \Delta_{\pi_{j+1}, \ldots, \pi_k})(P|_A \otimes P/_A ) \\
 &=  \mu_{A, B} ( (P|_A)^{\otau} \otimes (P/_A)^{\osigma} )  ) =  (P|_A)^{\otau} \times (P/_A)^{\osigma}
\end{split}
\end{equation*}
on the second equation we used the associativity and coassociativity properties.

It thus follows that $P^{\opi} $ is a vertex if and only if both $(P|_A)^{\otau} $ and $(P/_A)^{\osigma}$ are a vertex, concluding the first part of the claim.
The second part follows from the definition of $\oplus $ as well as the first part of this claim.
\end{proof}

\begin{prop}
The map $\ncPsi_{\mathtt{GPer}}: \mathtt{GPer} \rightarrow \WQSym$ is a Hopf monoid morphism.
\end{prop}

\begin{proof}

We establish that the map $\ncPsi$ is a Hopf monoid morphism by showing that it preserves the product, coproduct, unit and counit.

To show that it preserves the \textbf{product}, take disjoint sets $A$, $B$ and let $\Pol$, $\Qol$ be generalised permutahedra in $\R A, \R B$, respectively.
We use the product structure on the $\{\Mnco_{\opi}\}$ basis described in \cite[Section 5.3, Equation (5.4)]{penaguiao2020algebraic} and note that
\begin{align*}
 \ncPsi(\Pol) \cdot \ncPsi(\Qol) &= \sum_{\substack{\otau \models A \\ \otau \text{ is $\Pol$-generic}}}\sum_{\substack{\osigma \models B \\ \osigma \text{ is $\Qol$-generic}}} \Mnco_{\otau} \cdot \Mnco_{\osigma}\\
 &= \sum_{\substack{\otau \models A \\ \otau \text{ is $\Pol$-generic}}}\sum_{\substack{\osigma \models B \\ \osigma \text{ is $\Qol$-generic}}} \sum_{\substack{\opi \models A\sqcup B \\ \opi|_A = \otau\\ \opi|_B = \osigma}} \Mnco_{\opi} \\
 &= \sum_{\substack{\opi \models A\sqcup B }} \Mnco_{\opi} \sum_{\substack{\otau \models A \\ \otau \text{ is $\Pol$-generic}\\ \opi|_A = \otau}}\sum_{\substack{\osigma \models B \\ \osigma \text{ is $\Qol$-generic}\\ \opi|_B = \osigma}}  1 \\
 &= \sum_{\substack{\opi \models A\sqcup B \\ \opi|_A  \text{ is $\Pol$-generic}\\\opi|_B  \text{ is $\Qol$-generic}}} \Mnco_{\opi}  \\
& = \sum_{\substack{\opi \models A\sqcup B \\ \opi \text{ is $\Pol\times \Qol$-generic}}} \Mnco_{\opi}
\quad=\quad\ncPsi(\Pol\oplus \Qol)\, .
\end{align*}
where the last equality follows from \cref{lm:product-hopfmonoid}.

To show that it preserves the \textbf{coproduct}, let $A, B$ be disjoint sets, and $\Pol$ be generalized permutahedron in $\R{(A\sqcup B)}$.
The coproduct on the $\{\Mnco _{\opi}\}$ basis is defined in \cite[Section~$5.3$]{penaguiao2020algebraic}, hence we can compute:
\begin{equation*}
\Delta_{A, B }(\ncPsi(\Pol)) = \sum_{\substack{\opi \models A\sqcup B \\ \opi \text{ is $\Pol$-generic}}} \Delta_{A, B }\Mnco_{\opi} = \sum_{\substack{\opi \models A\sqcup B \\ \opi \text{ is $\Pol$-generic}\\ A <_{\opi} B}}\Mnco_{\opi|_A} \otimes \Mnco_{\opi|_B} \, .
\end{equation*}

From \cite{aguiar2023hopf}, the coproduct structure in $\mathtt{GPer}$ is as follows:
\begin{equation*}\ncPsi \otimes \ncPsi (\Delta_{A, B } \Pol) = \ncPsi \otimes \ncPsi (\Pol|_A \otimes \Pol/_A)  =  \sum_{\substack{\otau \models A \\ \otau \text{ is $\Pol|_A$-generic}}} \sum_{\substack{\osigma \models B \\ \osigma \text{ is $\Pol/_A$-generic}}} \Mnco_{\otau} \otimes \Mnco_{\osigma } \, , \end{equation*}
these two expressions are the same again due to \cref{lm:coproductp1-hopfmonoid}.

To show that $\ncPsi$ preserves unit and counit is a direct but tedious application of the definition.
We also remark that on the ground set $E = \emptyset $, the map $\ncPsi $ preserves the antipode trivially.
\end{proof}

Finally, since $\ncPsi: \mathtt{Mat} \rightarrow \WQSym$ is a Hopf monoid morphism, from \cite[Proposition 3.10]{aguiar2010monoidal} we get that the underlying map between Hopf algebras is a Hopf algebra morphism, as desired in \Cref{thm:ncpsi_is_hopfmorphism}.

\section{Matroidal double chains are nested matroids}\label{appendix:DCarenested}

In \cite{derksen2010valuative} it was shown matroidal double chains play a central role in valuative functions.
We make the relation between matroidal double chains and nested matroids explicit in this section, by first defining a matroid from a matroidal double chain (see \cref{thm:doublechains_are_matroids}).
We show that such matroids exist and are well defined, by giving an explicit \cref{constr:path} for the corresponding Schubert matroid that is isomorphic to it, see \cref{lm:paths_and_intervals}.
We further show that matroidal double chains uniquely define  nested matroids, and that every nested matroid has a representation as a matroidal double chain.
Finally we present an enumerative result, recovering the result for loopless nested matroids from \cite{hampe_intersection_2017}.

For convenience we recall the definition of matroidal double chains:
A pair of chains $(\uX, \ur)$, where $\uX~=~(X_1,~\ldots,~X_k~=~\groundset~)$ is a series of sets and $\ur = (r_1, \ldots, r_k)$ a series of integers, is called a \Def{matroidal double chains} if it satisfies the following properties:
\begin{enumerate}[(a)]
\item\label{it:DCsetinclusionappendix} $\emptyset \subsetneq X_1 \subsetneq \cdots \subsetneq X_k = \groundset $

\item\label{it:rankinequalitiesappendix} $0 \leq r_1 < \cdots < r_k = r$.

\item\label{it:DcRanktogetherappendix} $0 < |X_1| - r_1 < \cdots <|X_{k-1}| - r_{k-1} \leq |X_k| - r_k = |\groundset | - r$.
\end{enumerate}
The associated polytope was defined as
\begin{equation}\label{eq:defpolofdoublechainappendix}
    \Pol(\uX, \ur)\coloneqq \{\bx\in[0,1]^d\colon \sum_{i=1}^d \bx_i =r\,,  \sum_{i\in X_j}\bx_i \leq r_j \forall j\}\,.
\end{equation}

\begin{thm}\label{thm:doublechains_are_matroids}
Let $(\uX, \ur)$ be a matroidal double chain.
Then there exists a (necessarily unique) matroid $\matroid \coloneqq \matroid(\uX, \ur)$ isomorphic to a Schubert matroid such that for the matroid base polytope $\Pol_\matroid = \Pol(\uX, \ur)$ holds.
\end{thm}
We will directly show that these matroid base polytopes come from matroids isomorphic to Schubert matroids.
\begin{thm}\label{thm:doublechains_are_nestedmatroids}
For every nested matroid $M$, there is a unique matroidal double chain $(\uX, \ur)$ such that $M = M(\uX, \ur)$.
Furthermore, every matroidal double chain defines a unique nested matroid.
\end{thm}
%
We further define an auxiliary notion of matroidal double chains to build an association to Schubert matroids:
an \Def{ordered matroidal double chain}\footnote{We use a different convention than \cite[Section 6]{derksen2010valuative}}
is a matroidal double chain $(\uX, \ur)$ such that each $X_i$ is an interval containing $d$, i.e., $X_i=\{\alpha_i, \alpha_i+1,\dots,d\}$ for some $\alpha_1>\alpha_2>\dots>\alpha_k=1$%
.

Before giving the construction of a lattice path for ordered matroidal double chains, we need some further auxiliary definitions:
A $NE$-path from $(0, 0)$ to $(d-r, r)$ partitions the points on the $(d-r)\times r$ grid into three disjoint sets $U \uplus \Path \uplus L$, where $\Path$ contains the grid points on the path, $U$ contains the grid points strictly to the north/west of the path $\Path$ and $L$ contains the grid points strictly to the south/east of the path $\Path$.
Note that  $L$ and $U$ may be empty.
If $U$ is non-empty, it contains the point $(0, r)$, similarly, if $L$ is non-empty it contains $(d-r, 0)$.
We say that $\Path$ \Def{passes to the north/west relative to a grid point $g$} in the grid if $g \in L \uplus \Path$.
See \cref{fig:UpL} for an example.
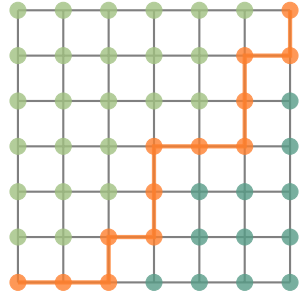
\begin{figure}
 \begin{tikzpicture}[scale=0.6]
  \draw[step=1cm, thick, gray] (0,0) grid (6,6);
  \draw[orange, ultra thick, opacity=0.8] (0,0) -- (2,0) -- (2,1) -- (3,1) -- (3,3) -- (5,3) -- (5,5) -- (6,5) -- (6,6);
  \filldraw[orange, opacity=0.8] (0,0) circle (5pt);
   \filldraw[orange, opacity=0.8] (1,0) circle (5pt);
   \filldraw[orange, opacity=0.8] (2,0) circle (5pt);
 \filldraw[orange, opacity=0.8] (2,1) circle (5pt);
 \filldraw[orange, opacity=0.8] (3,1) circle (5pt);
 \filldraw[orange, opacity=0.8] (3,2) circle (5pt);
 \filldraw[orange, opacity=0.8] (3,3) circle (5pt);
 \filldraw[orange, opacity=0.8] (4,3) circle (5pt);
 \filldraw[orange, opacity=0.8] (5,3) circle (5pt);
 \filldraw[orange, opacity=0.8] (5,4) circle (5pt);
 \filldraw[orange, opacity=0.8] (5,5) circle (5pt);
 \filldraw[orange, opacity=0.8] (6,5) circle (5pt);
 \filldraw[orange, opacity=0.8] (6,6) circle (5pt);
   \filldraw[turkisD, opacity=0.8] (3,0) circle (5pt);
   \filldraw[turkisD, opacity=0.8] (4,0) circle (5pt);
   \filldraw[turkisD, opacity=0.8] (5,0) circle (5pt);
   \filldraw[turkisD, opacity=0.8] (6,0) circle (5pt);
   \filldraw[turkisD, opacity=0.8] (4,1) circle (5pt);
   \filldraw[turkisD, opacity=0.8] (5,1) circle (5pt);
   \filldraw[turkisD, opacity=0.8] (6,1) circle (5pt);
   \filldraw[turkisD, opacity=0.8] (4,2) circle (5pt);
   \filldraw[turkisD, opacity=0.8] (5,2) circle (5pt);
   \filldraw[turkisD, opacity=0.8] (6,2) circle (5pt);
   \filldraw[turkisD, opacity=0.8] (6,3) circle (5pt);
   \filldraw[turkisD, opacity=0.8] (6,4) circle (5pt);
   \filldraw[grunH, opacity=0.8] (0,6) circle (5pt);
   \filldraw[grunH, opacity=0.8] (1,6) circle (5pt);
   \filldraw[grunH, opacity=0.8] (2,6) circle (5pt);
   \filldraw[grunH, opacity=0.8] (3,6) circle (5pt);
   \filldraw[grunH, opacity=0.8] (4,6) circle (5pt);
   \filldraw[grunH, opacity=0.8] (5,6) circle (5pt);
   \filldraw[grunH, opacity=0.8] (0,5) circle (5pt);
   \filldraw[grunH, opacity=0.8] (1,5) circle (5pt);
   \filldraw[grunH, opacity=0.8] (2,5) circle (5pt);
   \filldraw[grunH, opacity=0.8] (3,5) circle (5pt);
   \filldraw[grunH, opacity=0.8] (4,5) circle (5pt);
   \filldraw[grunH, opacity=0.8] (0,4) circle (5pt);
   \filldraw[grunH, opacity=0.8] (1,4) circle (5pt);
   \filldraw[grunH, opacity=0.8] (2,4) circle (5pt);
   \filldraw[grunH, opacity=0.8] (3,4) circle (5pt);
   \filldraw[grunH, opacity=0.8] (4,4) circle (5pt);
   \filldraw[grunH, opacity=0.8] (0,3) circle (5pt);
   \filldraw[grunH, opacity=0.8] (1,3) circle (5pt);
   \filldraw[grunH, opacity=0.8] (2,3) circle (5pt);
   \filldraw[grunH, opacity=0.8] (0,2) circle (5pt);
   \filldraw[grunH, opacity=0.8] (1,2) circle (5pt);
   \filldraw[grunH, opacity=0.8] (2,2) circle (5pt);
   \filldraw[grunH, opacity=0.8] (0,1) circle (5pt);
   \filldraw[grunH, opacity=0.8] (1,1) circle (5pt);
 \end{tikzpicture}
\caption{Disjoint union $U\uplus \Path\uplus L$ of grid points, grid points in $U$ in green, grid points in $\Path$ in orange, and grid points in $L$ in petrol.
}\label{fig:UpL}
\end{figure}

\begin{constr}[Schubert path of  an ordered matroidal double chain]\label{constr:path}
Take $(\uX, \ur)$ an ordered matroidal double chain, and write $X_i = \{\alpha_i, \alpha_i + 1, \ldots, d\}$.
On an $(d-r) \times r$ grid, where the bottom left point is identified with the coordinate point $(0, 0)$, and the top right point is identified with the coordinate point $(d-r, r)$, 
consider the set of grid points
\begin{equation*}
 G(\uX, \ur) \coloneqq \{ g_i \coloneqq (\alpha_i+r_i - r - 1, r - r_i) \text{ for }i = 1, \ldots, k - 1\}\,.
\end{equation*}
It is easy to see that all these points are in the interior of the grid.
See \cref{fig:construction_shubert} for an example.

These trace out a \Def{strictly decreasing} collection of points $ G(\uX, \ur)$.
That is, every two distinct points $g_i=(x_i, y_i), g_j=(x_j, y_j)$ for $1\leq i<j\leq k-1$  satisfy $x_i > x_j $ and $y_i > y_j$.

Indeed, from the matroidal double chain conditions we get
\begin{align}\label{eq:xcoordinate_old}
0 \leq \alpha_{k-1}-r+r_{k-1}-1<\dots<\alpha_1 -r+r_1-1<d-r\\
\quad \Leftrightarrow \quad 
 0 < |X_1| - r_1 < \cdots <|X_{k-1}| - r_{k-1} \leq d - r
\end{align}
and
\begin{align}\label{eq:ycoordinate_old}
 0<r-r_{k-1}<\dots<r-r_2<r-r_1\leq r
 \quad \Leftrightarrow \quad 
 0 \leq r_1 < \cdots < r_{k-1} < r \,.
\end{align}

Note that for $k = 1$ we consider an empty collection of points.
There is a southernmost NE-path $\Path = \Path(\uX, \ur)$ that passes to the north/west relative to $G(\uX, \ur)$.
In fact, since the points in $G(\uX, \ur)$ are strictly decreasing, all the points in $ G(\uX, \ur)$ are in $\Path(\uX, \ur)$. This observation will be useful later.
See the example below.
This defines a Schubert matroid $\SM(\Path(\uX,\ur))$.
We also directly write $\SM(\uX, \ur)$ for the corresponding Schubert matroid.
\end{constr}

\begin{smpl}\label{ex:doublechain}
Consider the  ordered double chain $(\uX, \ur)$ for $d = 21$, $k = 6$ and $r = 12$ given by $\alpha_1 = 21, \alpha_2 = 19, \alpha_3 = 12, \alpha_4 = 9, \alpha_5 = 5, \alpha_6 = 1$, and $r = (0, 1, 5, 7, 9, 12)$.

The collection of points $G$ is $(1, 3)$,  $(3, 5)$,  $(4, 7)$,  $(7, 11)$, and  $(8, 12)$.
For instance, for $i = 2$ we have $g_2 = (\alpha_2+r_2 - r - 1, r - r_2) = (19+1-12-1, 12-1) = (7, 11)$.
The corresponding path is displayed in \cref{fig:construction_shubert}.

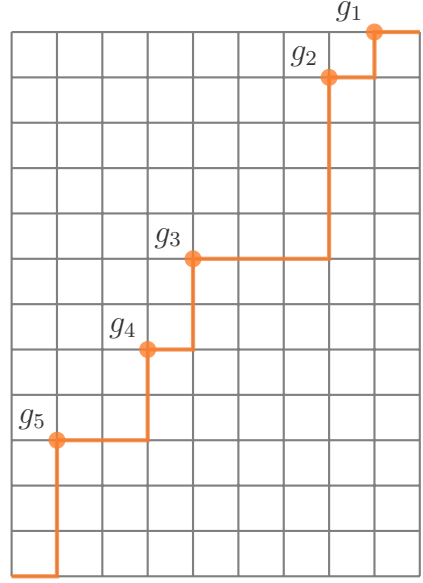
\begin{figure}
\begin{tikzpicture}[scale=0.6]
 \draw[step=1cm, thick, gray] (0,0) grid (9,12);
 \draw[orange, ultra thick, opacity=0.8] (0,0) -- (1,0) -- (1,3) -- (3,3) -- (3,5) -- (4,5) -- (4,7) -- (7,7) -- (7,11) -- (8,11) -- (8,12) -- (9,12);
 \filldraw[orange, opacity=0.8] (1,3) circle (5pt) node[anchor=south east, black]{$g_5$};
 \filldraw[orange, opacity=0.8] (3,5) circle (5pt)node[anchor=south east, black]{$g_4$};
 \filldraw[orange, opacity=0.8] (4,7) circle (5pt)node[anchor=south east, black]{$g_3$};
 \filldraw[orange, opacity=0.8] (7,11) circle (5pt)node[anchor=south east, black]{$g_2$};
 \filldraw[orange, opacity=0.8] (8,12) circle (5pt)node[anchor=south east, black]{$g_1$};
\end{tikzpicture}

\caption{The path corresponding to the given ordered matroidal double chain in \Cref{ex:doublechain}.}
 \label{fig:construction_shubert}
\end{figure}

\end{smpl}

\begin{lm}\label{lm:paths_and_intervals}
On a $(d-r) \times r$ grid, let $\Path$ be a $NE$-path from $(0, 0)$ to $(d-r, r)$, and let $B_{\Path}$ be the corresponding $r$-subset in $[d]$.
Fix a point $g = (x, y)$ in the grid.
Let $\alpha = x + y$ and $I = \{ \alpha + 1,  \alpha + 2, \ldots, d\}$.
Then $\Path$ passes to the north/west relative to $g$ if and only if $|B_{\Path}\cap I | \leq r - y$.
\end{lm}

\begin{proof}
    Consider the points $\ldots, g_{-2}, g_{-1}, g_0  = g, g_1, g_2, \ldots$ that are diagonally to the north/east and south/west from $g$, including $g$.
    That is, $g_i = (x - i, y + i)$.
    Note that every $NE$-path contains exactly one of these points, say $g_i$ for some $i$.
    
    The section of the path $\Path$ contained in the bottom left $(g_i)_x\times (g_i)_y=(x-i) \times (y+i)$ grid, corresponds to a subset $(B_\Path)^{(i)} $ of $\{1,2,\dots,x + y\}$ of size $(g_i)_y = y + i$.
    Therefore, 
    \begin{equation*}|B_\Path \cap I| = |B_\Path| - |B_\Path \cap ([d] \setminus I)| = r - |(B_\Path)^{(i)}| = r - y - i\, . \end{equation*}

    The path $\Path$ passes to the north/west relative to $g$ if and only if $i \geq 0$.
    Therefore, a path passes to the north/west relative to $g$ if and only if $|B_\Path \cap I| \leq r - y$, as desired.
\end{proof}

\begin{prop}\label{prop:balancedorderedSM}
If $(\uX, \ur)$ is  an ordered matroidal  double chain, then $\Pol (\uX, \ur)$ from Equation~\eqref{eq:defpolofdoublechainappendix} is the matroid base polytope of the Schubert matroid $\SM(\uX, \ur)$ from \cref{constr:path}.
\end{prop}

\begin{proof}
For simplicity, we write $\Pol_{\square} \coloneqq \Pol_{\SM(\Path(\uX, \ur))}$.
We start by constructing the path $\Path$ explicitly using \Cref{constr:path}, and we show that the polytopes $\Pol(\uX, \ur)$, constructed in \cref{eq:defpolofdoublechain}, and $\Pol_{\square}$ are the same.

As the \textbf{first} step we will prove that
 if $\vv $ is a vertex of $\Pol(\uX, \ur)$, then it is a $0$-$1$ vector.

We do so by contraposition: assume that $\vv $ is not a $0$-$1$ vector in $\Pol(\uX, \ur)$.
Because $\vv_i \in [0, 1]$ for all $i$, and $\sum_{i=1}^d \vv_i = r$ is an integer, there are at least two indices $i$ such that $\vv_i \in (0, 1)$.
Let $a < b$ be the second largest and largest such indices, respectively.

We define the map $\varphi : \R \to \R^d $ that sends $x \mapsto \vv + x(\be_a - \be_b)$.
We will show that there is some $\varepsilon > 0 $ such that $\varphi(-\varepsilon, \varepsilon) \subseteq \Pol(\uX, \ur)$.
This concludes this part of the proof, as no vertex is contained in an open line contained in the polytope $\Pol(\uX, \ur)$.

Let $\delta_0 = \min\{|\vv_a|, |1 - \vv_a|, |\vv_b|, |1-\vv_b|\}$.
Note that $\varphi(-\delta_0, \delta_0) \subseteq [0, 1]^d$.
Furthermore, $\sum_i \varphi(x)_i = \left(\sum_i \vv_i \right) + x - x = r$.
Let $j = 1, \ldots, k$. There are four cases to consider:

\begin{enumerate}[a)]

\item if $a, b \in X_j$, then $\sum_{i \in X_j} \varphi(x)_i = \left(\sum_{i \in X_j} \vv_i \right) + x - x = \sum_{i \in X_j} \vv_i \leq r_j$.
In this case let $\delta_j = \infty $.

\item if $a, b \not\in X_j$, then $\sum_{i \in X_j} \varphi(x)_i = \sum_{i \in X_j} \vv_i \leq r_j$.
In this case let $\delta_j = \infty $.

\item it is impossible to have $a \in X_j$, but $b\not\in X_j$, as $a < b$ and $X_j$ is an upper interval.

\item if $a \not\in X_j$, but $b \in X_j$,
we argue that $ \sum_{i \in X_j} \vv_i < r_j $.
Indeed, by construction $\vv_i $ is an integer for every $i > a$ except for $i = b$.
Thus $\sum_{i \in X_j} \vv_i = \vv_b + \sum_{\substack{i \in X_j \\ i \neq b}} \vv_i \not\in \Z$.

Therefore, for $x \in (-\delta_j, \delta_j)$ we have that $\sum_{i \in X_j} \varphi(x)_i = - x + \sum_{i \in X_j} \vv_i \leq r_j$.
\end{enumerate}
Let $\varepsilon = \min (\delta_j)_{j=0, 1, \ldots, k }$.
We conclude that for every $x \in (-\varepsilon, \varepsilon)$ we have that $\varphi(x) \in \Pol(\uX, \ur)$.
This concludes that $\varphi(0) = \vv $ is not a vertex of $\Pol(\uX, \ur)$.

Note that this also implies that every $0$-$1$ vector in $\Pol(\uX,\ur)$ is a vertex of $\Pol(\uX,\ur)$.

\textbf{Secondly,} we show that
the vertex sets of $\Pol_{\square}$ and $\Pol(\uX, \ur)$ are the same.

We have seen that vertices of either polytope are $0$-$1$ vectors.
Furthermore, they satisfy $\sum_i \vv_i = r$, so these correspond to 
$B\subset [d]$ of size $r$.

We show that $\onebb_B$ is a vertex of $\Pol_{\square}$ if and only if it is a vertex of $\Pol(\uX, \ur)$.
Indeed, $\onebb_B$ is a vertex of $\Pol_{\square}$ if and only if there is a path $\Qath \succeq \Path(\uX, \ur)$ such that $B = B_{\Qath}$.
The definition of $\Path(\uX, \ur)$ tells us that $\Qath \succeq \Path(\uX, \ur)$ if and only if it is to the north/west of each point $g_j(\uX, \ur) = (\alpha_j + r_j - r - 1, r - r_j)$ for $j = 1, \ldots, k$.

Furthermore, \cref{lm:paths_and_intervals} gives us that this is equivalent to $|B_{\Qath}\cap X_j | \leq r_j$ for $j = 1, \ldots, k$, which can be simplified to $\sum_{i\in X_j} (\onebb_B)_i \leq r_j$.
Because every vertex is a $0$-$1$ vector, this concludes the proof.
\end{proof}
Via a simple relabeling, the following claim can be established:
\begin{lm}\label{lm:isotoSM}
If $(\uY, \ur)$ is a matroidal double chain, there exists an ordered matroidal double chain $(\uX, \ur)$ such that $\matroid(\uX, \ur) \cong \matroid(\uY, \ur)$.
\end{lm}
\begin{proof}
Let $d = |\groundset|$ and set $Y_0 = \emptyset$.
Define a bijection $\linorder : \groundset \to [d]$ by assigning, for each $j = 1, \ldots, k$, the elements of $Y_j \setminus Y_{j-1}$ to the positions $\{d - |Y_j| + 1, \ldots, d - |Y_{j-1}|\}$ in any order.
By construction,
\begin{equation}\label{eq:linorder}
  \linorder(Y_j) = \{d - |Y_j| + 1, \ldots, d\} \quad \text{for each } j = 1, \ldots, k.
\end{equation}
Set $X_j = \linorder(Y_j)$ for each $j$.
Each subset $X_j$ is an interval of $[d]$ containing $d$, and $X_1 \subsetneq \cdots \subsetneq X_k = [d]$ since $|Y_1| < \cdots < |Y_k| = d$.
Since $|X_j| = |Y_j|$ for all $j$, the rank conditions \cref{it:rankinequalitiesappendix,it:DcRanktogetherappendix} are preserved, so $(\uX, \ur)$ is an ordered matroidal double chain.
Moreover, $\linorder$ maps $\Pol(\uY, \ur)$ bijectively to $\Pol(\uX, \ur)$, since $\sum_{e \in Y_j} y_e \leq r_j$ if and only if $\sum_{i \in X_j} y_{\linorder^{-1}(i)} \leq r_j$.
Hence $\matroid(\uX, \ur) \cong \matroid(\uY, \ur)$ via $\linorder$.
\end{proof}
%
\begin{proof}[Proof of \cref{thm:doublechains_are_matroids}]
This follows from relabeling the ground set to obtain an ordered matroidal double chain, then applying \cref{prop:balancedorderedSM}, and finally applying the inverse relabeling.
\end{proof}
We have shown that every matroidal double chain defines a matroid $\matroid(\uX,\ur)$ isomorphic to a Schubert matroid, hence a generalized Catalan matroid (following the nomenclature in \cref{ssec:nestedmatroids}).

However, we have not yet shown that different matroidal double chains define different matroids.
We will show that  $\uX$ is the collection of all non-empty cyclic flats of the matroid $\matroid(\uX,\ur)$ (\cref{lm:cyclicflatsofdoublechains}).
Hence $\matroid(\uX,\ur)$ is a nested matroid.
Together with the information that $\rk(X_i)=r_i$, we can then argue that every nested matroid has a unique representation as matroidal double chain.

For that we will first prove some facts about independent sets and circuits in matroidal double chains.

\begin{lm}\label{lm:indepsetsSM}
Let $(\uX, \ur)$ be  a matroidal double chain with ground set $[d]$.
Let $I\subseteq [d]$.
If $|I \cap X_i| \leq r_i $ for each $i$, then $I$ is an independent set.
\end{lm}

\begin{proof}

Assume that $(\uX, \ur)$ is an ordered matroidal double chain.
From \cref{lm:isotoSM}, we do not lose any generality with this.
We show this by adding the smallest $r - |I|$ terms from  $I^C=[\nd]\setminus I$ to $I$. This gives us a basis $B$ and we want to show that it satisfies $\onebb_B \in \Pol(\uX, \ur)$.
According to \cref{prop:balancedorderedSM}, this is enough to show that $I$ is independent.

It is immediate that $\onebb_B \in[0,1]^d$ and $ \sum_{j=1}^d (\onebb_B)_j =r$.
We need to show that for every $i=1,\dots,k$ we have
\begin{equation*}
    \sum_{j\in X_i}(\onebb_B )_j \leq r_i\,.
\end{equation*}
(Recall Equation~\eqref{eq:defpolofdoublechainappendix}.)
This is equivalent to  $|B\cap X_i| \leq r_i$ for every $i=1,\dots,k$.

There are two cases. 
First assume $B\cap X_i = I\cap X_i$, then $|B\cap X_i| = |I\cap X_i| \leq r_i$ by assumption.

Otherwise $B\cap I^C \cap X_i \neq \emptyset$.
See \cref{fig:expleq1} for an example.
Since $X_i=\{\alpha_i, \ldots, \nd\}$ and we added the smallest elements in $I^C$ to build the basis $B$, it follows that $X_i^C \subseteq B$, therefore
\begin{equation}\label{eq:proofindep1}
    |B \cap X_i| \underbrace{=}_{\substack{X_i^C \subseteq B \\ |B|=r }} r - \underbrace{|B \setminus X_i|}_{\substack{=|B\cap X_i^C|\\=|X_i^C|}}
    \underbrace{=}_{|X_i^C|=\alpha_i-1} r - \alpha_i + 1 = r - (d - |X_i|) = |X_i| - (d - r)\,.
\end{equation}
\begin{figure}
 \begin{tikzpicture}
  \draw (0,0) node[grunH, circle, fill, scale =.5] {};
  \draw (1,0) node[grunH, circle, fill, scale =.5] {};
  \draw (2,0) node[grunH, circle, fill, scale =.5] {};
  \draw (3,0) node[black, circle, fill, scale =.5] {};
  \draw (4,0) node[black, circle, fill, scale =.5] {};
  \draw (5,0) node[turkisD, circle, fill, scale =.5] {};
  \draw (6,0) node[turkisD, circle, fill, scale =.5] {};
  \draw (7,0) node[turkisD, circle, fill, scale =.5] {};
  \draw (8,0) node[black, circle, fill, scale =.5] {};

  \node at (0,-0.5) {$1$};
  \node at (1,-0.5) {$2$};
  \node at (4,-0.5) {$\dots$};
  \node at (7,-0.5) {$d-1$};
  \node at (8,-0.5) {$d$};
  \node at (0, 0.5) {\textcolor{grunH}{$B\setminus I$}};
  \node at (5, 0.5) {\textcolor{turkisD}{$I$}};
  \node at (1.5,0.5) {$X_i$};

  \draw[turkisD] (6,0) ellipse (1.4cm and 0.3cm);
  \draw[grunH] (1,0) ellipse (1.4cm and 0.3cm);
  \draw[thick] (8.2,1) arc(90:270:6.7cm and 1cm);

 \end{tikzpicture}
 \caption{Schematic explanation for Equation~\eqref{eq:proofindep1}}
 \label{fig:expleq1}
\end{figure}
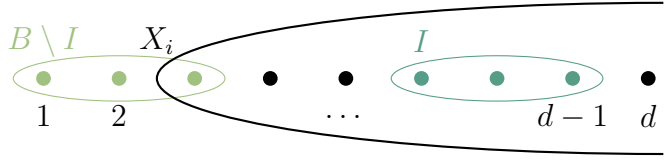

However, the matroidal double chain axiom \cref{it:DcRanktogetherappendix} states that $|X_i| - r_i \leq d - r$.
After substituting and rearranging we get $|B \cap X_i| \leq r_i$ and
 conclude that $B$ is a basis, as desired.
\end{proof}

\begin{cor}\label{cor:indepandcyclesinSM}
Let $(\uX, \ur)$ be a matroidal double chain with ground set $[d]$.
\begin{enumerate}[label=(I\arabic*)]
    \item\label{claimI1}  Fix an index $j\in\{1,\dots, k\}$ and let $I\subseteq X_j $ be such that
    \begin{equation*}
       |I \cap (X_i\setminus X_{i-1}) | \leq r_i - r_{i-1} \text{ for all }1\leq i \leq j \,,
    \end{equation*}
    where we set $r_0=0$.
    Then the set $I$ is independent.
    \item\label{claimI2} Fix indices $j, l\in\{1,\dots,k-1\}$ with $j > l$ and let $I\subseteq X_j$ be such that
    \begin{equation*}
    \begin{split}
        |I \cap (X_i\setminus X_{i-1}) | &\leq r_i - r_{i-1} \text{ for } i < j \text{ and }i\neq l \,,\\
        |I \cap (X_j\setminus X_{j-1}) | &\leq r_j - r_{j-1}+1\,,\\
       \text{  and } \quad  |I \cap (X_l\setminus X_{l-1}) | &\leq r_l - r_{l-1} - 1\,.
    \end{split}
    \end{equation*}
    Then the set $I$ is independent.
    \end{enumerate}
    \begin{enumerate}[label=(C)]
    \item\label{claimC}  Fix an index $j\in\{1,\dots, k\}$ and let $C\subseteq X_j$ be such that
    \begin{equation*}
    \begin{split}
     |C \cap (X_i\setminus X_{i-1}) | &= r_i - r_{i-1} \quad \text{ for all } i<j\,, \\
     \text{ and }\quad   |C \cap (X_j\setminus X_{j-1}) | &= r_j - r_{j-1} + 1\,,
    \end{split}
    \end{equation*}
    where we again set $r_0=0$.
    Then $C$ is a circuit.
\end{enumerate}
\end{cor}

\begin{proof}
Assume that $(\uX, \ur)$ is an ordered matroidal double chain.
From \cref{lm:isotoSM}, we do not lose any generality with this.
Note that
\begin{equation*}
 X_j=\biguplus_{i=1}^j X_i\setminus X_{i-1}\,,
\end{equation*}
where $X_0=\emptyset$.
The decomposition into a disjoint union carries over to
\begin{equation*}
 I\cap X_j=\biguplus_{i=1}^j I\cap( X_i\setminus X_{i-1})\,.
\end{equation*}
Then \ref{claimI1} and \ref{claimI2} are applications of \cref{lm:indepsetsSM}.

To establish \ref{claimC}, recall that the definition of the matroid $\SM(\uX,\ur)$ (\Cref{prop:balancedorderedSM}) via the matroid base polytope in Equation~\eqref{eq:defpolofdoublechainappendix} implies $\rk(X_j)\leq r_j$ (\cref{obs:rkpol}).
From the assumptions we have  that $C\subseteq X_j$ and since the rank function is monotone we can compute
\begin{equation*}
 \rk(C) \leq \rk(X_j) \leq r_j = \sum_{i=1}^j r_i - r_{i-1}  = - 1 + \sum_{i=1}^j |C \cap (X_i \setminus X_{i-1})|= |C| - 1\,,
\end{equation*}
where we use the convention that $r_0 = 0 $.
So $C$ is dependent.
Removing any element from $C$ will give us a set of the type presented in either \ref{claimI1} or \ref{claimI2}, concluding the proof.
\end{proof}
\begin{cor}\label{cor:rankofXj}
Let $(\uX, \ur)$ be a matroidal double chain with ground set $[d]$.
Then the rank of $X_j$ equals $r_j$ for each $j = 1, \ldots, k$.
In particular, there is an independent set $I$ contained in $X_j$ such that for all $i \leq j$ we have that $|I\cap (X_i\setminus X_{i-1})| = r_i - r_{i-1}$.
\end{cor}

\begin{proof}
Assume that $(\uX, \ur)$ is an ordered matroidal double chain.
From \cref{lm:isotoSM}, we do not lose any generality with this.

    We construct an independent set $I \subseteq X_j$ of size $|I| = r_j$.
    Indeed, let
    \begin{equation*}
       I \coloneqq \bigcup_{i=1}^j \min_{r_i - r_{i-1}} X_i \,,
    \end{equation*}
    where $\min_s A $ denotes  the set of the $s$ smallest elements in $A$.
    Because ($\uX, \ur)$ is matroidal, we have $|X_i| - r_i > |X_{i+1}| - r_{i+1}$, this ensures that the $r_i - r_{i-1}$ smallest elements of $X_i$ are not in $X_{i+1}$.
    From \ref{claimI1}, $I$ is independent.
    To conclude the proof note that
    \begin{equation*}
     r_j=\lvert I\rvert=\rk(I)\leq \rk(X_j)\leq r_j.
    \end{equation*}
\end{proof}

\begin{prop}\label{lm:cyclicflatsofdoublechains}
If $(\uX, \ur)$ is a  matroidal double chain, then $\uX$ is the set of all non-empty cyclic flats of $M(\uX, \ur)$.
\end{prop}

\begin{proof}
Assume that $(\uX, \ur)$ is an ordered matroidal double chain.
From \cref{lm:isotoSM}, we do not lose any generality with this.
In this way, from \cref{lm:paths_and_intervals} we have that  $\matroid(\uX, \ur)$ is a Schubert matroid.

\noindent\textbf{Any non-empty connected cyclic flat contains $\nd$.}
Recall that a flat $F$ is connected if the restriction $\matroid|_F$ is connected.
Indeed, assume for sake of contradiction that $d \not\in F$.
Then
\begin{equation}
    \rk(F)< \rk(F\cup \{d\}) \,. \label{eq:flat}
\end{equation}
Let $F = \{f_1 < \cdots < f_{|F|} \}$, let $F_j = \{f_1, \ldots, f_j\}$ and let $j\in\{1,\dots,|F|\}$ be the smallest index such that $\rk(F_j) < j$.
It follows that $\rk(F_j) = j-1 = \rk(F_{j-1})$.
From the submodularity of the rank function we get that
\begin{align}
   \rk( F \cup \{ d \} ) + \rk(F_j) &\leq \rk(F_j\cup \{d\}) + \rk(F)\, \text{ and } \label{eq:submod1} \\
   \rk( F_j \cup \{ d \} ) + \rk(F_{j-1}) &\leq \rk(F_j) + \rk(F_{j-1} \cup \{d\} )\, .  \label{eq:submod2}
\end{align}
Equations~\eqref{eq:submod1} and \eqref{eq:flat} imply that $\rk(F_j) < \rk(F_j\cup \{d\})$.
Equation~\eqref{eq:submod2} together with the monotonicity of the rank function implies  that $\rk( F_j \cup \{ d \})  \leq \rk(F_{j-1} \cup \{d\} )$.
Putting it all together, we have
\begin{equation*}
    j-1 < \rk(F_{j-1} \cup \{d\} ) \leq j\,,
\end{equation*}
 and therefore $F_{j-1} \cup \{d\}$ is an independent set, whereas $F_j$ is not.
This is impossible because $f_j < d$ so $F_j \leqG F_{j-1} \cup \{d\}$.

\noindent\textbf{Any non-connected cyclic flat is a  union of connected cyclic flats.}
Let $F$ be a cyclic flat and take the decomposition $\matroid|_F = \oplus_i \matroid_i $.
Let $F_i \coloneqq \groundset (\matroid_i)$ denote the ground set of $\matroid_i$, we will show that these are cyclic flats.
We denote by $\cl_\matroid(S)$ the smallest flat in $\matroid$ containing $S$.
It is immediate that $F_i$ is a flat, because $\cl_\matroid(F_i) = \cl_F(F_i) = F_i$.

Assume that $F_i$ is not a cyclic flat.
Therefore there is some $f \in F_i$ that is not contained in any circuit in $F_i$.
Now, recall that the circuits of the direct sum of matroids is the union of its circuits (see \cite[Theorem 4.2]{oxley-briefly} for a sketch of the direct sum construction and its effects on circuits).
Thus $f $ is not contained in any circuit in $F$, a contradiction.

\noindent\textbf{Any non-empty cyclic flat $F$ is of the form $X_j$ for some $j$.}
From the above, every cyclic flat $F$ is the union of cyclic flats.
Furthermore, it is established in \cite[Theorem 3.11]{bonin2006lattice} that every non-trivial connected flat of a lattice path matroid is an interval.
Above we saw that these connected flats contain $d$, therefore $F$ is the union of intervals that contain $d$, so it is itself an interval that contains $d$.
Call $\alpha \coloneqq \min F $, i.e., $F$ is of the form $F=\{\alpha, \alpha+1,\dots, d\}$.

Then there exists a $j\in\{1,\dots,k\}$  such that $F\subseteq X_j$ and $F\not\subseteq X_{j-1}$, that is, $\alpha_{j-1}> \alpha \geq \alpha_j$.
Now assume for sake of contradiction that $\alpha > \alpha_j$.
We consider two cases:

First, assume $|F \setminus X_{j-1} | \geq r_j - r_{j-1}$.
Let $I$ be some independent set of rank $r_{j-1}$ contained in $X_{j-1}\subset F$
(which exists by \cref{cor:rankofXj}), and let
\begin{equation*}
 J \coloneqq \{\alpha, \alpha + 1 , \ldots, \alpha + r_j - r_{j-1} - 1\} \uplus I\,.
\end{equation*}
Then we use \ref{claimI1} to show that $J$ is an independent set.
Indeed,
\begin{equation*}
 |J\cap (X_j\setminus X_{j-1})| = | \{\alpha, \alpha + 1 , \ldots, \alpha + r_j - r_{j-1} - 1\} | = r_j - r_{j-1}
\end{equation*}
and for every $i < j$ we have
\begin{equation*}
 |J\cap (X_i\setminus X_{i-1})| = |I\cap (X_i\setminus X_{i-1})| = r_i - r_{i-1}
\end{equation*}
by \cref{cor:rankofXj}.
It follows that
\begin{equation*}
     r_j\leq \rk ( F ) \leq \rk(X_j)=r_j\,,
\end{equation*}
 so $F$ can only be a flat if $F = X_j$.

Second, if  $|F \setminus X_{j-1} |  < r_j - r_{j-1}$,  then we claim,  there is no circuit $C$ such that $\alpha \in C \subseteq F$.
Indeed, assume for sake of contradiction that such $C$ exists, so from \cref{lm:indepsetsSM} there is some $l\in \{1, \ldots , k\}$ such that $|C\cap X_l| > r_l $.
By hypothesis, $C\setminus \{\alpha\}$ is independent, so $|(C \setminus \{\alpha\})\cap X_{i}| \leq r_i $ for all $i$ by \cref{cor:rankofXj}.

For $i < j$ we have $|C \cap X_{i}| = |(C \setminus \{\alpha\})\cap X_{i}|$ since $\alpha < \alpha_{j-1}$.
So $|C \cap X_{i}| \leq r_i$.

For $i > j$ we have
\begin{equation*}
    |C \cap X_{i}| \underbrace{=}_{C \subseteq F \subseteq X_j \subseteq X_i} |C \cap X_j| = |(C \setminus \{\alpha\})\cap X_{j}| + 1
    \underbrace{\leq}_{\substack{C \setminus \{\alpha\} \text{independent}\\\rk(X_j)\leq r_j}}
    r_j + 1 \underbrace{\leq}_{\text{matroidal double chain}} r_i\, .
\end{equation*}
Therefore the index $l$ with $|C \cap X_{l}|>r_l $ must be $l=j$.
Then we have $|C\cap X_j| \geq r_j + 1$.
On the other hand we have
\begin{equation*}
   |C\cap X_j| = |C\cap (X_j \setminus X_{j-1})| + \underbrace{|C \cap X_{j-1}| }_{\leq r_{j-1}}
   \underbrace{\leq}_{C\subset F \subset X_j}
   \underbrace{|F \setminus X_{j-1} |}_{< r_j-r_{j-1}} + r_{j-1} < r_j \,,
\end{equation*}
 a contradiction.

\noindent\textbf{The set $X_j$ is a cyclic flat for $i=1,\dots,k$.}
First, $\rk(X_j) = r_j$, as it contains the independent set
\begin{equation*}
    I_j \coloneqq \biguplus_{i = 1}^j \min_{r_i - r_{i-1}} X_i \, ,
\end{equation*}
(see \cref{cor:rankofXj}).
Furthermore $\rk(X_j \cup \{e\} ) = r_j + 1$ if $e\not\in X_j$, as $I_j \cup \{e\}$ is an independent set from \ref{claimI1}.
Therefore $X_j $ is a flat.

We show now that $X_j$ is a cyclic flat.
Take $e \in X_j$, let $l$ be minimal such that $e \in X_l$ (note that $e \in X_l\setminus X_{l-1}$) and let  $S_l \subseteq X_l\setminus X_{l-1}$ be some set of size $r_l - r_{l-1} + 1$ containing $e$.
Define
\begin{equation*}
    C \coloneqq S_l\ \uplus \left(\biguplus_{i = 1}^{l-1}  \min_{r_i-r_{i-1}} X_i\right)\,,
\end{equation*}
Note that $C\subseteq X_l$.
It follows from \Cref{cor:indepandcyclesinSM}~\ref{claimC} that $C$ is a circuit.
This concludes the proof that $X_j$ is a cyclic flat.

\end{proof}

\begin{proof}[Proof of \cref{thm:doublechains_are_nestedmatroids}]
  Combining \cref{lm:cyclicflatsofdoublechains} and \cref{cor:rankofXj} we have shown that a matroidal double chain $(\uX,\ur)$ is the chain of cyclic flats of a (nested) matroid together with the rank function restricted the cyclic flats.
 It follows from \cite[Theorem 3.2.]{bonin2008lattice}, that this, together with the conditions of the matroidal double chains, gives a cryptomorphic characterization of  (nested) matroids.
\end{proof}

We close this section with counting the number of loopless nested matroids on $[\nd]$.
\begin{prop}\label{prop:hampe_dfact}
 There are $d!$ many loopless nested matroids on the ground set $[d]$.
\end{prop}
This result recovers a coarser version of \cite[Theorem 4.5]{hampe_intersection_2017}, where the author shows that the number of loopless nested matroids on $[\nd]$ for a fixed rank $r$ is the Eulerian number $\Eulerian_{r-1,\nd}$.
\begin{proof}[Proof of \cref{prop:hampe_dfact}]
Recall the set of matroidal double chains for fixed $\nd, k, r$:
\begin{equation*}
\begin{BNiceArray}{c|c}
 & \emptyset\subsetneq X_1 \subsetneq\dots\subsetneq X_k=[\nd]\\
(\uX, \ur) &  0 \leq r_1 < \cdots < r_k = r\\
 &  0 < |X_1| - r_1 < \cdots <|X_{k-1}| - r_{k-1} \leq |X_k| - r_k = d - r
\end{BNiceArray}\, .
\end{equation*}
We want to give an equivalent description in terms of differences. Define
\begin{equation*}
 Y_i\coloneqq X_i\setminus X_{i-1}
 \quad\text{and}\quad
 h_i\coloneqq r_i-r_{i-1} \quad\text{ for }i=1,\dots,k\,,
\end{equation*}
where we set $ X_0\coloneqq\emptyset$ and $r_0=0$. Note that $\lvert Y_i\rvert=\lvert X_i\rvert-\lvert X_{i-1}\rvert$,
\begin{equation*}
 X_i=\bigcup_{j=1}^i Y_j,
  \quad\text{and}\quad
 r_i = \sum_{j=1}^i h_j\,.
\end{equation*}
With that we can describe the set of matroidal double chains for fixed $d, k, r$ equivalently as
\begin{equation*}
\begin{BNiceArray}{c|cc}
(Y_1,\dots,Y_k)\models [d] & 0\leq h_1<\lvert Y_1\rvert,&0 < h_k\leq\lvert Y_k\rvert\\
  (h_1,\dots,h_k)\models r & 0 < h_i<\lvert Y_i\rvert\,, &\text{for }i=2,\dots,k-1
\end{BNiceArray}
\end{equation*}
If we want to restrict to loopless nested matroids, recall from \cref{prop:looplessnestedmatroid} that $r_1>0$.
Then the set of matroidal double chains for fixed $\nd, k, r$ can be described as:
\begin{equation*}
\begin{BNiceArray}{c|cc}
(Y_1,\dots,Y_k)\models [d] & 0\leq h_1<\lvert Y_1\rvert,& \\
  (h_1,\dots,h_k)\models r & 0 < h_i<\lvert Y_i\rvert\,, &\text{for }i=2,\dots,k
\end{BNiceArray}
\end{equation*}
For a fixed $d$ we want to count the cardinality of 
\begin{equation*}
 \bigcup_{k,r} {\small\begin{BNiceArray}{c|cc}
(Y_1,\dots,Y_k)\models [d] & 0\leq h_1<\lvert Y_1\rvert,& \\
  (h_1,\dots,h_k)\models r & 0 < h_i<\lvert Y_i\rvert\,, &\text{for }i=2,\dots,k
\end{BNiceArray} }\, .
\end{equation*}
We do so by  considering exponential generating functions.
Using the methods of exponential power series presented in \cite[Sections II.1 and II.2]{flajolet2009analytic}, we investigate
\begin{equation*}
 1+\sum_{d\geq1} \#\left(\bigcup_{k,r} {\small\begin{BNiceArray}{c|cc}
(Y_1,\dots,Y_k)\models [d] & 0\leq h_1<\lvert Y_1\rvert,& \\
  (h_1,\dots,h_k)\models r & 0 < h_i<\lvert Y_i\rvert\,, &\text{for }i=2,\dots,k
\end{BNiceArray} }\right) \frac{x^d}{d!}\,.
\end{equation*}
Observe that a loopless nested matroid for a fixed $k$ can be described as a tuple
\begin{equation*}
 ((Y_1,h_1),\dots,(Y_k,h_k))\quad\text{with}\begin{cases}
                                             0\leq h_1<\lvert Y_1\rvert\,, \\
                                             1\leq h_i<\lvert Y_i\rvert\,,\text{ for }i=2,\dots,k
                                            \end{cases}
\end{equation*}
It then follows that $d=\lvert Y_1\rvert + \dots +\lvert Y_k\rvert $ and $r=h_1+\dots+h_k$.
So, we can now compute
\begin{align*}
 &1+\sum_{d\geq1} \#\left(\bigcup_{k,r} {\small\begin{BNiceArray}{c|cc}
(Y_1,\dots,Y_k)\models [d] & 0\leq h_1<\lvert Y_1\rvert,& \\
  (h_1,\dots,h_k)\models r & 0 < h_i<\lvert Y_i\rvert\,, &\text{for }i=2,\dots,k
\end{BNiceArray} }\right) \frac{x^d}{d!}\\
=\ &1+\sum_{k\geq1}
    \left(\sum_{j\geq 1} \frac{j}{j!} x^j\right)
    \left(\sum_{j\geq 2} \frac{j-1}{j!} x^j\right)^{k-1}\\
 =\ & 1+\left(\sum_{j\geq 0} \frac{1}{j!} x^{j+1}\right) \sum_{k\geq0}
    \left(\sum_{j\geq 2} \frac{j}{j!} x^j - \sum_{j\geq 2} \frac{1}{j!} x^j\right)^{k}\\
 =\ &  1+(x e^x)\ \sum_{k\geq0}
    \left(\sum_{j\geq 1} \frac{1}{j!} x^{j+1} 
        - \sum_{j\geq 2} \frac{1}{j!} x^j\right)^{k} \\
=\ &  1+(x e^x)\ \sum_{k\geq0}
    \left((x (e^x-1))-(e^x-x-1) \right)^{k} \\
  =\ & 1+(x e^x)\ \sum_{k\geq0}
    \left( e^x(x-1) +1 \right)^{k}\\
 =\ &  1+ \frac{x e^x}{ 1-  (e^x(x-1) +1)}
\  =\  \frac{  -e^x(x-1) +x e^x}{ -e^x(x-1) } \\
  =\ &  \frac{  e^x}{ -e^x(x-1) }
 \ =\ \frac{1}{ 1-x }
\end{align*}
Hence, the number of loopless nested matroids on $[d]$ is $d!$.
\end{proof}

\newpage

\section{Data}\label{sec:data}
\begin{table}[h]
\small{
\begin{tabular}{l c}
Set composition & SM(24)\\
$\mathbf{1234}$,  $123\lvert 4$,  $\mathbf{124\lvert 3}$,  $\mathbf{134\lvert 2}$, $ \mathbf{234 \lvert  1} $,  &  \redcross \\
$ 12 \lvert  34   $ & \redcross\\
$\mathbf{13 \lvert  24 }  $ & $ 24                       $ \\
$\mathbf{14 \lvert  23}   $ & $ 23                       $ \\
$\mathbf{23 \lvert  14}   $ & $ 14                       $ \\
$\mathbf{24 \lvert  13}   $ & $ 13                       $ \\
$\mathbf{34 \lvert  12 }  $ & $ 12                       $ \\
$1 \lvert  234 $, $ \mathbf{2 \lvert  134} $, $ \mathbf{3 \lvert  124 }$, $ \mathbf{4 \lvert  123} $  & \redcross \\
$ 12 \lvert  3 \lvert  4 $, $ 12 \lvert  4 \lvert  3 $ & \redcross\\
$13 \lvert  2 \lvert  4 $, $ 13 \lvert  4 \lvert  2  $ & $ 24                       $ \\
$14 \lvert  2 \lvert  3 $, $ \mathbf{14 \lvert  3 \lvert  2}  $ & $ 23                       $ \\
$23 \lvert  1 \lvert  4 $, $ 23 \lvert  4 \lvert  1  $ & $ 14                       $ \\
$24 \lvert  1 \lvert  3 $, $ \mathbf{24 \lvert  3 \lvert  1}  $ & $ 13                       $ \\
$34 \lvert  1 \lvert  2 $, $ \mathbf{34 \lvert  2 \lvert  1}  $ & $ 12                       $ \\
$ 3 \lvert  12 \lvert  4 $, $ 4 \lvert  12 \lvert  3  $& \redcross \\
$2 \lvert  13 \lvert  4  $ & $ 14                       $ \\
$\mathbf{4 \lvert  13 \lvert  2 } $ &  \redcross \\
$\mathbf{2 \lvert  14 \lvert  3}  $ & $ 13                       $ \\
$\mathbf{3 \lvert  14 \lvert  2 }$& \redcross \\
$ \mathbf{4 \lvert  23 \lvert  1 } $& \redcross \\
$1 \lvert  23 \lvert  4 $ & $24$ \\ 
$1 \lvert  24 \lvert  3  $ & $ 23                       $ \\
$\mathbf{3 \lvert  24 \lvert  1} $ & \redcross \\
$ 1 \lvert  34 \lvert  2 $, $ 2 \lvert  34 \lvert  1 $ & \redcross \\
$3 \lvert  4 \lvert  12 $, $ \mathbf{4 \lvert  3 \lvert  12}  $ & $ 12                       $ \\
$2 \lvert  4 \lvert  13 $, $ \mathbf{4 \lvert  2 \lvert  13}  $ & $ 13                       $ \\
$2 \lvert  3 \lvert  14 $, $ \mathbf{3 \lvert  2 \lvert  14}  $ & $ 14                       $ \\
$1 \lvert  4 \lvert  23 $, $ 4 \lvert  1 \lvert  23  $ & $ 23                       $ \\
$1 \lvert  3 \lvert  24 $, $ 3 \lvert  1 \lvert  24  $ & $ 24                       $ \\
$1 \lvert  2 \lvert  34 $, $ 2 \lvert  1 \lvert  34  $ & \redcross \\
$1 \lvert  2 \lvert  3 \lvert  4 $ & $ 24$ \\
$\mathbf{4 \lvert 3 \lvert 2 \lvert 1} $ & 12\\
...&
\end{tabular}}
\caption{(Almost) All set compositions of $[4]$, grouped by the bases of $\SM(\{2,4\})$ that has maximal $\opi$-score.
A set composition is marked \redcross\ if it is not $\SM(\{2,4\})$-generic;
otherwise, the displayed pair indicates the selected basis of $\SM(\{2,4\})$.
All max-min set compositions are highlighted in bold.
This corresponds to a single row of the matrix $\wqscmatrixconj$ from \cref{smpl:d4nested}.
}
\label{tab:setcomp_SM24}
\end{table}

{
\renewcommand{\arraystretch}{1.3}

\begin{table}[htbp]
 \begin{tabular}{*{6}{c} c c }
 $d$ & $r$ & $k$ & $\uX$ & $\ur$ & path & bases  & SM-defining set\\
 \hline
 2 & 0 & 1 & $\emptyset\subsetneq[2]$ & $(0)$ &\parbox[c]{1.1cm}{ \begin{tikzpicture}[scale=0.5]
 \draw[gray] (0,0) grid (2,0);
 \draw[orange, very thick, opacity=0.5] (0,0) -- (1,0) -- (2,0);
\end{tikzpicture} }& $\{\emptyset\}$ &  $\emptyset$ \\
   & 1 & 1 & $\emptyset\subsetneq[2]$ & $(1)$ &\parbox[c]{0.6cm}{\begin{tikzpicture}[scale=0.5]
 \draw[gray] (0,0) grid (1,1);
 \draw[orange, very thick, opacity=0.5] (0,0) -- (1,0) -- (1,1);
\end{tikzpicture}}& $\{\{1\}\}$ & $\{1 \}$\\
   & 1 & 2 & $\emptyset\subsetneq \{1\}\subsetneq[2]$ & $(0,1)$ &  & \\
   & 1 & 2 & $\emptyset\subsetneq \{2\}\subsetneq[2]$ & $(0,1)$ &\parbox[c]{0.6cm}{ \begin{tikzpicture}[scale=0.5]
 \draw[gray] (0,0) grid (1,1);
 \draw[orange, very thick, opacity=0.5] (0,0) -- (0,1) -- (1,1);
\end{tikzpicture}} & $\{\{1\}, \{2\}\}$ & $\{2 \}$ \\
   & 2 & 1 & $\emptyset\subsetneq[2]$ & $(2)$ &\parbox[c]{1em}{\begin{tikzpicture}[scale=0.5]
 \draw[gray] (0,0) grid (0,2);
 \draw[orange, very thick, opacity=0.5] (0,0) -- (0,2);
\end{tikzpicture}}& $\{\{1,2\}\}$ & $\{1,2 \}$ \\
\hline   
 3 & 0 & 1 & $\emptyset\subsetneq[3]$ & $(0)$ & \parbox[c]{1.6cm}{ \begin{tikzpicture}[scale=0.5]
 \draw[gray] (0,0) grid (3,0);
 \draw[orange, very thick, opacity=0.5] (0,0) -- (1,0) -- (3,0);
\end{tikzpicture} } & $\{\emptyset\}$ & $\emptyset$\\
   & 1 & 1 & $\emptyset\subsetneq[3]$ & $(1)$ & \parbox[c]{1.1cm}{\begin{tikzpicture}[scale=0.5]
 \draw[gray] (0,0) grid (2,1);
 \draw[orange, very thick, opacity=0.5] (0,0) -- (2,0) -- (2,1);
\end{tikzpicture}}& $\{\{1\}, \{2\},\{3\}\}$ & $\{3\}$\\
   & 1 & 2 & $\emptyset\subsetneq \{1\} \subsetneq[3]$ & $(0,1)$ & & & \\
   & 1 & 2 & $\emptyset\subsetneq \{2\} \subsetneq[3]$ & $(0,1)$ & & & \\
   & 1 & 2 & $\emptyset\subsetneq \{3\} \subsetneq[3]$ & $(0,1)$ &  \parbox[c]{1.1cm}{\begin{tikzpicture}[scale=0.5]
 \draw[gray] (0,0) grid (2,1);
 \draw[orange, very thick, opacity=0.5] (0,0) -- (1,0) -- (1,1) -- (2,1);
\end{tikzpicture}} &$\{\{1\}, \{2\}\}$ & $\{2 \}$ \\
   & 1 & 2 & $\emptyset\subsetneq \{1,2\} \subsetneq[3]$ & $(0,1)$ & & & \\
   & 1 & 2 & $\emptyset\subsetneq \{1,3\} \subsetneq[3]$ & $(0,1)$ & & & \\
   & 1 & 2 & $\emptyset\subsetneq \{2,3\} \subsetneq[3]$ & $(0,1)$ &  \parbox[c]{1.1cm}{\begin{tikzpicture}[scale=0.5]
 \draw[gray] (0,0) grid (2,1);
 \draw[orange, very thick, opacity=0.5] (0,0) -- (0,1) -- (2,1);
\end{tikzpicture}} &$\{\{1\}\}$ & $\{1 \}$\\
   & 2 & 1 & $\emptyset\subsetneq[3]$ & $(2)$ &  \parbox[c]{0.6cm}{\begin{tikzpicture}[scale=0.5]
 \draw[gray] (0,0) grid (1,2);
 \draw[orange, very thick, opacity=0.5] (0,0) -- (1,0) -- (1,2);
\end{tikzpicture}}& $\{\{1, 2\},\{1,3\}, \{2,3\}\}$ & $\{2,3 \}$\\
   & 2 & 2 & $\emptyset\subsetneq \{1\} \subsetneq[3]$ & $(0,2)$ & & & \\
   & 2 & 2 & $\emptyset\subsetneq \{2\} \subsetneq[3]$ & $(0,2)$ & & & \\
   & 2 & 2 & $\emptyset\subsetneq \{3\} \subsetneq[3]$ & $(0,2)$ &  \parbox[c]{0.6cm}{\begin{tikzpicture}[scale=0.5]
 \draw[gray] (0,0) grid (1,2);
 \draw[orange, very thick, opacity=0.5] (0,0) -- (0,2) -- (1,2);
\end{tikzpicture}} & $\{\{1, 2\}\}$ & $\{1,2 \}$ \\
   & 2 & 2 & $\emptyset\subsetneq \{1,2\} \subsetneq[3]$ & $(1,2)$ & & & \\
   & 2 & 2 & $\emptyset\subsetneq \{1,3\} \subsetneq[3]$ & $(1,2)$ & & & \\
   & 2 & 2 & $\emptyset\subsetneq \{2,3\} \subsetneq[3]$ & $(1,2)$ &   \parbox[c]{0.6cm}{\begin{tikzpicture}[scale=0.5]
 \draw[gray] (0,0) grid (1,2);
 \draw[orange, very thick, opacity=0.5] (0,0) -- (0,1) -- (1,1) -- (1,2);
\end{tikzpicture}} & $\{\{1,3\}, \{2,3\}\}$ & $\{1,3 \}$\\
   & 3 & 1 & $\emptyset \subsetneq[3]$ & $(3)$ &  \parbox[c]{0.1cm}{\begin{tikzpicture}[scale=0.5]
 \draw[gray] (0,0) grid (0,3);
 \draw[orange, very thick, opacity=0.5] (0,0) -- (0,3);
\end{tikzpicture}} & $\{\{1,2,3\} \}$ & $\{1,2,3 \}$\\
\end{tabular}
\caption{All loopless nested matroids on $[d]$ for $d = 2$ and $d = 3$, together with their matroidal double chain $(\uX, \ur)$, the corresponding NE-path in the grid, the set of bases, and the SM-defining set.
Rows with empty path and bases entries correspond to nested matroids that are isomorphic to a previously listed Schubert matroid under relabeling, and are included for completeness.}
\label{tab:loopless_SM_d2d3}
\end{table}

}

\clearpage


\printbibliography

\end{document}